\documentclass[11pt, reqno]{amsart}
\usepackage{amsthm}
\usepackage{amsmath,amssymb,amsthm,mathrsfs,amscd}
\usepackage{MnSymbol}
\usepackage{color}
\usepackage[all]{xy}
\usepackage[english]{babel}
\usepackage{graphicx}
\usepackage[latin1]{inputenc}
\usepackage[T1]{fontenc}
\usepackage[final=true]{hyperref}
\usepackage{tikz}
\usetikzlibrary{patterns}
\usetikzlibrary{arrows}
\usepackage{rotating}
\usepackage{cases}

\theoremstyle{plain}
\newtheorem{thm}{Theorem}[section]
\newtheorem{lem}[thm]{Lemma}
\newtheorem{cor}[thm]{Corollary}
\newtheorem{prop}[thm]{Proposition}

\theoremstyle{definition}
\newtheorem{examplex}[thm]{Example}
\newenvironment{ex}
  {\pushQED{\qed}\examplex}
  {\popQED\endexamplex}
\newtheorem{rem}[thm]{Remark}
\theoremstyle{definition}
\newtheorem{defi}[thm]{Definition}

\newcommand{\dual}[1]{\widehat{#1}}

\DeclareMathOperator{\sln}{sl_{n}}

\DeclareMathOperator{\N}{\mathbb{N}}
\DeclareMathOperator{\GG}{\mathbb{G}_m}
\DeclareMathOperator{\Z}{\mathbb{Z}}

\DeclareMathOperator{\C}{\mathbb{C}}

\DeclareMathOperator{\sgn}{sgn}

\DeclareMathOperator{\Hom}{Hom}

\DeclareMathOperator{\unip}{\mathcal{N}}

\newcommand{\cl}[1] {\overline{#1}}
\newcommand{\op}[1] {\mathring{#1}}
\newcommand{\li}{\mathcal{L}}
\newcommand{\h}{\mathcal{H}}
\newcommand{\mm}[2]{x_{ #1, #2}}
\newcommand{\nn}[2]{y_{ #1, #2}}

\newcommand{\Gam}[1]{\xi^{#1}}

\DeclareMathOperator{\starop}{\star op}
\DeclareMathOperator{\trl}{\Phi}
\DeclareMathOperator{\trs}{\Psi}
\DeclareMathOperator{\pol}{pol}
\newcommand{\stp}{\mathbf{s}}

\DeclareMathOperator{\CA}{CA}
\DeclareMathOperator{\BZ}{\mathcal{P}}

\DeclareMathOperator{\ii}{\mathbf{i}}
\DeclareMathOperator{\jj}{\mathbf{j}}
\DeclareMathOperator{\trop}{trop}
\newcommand{\stc}{\mathbb{S}}
\DeclareMathOperator{\rpot}{\varrho}

\DeclareMathOperator{\RR}{\phi}
\DeclareMathOperator{\zzz}{z}
\DeclareMathOperator{\bbb}{b}

\newcommand{\TTix}{\dual{\mathbb{T}}_{\ii}}

\newcommand{\TTixp}{\dual{\mathbb{T}}_{\jj}}
\DeclareMathOperator{\TTia}{\mathbb{T}_{ \ii}}
\DeclareMathOperator{\TTiap}{\mathbb{T}_{\jj}}

\DeclareMathOperator{\TTse}{\mathbb{T}_{\Sigma}}
\DeclareMathOperator{\TTsea}{\mathbb{T}_{ \Sigma}}
\newcommand{\TTsex}{\dual{\mathbb{T}}_{\Sigma}}
\newcommand{\TTsexp}{\dual{\mathbb{T}}_{\Sigma'}}
\DeclareMathOperator{\SE}{\Sigma}
\DeclareMathOperator{\SEp}{\Sigma'}
\DeclareMathOperator{\TTsep}{\mathbb{T}_{\Sigma'}}

\newcommand{\ttsx}{\dual{\mu}_{\SEp}^{\SE}}
\DeclareMathOperator{\ttsa}{{\mu}_{\SEp}^{\SE}}
\newcommand{\ttix}{\dual{\mu}_{\jj}^{\ii}}
\DeclareMathOperator{\ttia}{{\mu}_{\jj}^{\ii}}
\DeclareMathOperator{\Csebk}{\mathcal{C}_{\SE}^{\text{BK}}}

\DeclareMathOperator{\SEi}{\Sigma_{\ii}}

\newcommand{\trod}[1]{\left[#1 \right]_{\trop}}
\newcommand{\tro}[1]{[#1]_{\trop}}
\DeclareMathOperator{\rai}{\rpot_{\text{a},\ii}}
\DeclareMathOperator{\raj}{\rpot_{\text{a}, \jj}}
\DeclareMathOperator{\ra}{\rpot_{\text{a}}}
\DeclareMathOperator{\rr}{\varrho}
\newcommand{\ssai}{\dual{\rr}_{\text{a},\ii}}
\newcommand{\ssaj}{\dual{\rr}_{\text{a}, \jj}}
\newcommand{\ssa}{\dual{\rr}_{\text{a}}}

\DeclareSymbolFont{matha}{OML}{txmi}{m}{it}
\DeclareMathSymbol{\varv}{\mathord}{matha}{118}

\DeclareFontFamily{U}{mathx}{\hyphenchar\font45}
\DeclareFontShape{U}{mathx}{m}{n}{
      <5> <6> <7> <8> <9> <10>
      <10.95> <12> <14.4> <17.28> <20.74> <24.88>
      mathx10
      }{}
\DeclareSymbolFont{mathx}{U}{mathx}{m}{n}
\DeclareFontSubstitution{U}{mathx}{m}{n}
\DeclareMathAccent{\widecheck}{0}{mathx}{"71}

\DeclareMathOperator{\GGi}{\mathbb{G}_m^{\mathcal{T}_{\ii}}}
\newcommand{\TTis}{\dual{\mathbb{L}}_{\ii}}
\DeclareMathOperator{\TTil}{{\mathbb{L}}_{\ii}}
\DeclareMathOperator{\TTilp}{\mathbb{L}_{\jj}}
\newcommand{\TTisp}{\dual{\mathbb{L}}_{\jj}}

\DeclareMathOperator{\pp}{p}
\DeclareMathOperator{\ppi}{p_{\ii}}

\newcommand{\rvec}[1]{\mathfrak{r} {#1}}
\newcommand{\rvecs}[1]{(\mathfrak{r} {#1})}
\newcommand{\rvecd}[1]{\mathfrak{r} {#1}}
\newcommand{\rveck}[1]{\mathfrak{r} ({#1})}
\newcommand{\rvecds}[1]{\left(\mathfrak{r} {#1}\right)}
\newcommand{\drvec}[1]{\mathfrak{s} {#1}}
\newcommand{\drvecx}[2]{(\mathfrak{s} #1) (#2)}
\newcommand{\drvecxd}[2]{\left(\mathfrak{s} #1\right) \left(#2 \right)}
\newcommand{\drveck}[1]{\mathfrak{s} (#1)}

\DeclareMathOperator{\iot}{NA}

\DeclareMathOperator{\CAi}{CA_{\ii}}
\newcommand{\CAid}{\dual{\CA}_{\ii}}

\newcommand{\CAjd}{\dual{\CA}_{\jj}}
\DeclareMathOperator{\ioti}{\iot_{\ii}}
\DeclareMathOperator{\iotj}{\iot_{\jj}}

\newcommand\restr[2]{{
  \left.\kern-\nulldelimiterspace 
  #1 
  \vphantom{\big|} 
  \right|_{#2} 
  }}
  \makeatletter
\renewcommand*\env@cases[1][1.2]{%
  \let\@ifnextchar\new@ifnextchar
  \left\lbrace
  \def\arraystretch{#1}%
  \array{@{}l@{\quad}l@{}}%
}
\makeatother

\begin{document}

\title{Combinatorics of canonical bases revisited: Type A}

\author{Volker Genz}
\address{Mathematical Institute, Ruhr-Universit{\"a}t Bochum}
\email{volker.genz@gmail.com}

\author{Gleb Koshevoy}
\address{CEMI Russian Academy of Sciences, MCCME and Poncelet laboratory (CNRS and IMU)}
\email{koshevoy@cemi.rssi.ru}

\author{Bea Schumann}
\address{Mathematical Institute, University of Cologne}
\email{schumann.bea@gmail.com}

\begin{abstract}
We initiate a new approach to the study of the combinatorics of several parametrizations of canonical bases. In this work we deal with Lie algebras of type $A$. Using geometric objects called rhombic tilings we derive a "Crossing Formula" for the action of the crystal operators on Lusztig data for an arbitrary reduced word of the longest Weyl group element. We provide the following three applications of this result. Using the tropical Chamber Ansatz of Berenstein-Fomin-Zelevinsky we prove an enhanced version of the Anderson-Mirkovi\'c conjecture for the crystal structure on MV polytopes. We establish a duality between Kashiwara's string and Lusztig's parametrization, revealing that each of them is controlled by the crystal structure of the other. We identify the potential functions of the unipotent radical of a Borel subgroup of $SL_n$ defined by Berenstein-Kazhdan and Gross-Hacking-Keel-Kontsevich, respectively, with a function arising from the crystal structure on Lusztig data.
\end{abstract}

\maketitle

\tableofcontents

\section*{Introduction}
Let $\mathfrak{g}$ be a complex simple Lie algebra and $\mathfrak{n}^-$ its negative part in the Cartan decomposition. Bases of PBW type for the universal enveloping algebra of $\mathfrak{n}^-$ play an important role in the study of representations of $\mathfrak{g}$. Passing to the negative part of its quantized enveloping algebra $U_q^-$, Lusztig constructed PBW-type bases by introducing so-called braid group operators. These bases consist of ordered monomials of root vectors $F_{\beta}$, one for each positive root  $\beta$, depending on a choice of convex ordering of the positive roots of $\mathfrak{g}$. 

Lusztig discovered in \cite{Lu} a basis $\mathbf B$ of $U_q^-$, called the canonical basis, with various favorable properties. The canonical basis is, by construction, in natural bijection with any PBW-type basis which leads to remarkable piecewise linear combinatorics. 

Fixing a convex ordering $\le$ of the positive roots we can identify a PBW-type basis element $b$ with the exponent vector of the corresponding monomial in the $\le$-ordered root vectors. This element of $\mathbb{N}^N$, where $N$ is the number of positive roots of $\mathfrak{g}$, is called the $\le$-Lusztig datum of $b$ (or equivalent the $\le$-Lusztig datum of the element $\mathbf b\in \mathbf B$ to which $b$ maps under the natural bijection of the PBW-type basis with $\mathbf B$). On the set of all $\le$-Lusztig data we have a crystal structure in the sense of \cite{Ka} inducing a colored graph on the canonical basis. The vertices of this graph, called the crystal graph, are given by the $\le$-Lusztig data and the arrows are induced by Kashiwara operators $f_a$, one for every simple root $\alpha_a$. The action of $f_a$ on a $\le$-Lusztig datum is given explicitly if $\alpha_a$ is the $\le$-minimal positive root. In general the computation of $f_a$ is difficult since it involves a composition of piecewise linear maps.

We aim to give a formula for the computation of $f_a$ that works for all choices of convex orderings. In this paper we solve this problem for simple Lie algebras of type $A$.
 
We assume from now on that $\mathfrak{g}=\sln(\mathbb{C})$. Let $I$ be the index set of simple roots of $\mathfrak{g}$. For each $a\in I$ and each convex order $\le$ of the positive roots of $\mathfrak{g}$ we define in Section \ref{orderlatticesec} finitely many sequences $\gamma=(\gamma_j)$ of positive roots of $\mathfrak{g}$ with certain properties which we call \emph{$a$-crossings}. These sequences come with an order relation $\preceq$. {To each $a$-crossing $\gamma$ we associate a vector $\rvec{\gamma}\in \mathbb{Z}^N$ and linear form $\drvec{\gamma}\in\Hom(\mathbb{Z}^N, \Z)$.} 
Our main result reads as follows.
\begin{thm}[Crossing Formula (Theorem \ref{formula1})] For a $\le$-Lusztig datum $x$ let $\gamma$ be $\preceq$ maximal such that $\drvecx{\gamma}{x}$ is maximal. Then $f_a x=x+\rvec{\gamma}.$ Furthermore, the set of all $\gamma$ with $\rvec{\gamma}=f_a x-x$ can be described explicitly.
\end{thm}

Our work was inspired by the results of \cite{Rei}. Here Reineke obtained an analogue of the Crossing Formula for convex orderings adapted to Dynkin quivers satisfying a homological condition. This was achieved using Ringel's Hall algebra which is isomorphic to $U_q^-$. Under this isomorphism the PBW-type bases for convex ordering are given naturally in terms of representations of quivers. However, for PBW-type bases corresponding to convex orderings that are not adapted to a quiver there are no comparable techniques available. The crossing formula is a generalization of \cite{Rei} to all convex orderings in type $A$ obtained by purely combinatorial methods.

Since a $\le$-Lusztig datum is given by a vector $x\in\mathbb{N}^N$, the action of the Kashiwara operator $f_a$ is given by adding a vector in $f_a x-x\in\mathbb{Z}^N$. We call vectors of the form $f_a x-x$   \emph{Reineke vectors}. Surprisingly these vectors play an important role in the interplay of various parametrizations of canonical bases. The rest of this work is concerned with this relationship.

In Section \ref{MV} we deal with Mirkovi\'c Vilonen (MV) polytopes. These are momentum polytopes associated to Mirkovi\'c Vilonen cycles, certain subvarieties of the affine Grassmannian. Each MV polytope packs together all Lusztig data for varying convex orderings $\le$ that are in bijection with a fixed canonical basis element. Each Lusztig datum can be read of from a distinguished path in the 1-skeleton of the polytope. This defines a crystal isomorphism between the crystal structure on MV polytopes and the crystal structure on $\le$-Lusztig data (\cite{Kam}).

An explicit description of the action of the Kashiwara operators on MV polytopes is desirable and in particular would lead to an inductive procedure to generate all such polytopes. Anderson and Mirkovi\'c conjectured such a description. This conjecture was proved in type $A$ by Kamnitzer (\cite{Kam}) and subsequently Saito (\cite{Sai}). While Kamnitzer gave a counterexample in type $C$ it remains an open question in simply-laced types. Using our description of Reineke vectors in Theorem \ref{formula1} we prove a stronger statement and obtain a new proof of the Anderson-Mirkovi\'c conjecture for type $A$, which provides a conceptual explanation.

In \cite{NS} a modified version of the Anderson-Mirkovi\'c conjecture for the types $B$ and $C$ is proved. It would be interesting to see if their result can also deduced from our methods using folding techniques.

Let $\unip$ be the unipotent radical of a Borel subgroup of $\text{SL}_n(\mathbb{C})$. There are two natural ways to realize an irreducible highest weight representation $V$ of $\mathfrak{g}$. The representation $V$ appears as a quotient of $U(\mathfrak{n}^-)$ but also canonically as a subspace of the coordinate ring $\mathbb{C}[\unip]$ of $\mathcal{N}$ which can be identified with the graded dual $U(\mathfrak{n}^-)^*$. This leads to another parametrizations of $\mathbf{B}$ in terms of string data following work of Kashiwara, Berenstein-Zelevinsky and Littelmann (\cite{Ka2}, \cite{BZ}, \cite{Lit}). This string parametrization, encoded in the string cone, can be thought of more naturally as a parametrization of the dual basis $\mathbf{B^*}$.

On the one hand the action of the Kashiwara operators on the string cone is given by an explicit formula while the inequalities of this cone seem intricate. On the other hand the cone corresponding to Lusztig data is the standard cone given by the positive orthant while the action of the Kashiwara operators seems mysterious. In the Crossing Formula (Theorem \ref{formula1}) we establish a formula for the action of the Kashiwara operators on any $\le$-Lusztig datum in the flavor of the description on string cones. 

In \cite{Ze} Zelikson obtained the remarkable result that for convex orderings $\le$ adapted to quivers of type $A$ the Reineke vectors associated to the set of $\le$-Lusztig data yield defining inequalities of the string cone associated to the reversed order of $\le$. In Section \ref{string} we extend Zelikson's result to all convex orderings in type $A$. This is obtained by attaching a Laurent polynomial $r_{\le}\in \mathbb{Z}[x_{\beta}\mid \beta \in \Phi^{+}]$ to the set of Reineke vectors associated to a convex ordering $\le$. Surprisingly, the functions $r_{\le}$ transform under the geometric lifting of the piecewise-linear transformation of the string cone defined by Berenstein and Zelevinsky in \cite{BZ2}. By this result we recover the string cone inequalities for type $A$ of Gleizer-Postnikov (\cite{GP}). We show that also the dual statement holds: The set of vectors describing the crystal structure on the string cone give the defining inequalities of the positive orthant and therefore the cone of $\le$-Lusztig data. 

In Section \ref{cluster} we recover the duality between Lusztig's parametrization and the string cone in the setup of cluster algebras and mirror symmetry. We identify $r_{\le}$ with two potential functions on the unipotent radical $\unip$ appearing in the current literature. The first potential function $f_{\chi}$ was defined by Berenstein and Kazhdan in the context of geometric crystals in \cite{BK} and the other potential function $W$ by Gross, Hacking, Keel and Kontsevich in the context of mirror duality for cluster varieties in \cite{GHKK}. Each convex ordering $\le$ on $\Phi^+$ defines a cluster in the coordinate ring of $\unip$ and thereby an open embedding of a torus $\mathbf{T}_{\le}$ into $\unip$. We show that $r_{\le}$ coincides (up to an explicit isomorphism) with the restrictions of $f_{\chi}$ and $W$ to $\mathbf{T}_{\le}$. Using this identification and the representation theoretical interpretation of $r_{\le}$ via the Crossing Formula, we obtain explicit formulas for $W$ and $f$ in the corresponding cluster coordinates. By this description we prove structural results about the potential functions (see Theorem \ref{GHKK=Reineke} and Theorem \ref{BK=Reineke}). In particular the enhanced version of the Anderson-Mirkovi\'c conjecture proved in Section \ref{MV} implies that the Landau-Ginzburg potential function defined by Gross, Hacking, Keel and Kontsevich restricted to any $\mathbf{T}_{\le}$ is a Laurent polynomial without constant term, all coefficients in $\{0,1\}$ and exponents in $\{0,-1\}$. As another consequence we obtain that the cones attached to the tropicalizations of the potential functions coincide with the string cones (up to an explicit change of coordinates).

\section*{Acknowledgment}
We are very grateful to Jacinta Torres for sharing an observation that led to this collaboration. This project was initiated during G.K.'s visit at the MPIM Bonn which he thanks for financial support. The project was finalized during B.S.'s visit at the University of Tokyo and V.G.'s and B.S.'s visit at the Laboratoire J.-V. Poncelet. V.G. and B.S. would like to thank Yoshihisa Saito for explanations about MV-polytopes that inspired Section \ref{MV}. We would also like to thank Xin Fang and Hironori Oya for very helpful comments. V.G. was partially supported by the Laboratoire J.-V. Poncelet. G.K. was supported by the grant RSF 16-11- 10075. B.S. was supported by a JSPS Postdoctoral Fellowship for Foreign Researchers, partially supported by the DFG Priority program Darstellungstheorie 1388 and the SFB/TRR 191 'Symplectic Structures in Geometry, Algebra and Dynamics', funded by the DFG.

\section{Rhombic tilings and Lusztig data}
\subsection{Notations}\label{notation}
Let $\mathfrak{g}=\sln(\mathbb{C})$ be the Lie algebra of type $A_{n-1}$ and $\mathfrak{h}\subset \mathfrak{g}$ the Cartan subalgebra consisting of the diagonal matrices in $\mathfrak{g}$. For $s\in[n]:=\{1, \ldots, n\}$ let $\epsilon_s\in \mathfrak{h}^*$ be the functional with $\epsilon_s(\text{diag}(h_1, \ldots,h_n))=h_s$. 
Let $\Phi^+$ be the set of positive roots in $\mathfrak{g}$ given by
\begin{equation*}
\Phi^+=\{\epsilon_s-\epsilon_{t} \mid 1\leq s<t \leq n\}.
\end{equation*}
We abbreviate by $N=\frac{n(n-1)}{2}$ the cardinality of $\Phi^+$. We denote by $\alpha_a:=\epsilon_{a}-\epsilon_{a+1}$ ($a\in[n-1]$) the simple roots of $\mathfrak{g}$ and write $\alpha_{s,t}:=\epsilon_s-\epsilon_{t+1}\in\Phi^+$. The fundamental weights of $\mathfrak{g}$ are given by $\omega_a:=\sum_{s\in [a]}\epsilon_s$.

We denote by $U_q^-$ the negative part of the quantized enveloping algebra of $\mathfrak{g}$ with generators $F_a$, by $B(\infty)$ its crystal basis and by $f_a$ and $e_a$ the corresponding Kashiwara-operators (see \cite{Ka}).

Let $W$ be the Weyl group of $\mathfrak{g}$ which is naturally isomorphic to $\mathfrak{S}_n$, the symmetric group in $n$ letters. The group $W$ is generated by the simple reflections $\sigma_i$ ($i \in [n-1]$) interchanging $i$ and $i+1$ with relations
\begin{align} \notag \sigma_i^2 &=id, \\ \notag
 \sigma_{i_1}\sigma_{i_2} & =\sigma_{i_2}\sigma_{i_1} \qquad \, \, \, \text{for }|i_1-i_2|\ge 2 \quad \text{ (commutation relation)},  \\
\sigma_{i_1}\sigma_{i_2}\sigma_{i_1}&=\sigma_{i_2}\sigma_{i_1}\sigma_{i_2} \quad \text{ for }|i_1-i_2|=1 \quad \text{ (braid relation)}. \label{braidrel}
\end{align}

The Weyl group $W$ has a unique longest element $w_0$ of length $N$. For a reduced expression $\sigma_{i_1}\cdots \sigma_{i_N}$ of $w_0$ we write $\ii:=(i_1, \ldots, i_N)$ and call $\ii$ a \emph{reduced word} (for $w_0$).
We call a total ordering $\le$ on $\Phi^+$ \emph{convex} if whenever $\beta_1,\beta_2,\beta_1+\beta_2\in \Phi^+$ we have either $\beta_1 \le \beta_1+\beta_2 \le \beta_2$ or $\beta_2 \le \beta_1+\beta_2 \le \beta_1$.
The set of total convex ordering is in natural bijection with the set of reduced words. Namely, for a reduced word $\ii=(i_1, \ldots,i_N)$ the total ordering $\le_{\ii}$ on $\Phi^+$ given by
\begin{equation}\label{iorder}
\alpha_{i_1}\le_{\ii} s_{i_1}(\alpha_{i_2}) \le_{\ii} \ldots \le_{\ii} s_{i_1}\cdots s_{i_{N-1}}(\alpha_{i_N})
\end{equation}
is convex and every convex ordering on $\Phi^+$ arises that way.

\subsection{Rhombic tilings}
Our main combinatorial tool are rhombic tilings, which are geometric objects corresponding to commutation equivalence classes of convex orderings on $\Phi^+$.
Let $P_0$ be a regular $2n$-gon with side length $1$. By a \emph{(rhombic) tiling $\mathcal{T}$} we mean a tiling of $P_0$ into rhombi (\emph{tiles}) with side length $1$. We label the edges on the left boundary of $P_0$ consecutively starting with the lowest vertex $v_0$ of $P_0$ and proceeding clockwise. This induces a labeling of all edges in $\mathcal{T}$ such that parallel edges have the same label.

\begin{ex}\label{tilingex} We give an example for $n=5$. Since all other edge labels are determined by the labels of the edges lying on the left boundary of $P_0$ we omit them.
$$
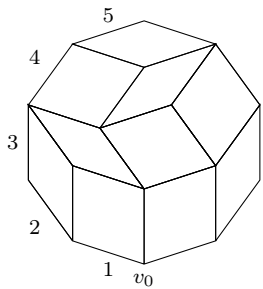
\begin{tikzpicture}
\draw (0.0000000000,1.0000000000)  -- (-0.5877852522,1.8090169943)  -- (0.3632712640,2.1180339887) --
(0.9510565162,1.3090169943)-- (0.0000000000,1.0000000000); 

\draw (0.0000000000,1.0000000000)  -- (-0.9510565162,1.3090169943) -- (-1.5388417685,2.1180339887) --
(-0.5877852522,1.8090169943)-- (0.0000000000,1.0000000000); 

{\draw (0.0000000000,1.0000000000)  -- (-0.9510565162,1.3090169943) -- (-1.5388417685,2.1180339887) --
(-0.5877852522,1.8090169943)-- (0.0000000000,1.0000000000);}

{\draw (0.0000000000,1.0000000000)  -- (-0.5877852522,1.8090169943)  -- (0.3632712640,2.1180339887) --
(0.9510565162,1.3090169943)-- (0.0000000000,1.0000000000);}

{\draw (0.9510565162,1.3090169943)  -- (0.3632712640,2.1180339887) -- (0.9510565162,2.9270509831) --
(1.5388417685,2.1180339887)-- (0.9510565162,1.3090169943);}

\draw (-0.9510565162,0.3090169943)  -- node[font=\scriptsize, below left] {2}(-1.5388417685,1.1180339887)  -- node[font=\scriptsize, left] {3} (-1.5388417685,2.1180339887) --
(-0.9510565162,1.3090169943)-- (-0.9510565162,0.3090169943); 

{\draw (-0.9510565162,0.3090169943)  -- (-1.5388417685,1.1180339887)  -- (-1.5388417685,2.1180339887) --
(-0.9510565162,1.3090169943)-- (-0.9510565162,0.3090169943);}

\draw (-0.5877852522,1.8090169943)  -- (0.0000000000,2.6180339887) -- (0.9510565162,2.9270509831) --
(0.3632712640,2.1180339887)-- (-0.5877852522,1.8090169943); 

\draw (0.0000000000,0.0000000000) --  node[font=\scriptsize, below] {1} (-0.9510565162,0.3090169943) -- (-0.9510565162,1.3090169943) --
(0.0000000000,1.0000000000)-- (0.0000000000,0.0000000000) node[font=\scriptsize, below] {$v_0$} ; 

\draw (0.0000000000,2.6180339887)  -- (-0.9510565162,2.9270509831) -- node[font=\scriptsize, above] {5}  (0.0000000000,3.2360679774) --
(0.9510565162,2.9270509831)-- (0.0000000000,2.6180339887); 

\draw (0.9510565162,1.3090169943)  -- (0.3632712640,2.1180339887) -- (0.9510565162,2.9270509831) --
(1.5388417685,2.1180339887)-- (0.9510565162,1.3090169943); 

\draw (0.9510565162,0.3090169943)  -- (0.9510565162,1.3090169943) -- (1.5388417685,2.1180339887) --
(1.5388417685,1.1180339887)-- (0.9510565162,0.3090169943); 

\draw (-0.5877852522,1.8090169943)  -- (-1.5388417685,2.1180339887)  -- node[font=\scriptsize, above left] {4}  (-0.9510565162,2.9270509831) --
(0.0000000000,2.6180339887)-- (-0.5877852522,1.8090169943); 

\draw (0.0000000000,0.0000000000)  -- (0.0000000000,1.0000000000) -- (0.9510565162,1.3090169943) --
(0.9510565162,0.3090169943)-- (0.0000000000,0.0000000000); 

\end{tikzpicture}$$
\end{ex}
We call a connected set of edges in $\mathcal{T}$ a \emph{border} if it contains exactly one edge with each label.
 Fixing $s\in [2n]$ we define a partition
\begin{equation}\label{eq:kappa}
\bigsqcup_{\kappa} \mathcal{T}_{\kappa}=\mathcal{T}
\end{equation}
of the tiles in $\mathcal{T}$ as follows. Let $b_1, \ldots, b_{2n}$ be the edges of the boundary of $P_0$ in clockwise order starting from $v_0$. Let $B_1$ be the border consisting of the edges $b_{n+s+1}, \ldots, b_{2n+s}$ where we read the indices modulo $2n$. Let $\mathcal{T}_1$ be the set of all tiles in $\mathcal{T}$ intersecting $B_1$ in two edges (a tile of this form always exists by \cite[Lemma 2.1]{El}). We change the border $B_1$ into the new border $B_2$ by replacing for every tile $T\in \mathcal{T}_1$ the two edges intersecting $B_1$ by the other two edges of $T$. We define $\mathcal{T}_2$ as the set of all tiles in $\mathcal{T}\setminus \mathcal{T}_1$ intersection $B_2$ with two edges. We proceed inductively until $B_k$ consists only of edges lying on the boundary of $P_0$.

\begin{ex} We continue with Example \ref{tilingex} and denote by $1\le t < u \leq n$ the unique tile $T$ by $[t,u]$ which has edges with labels in the set $\{t,u \}$. The partition of the tiles with respect to $s=5$ is
$$\mathcal{P}_1=\{[2,3]\}, \quad \mathcal{P}_2=\{[1,3]\}, \quad \mathcal{P}_3=\{[1,2]\}, \quad \mathcal{P}_4=\{[1,4]\}, \quad \mathcal{P}_5=\{[1,5]\}, $$
$$\mathcal{P}_6=\{[4,5]\}, \quad \mathcal{P}_7=\{[2,5]\}, \quad \mathcal{P}_8=\{[3,5],[2,4]\}, \quad  \mathcal{P}_{9}=\{[3,4]\}.$$
The partition of the tiles with respect to $s=3$ is
$$\mathcal{P}_1=\{[2,3]\}, \quad \mathcal{P}_2=\{[1,3]\}, \quad \mathcal{P}_3=\{[1,2],[3,5]\}, \quad \mathcal{P}_4=\{[2,5], [3,4]\}$$
$$\mathcal{P}_5=\{[2,4]\}, \quad \mathcal{P}_6=\{[4,5]\}, \quad \mathcal{P}_7=\{[1,4]\}, \quad \mathcal{P}_8=\{[1,5]\}.$$
\end{ex}
\begin{defi} A sequence $\gamma=(\gamma_i)_{1\leq i\leq m}$ of tiles in $\mathcal{T}$ is called \emph{neighbour sequence} if for all $i$ with $1\leq i < m-1$ the tiles
$\gamma_{i}$ and $\gamma_{i+1}$ share an edge. For $s\in [n]$ we define the \emph{$s$-strip $\li^s$} to be the neighbour sequence $\gamma=(\gamma_i)_{1\leq i\leq n-1}$ consisting of all tiles with edges labeled by $s$ oriented by requiring that $\gamma_1$ intersects the left boundary of $P_0$.
\end{defi}

\begin{ex} We continue with Example \ref{tilingex} and depict the $2$-strip $\li^2$.
$$
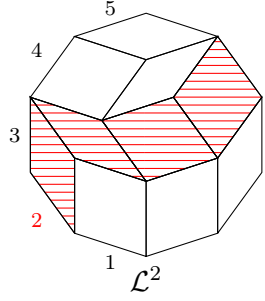
\begin{tikzpicture}
\draw (0.0000000000,1.0000000000)  -- (-0.5877852522,1.8090169943)  -- (0.3632712640,2.1180339887) --
(0.9510565162,1.3090169943)-- (0.0000000000,1.0000000000); 

\draw (0.0000000000,1.0000000000)  -- (-0.9510565162,1.3090169943) -- (-1.5388417685,2.1180339887) --
(-0.5877852522,1.8090169943)-- (0.0000000000,1.0000000000); 

{\draw [pattern color=red, pattern=horizontal lines] (0.0000000000,1.0000000000)  -- (-0.9510565162,1.3090169943) -- (-1.5388417685,2.1180339887) --
(-0.5877852522,1.8090169943)-- (0.0000000000,1.0000000000);}

{\draw [pattern color=red, pattern=horizontal lines] (0.0000000000,1.0000000000)  -- (-0.5877852522,1.8090169943)  -- (0.3632712640,2.1180339887) --
(0.9510565162,1.3090169943)-- (0.0000000000,1.0000000000);}

{\draw [pattern color=red, pattern=horizontal lines] (0.9510565162,1.3090169943)  -- (0.3632712640,2.1180339887) -- (0.9510565162,2.9270509831) --
(1.5388417685,2.1180339887)-- (0.9510565162,1.3090169943);}

\draw (-0.9510565162,0.3090169943)  -- node[font=\scriptsize, below left] {\color{red} 2}(-1.5388417685,1.1180339887)  -- node[font=\scriptsize, left] {3} (-1.5388417685,2.1180339887) --
(-0.9510565162,1.3090169943)-- (-0.9510565162,0.3090169943); 

{\draw [pattern color=red, pattern=horizontal lines] (-0.9510565162,0.3090169943)  -- (-1.5388417685,1.1180339887)  -- (-1.5388417685,2.1180339887) --
(-0.9510565162,1.3090169943)-- (-0.9510565162,0.3090169943);}

\draw (-0.5877852522,1.8090169943)  -- (0.0000000000,2.6180339887) -- (0.9510565162,2.9270509831) --
(0.3632712640,2.1180339887)-- (-0.5877852522,1.8090169943); 

\draw (0.0000000000,0.0000000000) node[below] {$\li^2$} --  node[font=\scriptsize, below] {1} (-0.9510565162,0.3090169943) -- (-0.9510565162,1.3090169943) --
(0.0000000000,1.0000000000)-- (0.0000000000,0.0000000000); 

\draw (0.0000000000,2.6180339887)  -- (-0.9510565162,2.9270509831) -- node[font=\scriptsize, above] {5}  (0.0000000000,3.2360679774) --
(0.9510565162,2.9270509831)-- (0.0000000000,2.6180339887); 

\draw (0.9510565162,1.3090169943)  -- (0.3632712640,2.1180339887) -- (0.9510565162,2.9270509831) --
(1.5388417685,2.1180339887)-- (0.9510565162,1.3090169943); 

\draw (0.9510565162,0.3090169943)  -- (0.9510565162,1.3090169943) -- (1.5388417685,2.1180339887) --
(1.5388417685,1.1180339887)-- (0.9510565162,0.3090169943); 

\draw (-0.5877852522,1.8090169943)  -- (-1.5388417685,2.1180339887)  -- node[font=\scriptsize, above left] {4}  (-0.9510565162,2.9270509831) --
(0.0000000000,2.6180339887)-- (-0.5877852522,1.8090169943); 

\draw (0.0000000000,0.0000000000)  -- (0.0000000000,1.0000000000) -- (0.9510565162,1.3090169943) --
(0.9510565162,0.3090169943)-- (0.0000000000,0.0000000000); 

\end{tikzpicture}$$
\end{ex}

\begin{defi}\label{Gidef}
	For $s\in [2n]$ a neighbour sequence $(\gamma_1, \dots, \gamma_m)\subset\mathcal{T}$ is called \emph{$s$-ascending} if for $1\leq i<m-1$
	$$
	\kappa\left(\gamma_i \right) < \kappa\left(\gamma_{i+1}\right),
	$$
where $\kappa (\gamma_i)$ is defined by $\gamma_i\in \mathcal{T}_{\kappa(\gamma_i)}$ for the partition $\mathcal{T}= \sqcup \mathcal{T}_{\kappa}$ defined in (\ref{eq:kappa}).
\end{defi}
\begin{defi}\label{s-order}
	For $s\in [2n]$ we define a partial ordering $\leq_s$ on $\mathcal{T}$ by setting for $T_1,T_2\in\mathcal{T}$
	$$
	T_1 \leq_s T_2 \Leftrightarrow \exists \text{ $s$-ascending neighbour sequence $\left(\gamma_i\right)_{1\leq i \leq m}$ with $T_1=\gamma_1, T_2=\gamma_m.$}
	$$
\end{defi}

Given two neighbour sequences $\gamma=(\gamma_i)_{1\leq i\leq m}$ and $\lambda=(\lambda_i)_{1\leq i\leq m'}$ we denote by $\gamma\circ\lambda$ the neighbour sequence given by 
\begin{align*}
\gamma \circ \lambda &= \left\{ \begin{array}{ll}
(\rho_i)_{1\leq i \leq m+m'-1} & \text{if $\gamma_1=\lambda_{m'},$} \\ \emptyset & \text{else,}\end{array}\right.\\
\rho_i &= \begin{cases} \lambda_i & \text{if $i< m'$}, \\ \gamma_{i-m'} & \text{if $i>m'$}. \end{cases}
\end{align*}
Setting
$$-\gamma  := \left(\gamma_{m-i+1}\right)_	{1\leq i \leq m}$$
we obtain the following characterization of $a$-ascending neighbour sequences.
\begin{rem}
An $a$-ascending neighbour sequence is composed of segments of $\sgn(a-s) \mathcal{L}^s$ with $s\in[n]$, i.e. the $a$-ascending neighbour sequences are the sequences in $\mathcal{T}$ of the form
$$
\sgn\left( a - s_1 \right) \left( \mathcal{L}^{s_1}_j \right)_{m_1 \leq j \leq m_1'} \circ \dots \circ  \sgn\left( a - s_M \right) \left( \mathcal{L}^{s_M}_j \right)_{m_M \leq j \leq m_M'}.
$$
\end{rem}

\subsection{Tilings associated to reduced words}\label{wordtiling}

For a vertex $v$ of $\mathcal{T}$ we consider a minimal path of edges connecting the lowest vertex $v_0$ and $v$. The set $S$ containing the labels of the edges of such a minimal path is independent of the path and we write $v=v_S$. We denote a tile $T$ with vertices $v_S$, $v_{S \cup \{t\}}$, $v_{S \cup \{t, u\}}$, $v_{S \cup \{u\}}$ by $T=[t,u ; S]$. For a fixed tiling $\mathcal{T}$ and $t,u \in[n]$ with $t\neq u$ there exists a unique tile $T=[t, u; S] = [t, u ; S] \in\mathcal{T}$ which we denote by $[t,u]_{\mathcal{T}}$ or by $[t,u]$.

\begin{ex} We continue with example \ref{tilingex} and depict the vertex $v_{\{2,3,4\}}$ of $\mathcal{T}$ and the tile $T=[4,5; \{2,3\}]$.
$$
\begin{tikzpicture}
\draw (0.0000000000,1.0000000000)  -- (-0.5877852522,1.8090169943)  --
(0.3632712640,2.1180339887) --
(0.9510565162,1.3090169943)-- (0.0000000000,1.0000000000); 

\draw (0.0000000000,1.0000000000)  -- (-0.9510565162,1.3090169943) --
(-1.5388417685,2.1180339887) --
(-0.5877852522,1.8090169943)-- node[font=\scriptsize, left] {2} (0.0000000000,1.0000000000); 

{\draw
(0.0000000000,1.0000000000)  -- (-0.9510565162,1.3090169943) --
(-1.5388417685,2.1180339887) --
(-0.5877852522,1.8090169943)-- (0.0000000000,1.0000000000);}

{\draw
(0.0000000000,1.0000000000)  -- (-0.5877852522,1.8090169943)  --
(0.3632712640,2.1180339887) --
(0.9510565162,1.3090169943)-- (0.0000000000,1.0000000000);}

{\draw
(0.9510565162,1.3090169943)  -- (0.3632712640,2.1180339887) --
(0.9510565162,2.9270509831) --
(1.5388417685,2.1180339887)-- (0.9510565162,1.3090169943);}

\draw (-0.9510565162,0.3090169943)  -- node[font=\scriptsize, below left]
{2}(-1.5388417685,1.1180339887)  -- node[font=\scriptsize,
left] {3} (-1.5388417685,2.1180339887) --
(-0.9510565162,1.3090169943)-- (-0.9510565162,0.3090169943); 

{\draw
(-0.9510565162,0.3090169943)  -- (-1.5388417685,1.1180339887)  --
(-1.5388417685,2.1180339887) --
(-0.9510565162,1.3090169943)-- (-0.9510565162,0.3090169943);}

\draw (-0.5877852522,1.8090169943)  -- (0.0000000000,2.6180339887) --
(0.9510565162,2.9270509831) --
(0.3632712640,2.1180339887)-- (-0.5877852522,1.8090169943); 

\draw (0.0000000000,0.0000000000)  --
node[font=\scriptsize, below] {1} (-0.9510565162,0.3090169943) --
(-0.9510565162,1.3090169943)  --
(0.0000000000,1.0000000000)-- node[font=\scriptsize, left ] {3} (0.0000000000,0.0000000000); 

\draw (0.0000000000,2.6180339887)  -- (-0.9510565162,2.9270509831) --
node[font=\scriptsize, above] {5}  (0.0000000000,3.2360679774) --
(0.9510565162,2.9270509831)-- (0.0000000000,2.6180339887) node[font=\scriptsize, above] {$v_{\{2,3,4\}}$} node[font=\scriptsize, below]{$\quad T$}; 

\draw (0.9510565162,1.3090169943)  -- (0.3632712640,2.1180339887) --
(0.9510565162,2.9270509831) --
(1.5388417685,2.1180339887)-- (0.9510565162,1.3090169943); 

\draw (0.9510565162,0.3090169943)  -- (0.9510565162,1.3090169943) --
(1.5388417685,2.1180339887) --
(1.5388417685,1.1180339887)-- (0.9510565162,0.3090169943); 

\draw (-0.5877852522,1.8090169943)  -- (-1.5388417685,2.1180339887)  --
node[font=\scriptsize, above left] {4}  (-0.9510565162,2.9270509831) --
(0.0000000000,2.6180339887)-- node[font=\scriptsize, left]{$4$} (-0.5877852522,1.8090169943); 

\draw (0.0000000000,0.0000000000)  -- (0.0000000000,1.0000000000) --
(0.9510565162,1.3090169943) --
(0.9510565162,0.3090169943)-- (0.0000000000,0.0000000000); 

\end{tikzpicture}
$$
\end{ex}

Using \cite[Theorem 2.2]{El} we associate to a reduced decomposition $\ii$ (see Section \ref{notation}) a tiling $\mathcal{T}_{\ii}$ as follows. 
We define $\mathcal{T}_{\ii}$ to be the unique tiling such that the total ordering $\le_{\ii}$ (see \eqref{iorder}) on $\mathcal{T}_{\ii}$ is a refinement of $\leq_n$ (see Definition \ref{s-order}) under the identification
\begin{equation}\label{eq:roots}
\left[s, t \right]_{\mathcal{T}_{\ii}} \mapsto \alpha_{s, t-1} \quad  (s<t)
\end{equation}
of $\mathcal{T}_{\ii}$ with  the positive roots $\Phi^+$.

Assume that $\ii=(i_1, \ldots,i_N)$, $\jj =(j_1, \ldots,j_N)$ are two reduced words such their corresponding reduced expressions differ by a braid relation $\sigma_{i_k}\sigma_{i_{k+1}}\sigma_{i_{k+2}}=\sigma_{j_{k}}\sigma_{j_{k+1}}\sigma_{j_{k+2}}$ as in \eqref{braidrel}. Then $\mathcal{T}_{\ii}$ and $\mathcal{T}_{\jj}$ are obtained from each other by interchanging the two subtilings 
\begin{equation}\label{hexbraid}
\left\{\left[s,t;S\right], \left[s,u;S\cup\{t\}\right], \left[t,u;S\right]\right\}  \mathbf{\leftrightarrow}  \left\{\left[t,u;S\cup \{s\}\right], \left[s,u;S\right], \left[s,t; S\cup \{u\}\right]\right\}
\end{equation}
 of a regular 6-gon $[s,t;S] \cup [s,u;S\cup\{t\}] \cup [t,u;S]=[t,u;S\cup \{s\}]\cup [s,u;S]\cup [s,t; S\cup \{u\}]$ with $s<t<u$. This can be illustrated as follows.
$$
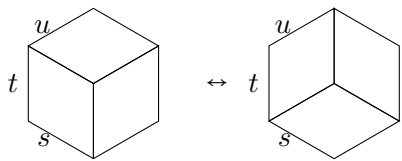
\begin{tikzpicture}
\draw (0.0000000000,0.0000000000)  -- node[left]{$s$} (-0.8660254037,0.4999999999) -- node[left]{$t$} (-0.8660254037,1.5000000000) --
(0.0000000000,1.0000000000)-- (0.0000000000,0.0000000000);
\draw (0.0000000000,1.0000000000)  -- (-0.8660254037,1.5000000000) -- node[left]{$u$} (0.0000000000,2.0000000000)  --
(0.8660254037,1.5000000000)-- (0.0000000000,1.0000000000);
\draw (0.0000000000,0.0000000000)  -- (0.0000000000,1.0000000000) -- (0.8660254037,1.5000000000) node[font=\scriptsize, right, shift={(-0.1,-0.5)}] {$\qquad \mathbf{\leftrightarrow}\qquad $} --
(0.8660254037,0.4999999999)-- (0.0000000000,0.0000000000);

\begin{scope}[xshift=3.2cm]
\draw (0.0000000000,0.0000000000)  -- node[left]{$s$} (-0.8660254037,0.4999999999) -- (0.0000000000,0.9999999999) --
(0.8660254037,0.4999999999)-- (0.0000000000,0.0000000000);
\draw (0.8660254037,0.4999999999)  -- (0.0000000000,0.9999999999) -- (0.0000000000,1.9999999999) --
(0.8660254037,1.5000000000)-- (0.8660254037,0.4999999999);
\draw (-0.8660254037,0.4999999999)  -- node[left]{$t$} (-0.8660254037,1.5000000000) -- node[left]{$u$} (0.0000000000,2.0000000000) --
(0.0000000000,0.9999999999)-- (-0.8660254037,0.4999999999);
\end{scope}
\end{tikzpicture}$$
We call a tiling $\mathcal{H}$ of the regular $6$-gon a \emph{hexagon}. We call the replacement of $\mathcal{H}$ in $\mathcal{T}$ with the unique different tiling of the $6$-gon a \emph{flip of $\mathcal{T}$ at $\mathcal{H}$}.
\begin{ex} We depicture a flip of the tiling $\mathcal{T}_{(1,2,3,1,2,1)}$ at the hexagon given by the tiles $[2,3],[2,4],[3,4]$ yielding the tiling $\mathcal{T}_{(1,2,3,2,1,2)}$.
$$
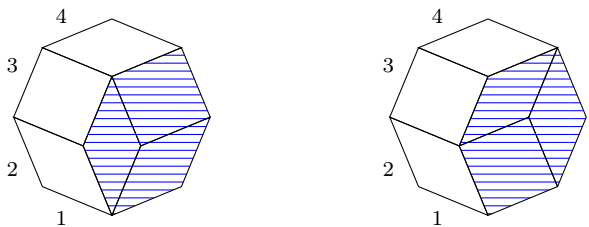
\begin{tikzpicture}
\draw[pattern color=blue, pattern=horizontal lines] (0.3826834323,0.9238795325)  -- (0.0000000000,1.8477590650) -- (0.9238795325,2.2304424973) --
(1.3065629648,1.3065629648)-- (0.3826834323,0.9238795325);
\draw (0.0000000000,0.0000000000)  -- node[font=\scriptsize, below left]
{1} (-0.9238795325,0.3826834323) -- node[font=\scriptsize, below left]
{2} (-1.3065629648,1.3065629648) --
(-0.3826834323,0.9238795325)-- (0.0000000000,0.0000000000);
\draw (-0.3826834323,0.9238795325)  --  (-1.3065629648,1.3065629648) -- node[font=\scriptsize, above left]
{3} (-0.9238795325,2.2304424973) --
(0.0000000000,1.8477590650)-- (-0.3826834323,0.9238795325);
\draw[pattern color=blue, pattern=horizontal lines] (0.0000000000,0.0000000000)  -- (0.3826834323,0.9238795325) -- (1.3065629648,1.3065629648) --
(0.9238795325,0.3826834323)-- (0.0000000000,0.0000000000);
\draw[pattern color=blue, pattern=horizontal lines] (0.0000000000,0.0000000000)  -- (-0.3826834323,0.9238795325) -- (0.0000000000,1.8477590650) --
(0.3826834323,0.9238795325)-- (0.0000000000,0.0000000000);
\draw (0.0000000000,1.8477590650)  -- (-0.9238795325,2.2304424973) -- node[font=\scriptsize, above left]
{4} (0.0000000000,2.6131259297) --
(0.9238795325,2.2304424973)-- (0.0000000000,1.8477590650);

\begin{scope}[shift={(5,0)}]
\draw[pattern color=blue, pattern=horizontal lines] (0.9238795325,0.3826834323)  -- (0.5411961001,1.3065629648) -- (0.9238795325,2.2304424973) --
(1.3065629648,1.3065629648)-- (0.9238795325,0.3826834323);
\draw (0.0000000000,0.0000000000)  -- node[font=\scriptsize, below left]
{1} (-0.9238795325,0.3826834323) -- node[font=\scriptsize, below left]
{2} (-1.3065629648,1.3065629648) --
(-0.3826834323,0.9238795325)-- (0.0000000000,0.0000000000);
\draw (-0.3826834323,0.9238795325)  -- (-1.3065629648,1.3065629648) -- node[font=\scriptsize, above left]
{3} (-0.9238795325,2.2304424973) --
(0.0000000000,1.8477590650)-- (-0.3826834323,0.9238795325);
\draw[pattern color=blue, pattern=horizontal lines] (0.0000000000,0.0000000000)  -- (-0.3826834323,0.9238795325) -- (0.5411961001,1.3065629648) --
(0.9238795325,0.3826834323)-- (0.0000000000,0.0000000000);
\draw[pattern color=blue, pattern=horizontal lines] (-0.3826834323,0.9238795325)  -- (0.0000000000,1.8477590650) -- (0.9238795325,2.2304424973) --
(0.5411961001,1.3065629648)-- (-0.3826834323,0.9238795325);
\draw (0.0000000000,1.8477590650)  -- (-0.9238795325,2.2304424973) -- node[font=\scriptsize, above left]
{4} (0.0000000000,2.6131259297) --
(0.9238795325,2.2304424973)-- (0.0000000000,1.8477590650);
\end{scope}
\end{tikzpicture}$$
\end{ex}

\subsection{Lusztig's parametrization of the canonical basis}\label{sec:crystal}
Using braid group operators Lusztig defined in \cite{Lu} for each reduced word $\ii=(i_1, \ldots,i_N)$ a PBW-type basis of $U_q^-$
$$B_{\ii}=\left\{F_{\ii,\beta_1}^{(x_{\beta_1})}  \cdots F_{\ii,\beta_N}^{(x_{\beta_N})} \mid (x_{\beta_1}, \ldots, x_{\beta_N})\in \mathbb{N}^N\right\},$$
where $\{\beta_1, \ldots,\beta_N\}=\Phi^+$,
$\beta_1 \le_{\ii} \ldots \le_{\ii} \beta_N$, $F_{\ii, \beta_j}=T_{i_1}\cdots T_{i_{j-1}}F_j$ is given via the braid group action $T_i$ defined in \cite[Section1.3]{Lu2}, $X^{(m)}$ is the $q$-divided power defined by $X^{(m)}:=\frac{X^m}{[2]\cdots [m]}$ and $[m]:=q^{m-1}+\ldots +q^{-m+1}$.

\begin{defi}For $x=(x_{1}, \ldots, x_{N})\in \mathbb{N}^N$, we denote the element $F_{\ii,\beta_1}^{(x_{1})} \cdots F_{\ii,\beta_N}^{(x_{N})}$ by $F^{x}$ and call $x$ its \emph{$\ii$-Lusztig datum}. Using the identification of the positive roots with tiles in $\mathcal{T}_{\ii}$ as in (\ref{eq:roots}), we write $x=(x_T)\in \mathbb{N}^{\mathcal{T}_{\ii}}$.
\end{defi}

The sets of $\ii-$ and $\jj-$Lusztig data are related via piecewise linear bijections $$\RR_{\jj}^{\ii}:\mathbb{N}^{\mathcal{T}_{\ii}}\rightarrow \mathbb{N}^{\mathcal{T}_{\jj}}$$ as follows. 
If {$\jj$} is obtained from {$\ii$} by a commutation move replacing the subword $(i_k,i_{k+1})$ by $(i_{k+1},i_k)$ with $\sigma_{i_k} \sigma_{i_{k+1}} = \sigma_{i_k+1} \sigma_{i_{k}}$ we set
\begin{equation}\label{tlmutc}
(\RR_{\jj}^{\ii} x)_{[s,t]}=x_{[s,t]}.
\end{equation}
If {$\jj$} is obtained from {$\ii$} by a braid move that corresponds to a flip at $\mathcal{H}:=\{[s,t], [s,u], [t,u]\} \subset \mathcal{T}_{\ii}$ with $s<t<u$, we set $y=\RR_{\jj}^{\ii} {x}$, where $y_T=x_T$ for $T\notin \{[s,t], [s,u], [t,u]\}$ and  	
\begin{align} \label{tlmutb}
\begin{split}
y_{[s,t]}&=x_{[s,t]}+x_{[s,u]}-\min(x_{[s,t]},x_{[t,u]}),\\ 
y_{[s,u]}&=\min(x_{[s,t]},x_{[t,u]}),\\ 
y_{[t,u]}&=x_{[t,u]}+x_{[s,u]}-\min(x_{[s,t]}, x_{[t,u]}).
\end{split}\end{align}

Any reduced word $\jj$ can be obtained from any other reduced word $\ii$ by a sequence of commutation and braid moves. One checks by direct computation, that the compositions of maps of type $\eqref{tlmutc}$ and $\eqref{tlmutb}$ respect the relations of the symmetric group.  We thus obtain a definition of $\RR^{\ii}_{\jj}$ for $\ii$ and $\jj$ arbitrary.

The $\mathbb{Z}[q]$-lattice $\mathfrak{L}$ spanned by $B_{{\bf i}}$ is independent of the choice of reduced expression $\ii$, as is the induced basis $B:=\pi (B_{\ii})$ of $\mathfrak{L}/q\mathfrak{L}$, where $\pi: \mathfrak{L}\rightarrow \mathfrak{L}/ q\mathfrak{L}$ is the canonical projection. There exists a unique basis ${\bf B}$ of $\mathfrak{L}$ whose image under $\pi$ is $B$ and which is stable under the $\mathbb{Q}$-algebra automorphism preserving the generators of $U_q^-$ and sending $q$ to $q^{-1}$.
The resulting parametrizations of ${\bf B}$ by $\mathbb{N}^{\mathcal{T}_{\ii}}$ are intertwined by the maps $\RR^{\ii}_{\jj}$ and $\mathbb{N}^{\mathcal{T}_{\ii}}$ has a crystal structure isomorphic to $B(\infty)$ (see \cite{L93,BZ2}). It is given as follows.
\begin{defi}\label{def:crystal} Let $a\in [n-1]$ and let ${\ii}=(i_1, \ldots,i_{N})$. We define 
maps
\begin{align*}
f_a &: \N^{\mathcal{T}_{\ii}} \rightarrow \N^{\mathcal{T}_{\ii}},\\
e_a &: \N^{\mathcal{T}_{\ii}}\rightarrow \N^{\mathcal{T}_{\ii}}\sqcup \{0\},\\
\varepsilon_a  &: \N^{\mathcal{T}_{\ii}}\rightarrow \N
\end{align*}
as follows. If $a=i_1$ we set for $x\in \N^{\mathcal{T}_{\ii}}$ and $T\in\mathcal{T}_{\ii}$
\begin{align*}
({f}_a{x})_T&=\begin{cases}
 x_T+1 & \text{ if }T=[a,a+1] \\
 x_T & \text{ else,}
 \end{cases} \\
 ({e}_a{x})_T&=\begin{cases}
0 & \text{if $x_{[a,a+1]}=0$,}\\
 x_T-1 & \text{ if $T=[a,a+1]$ and $x_{[a,a+1]} > 0$,} \\
 x_T & \text{ else,}
 \end{cases}\\
\varepsilon_a (x)&= x_{[a,a+1]}.
\end{align*}
Generally, we define  $f_a, e_a$ and $\varepsilon_a$ by requiring for all reduced words $\ii$ and $\jj$
\begin{align*}
{f}_a &= \RR^{\jj}_{\ii} \circ f_a \circ \RR^{\ii}_{\jj},\\
e_a  &=\RR^{\jj}_{\ii} \circ e_a \circ  \RR^{\ii}_{\jj},\\
\varepsilon_a &= \varepsilon_a \circ \RR^{\ii}_{\jj } .
\end{align*}
To ease notation we suppress the dependency of $f_a$, $e_a$ and $\varepsilon_a$ on $\ii$.
\end{defi}

\subsection{Kashiwara involution on Lusztig's parametrization}\label{Kashiwarainvolution}

Let $*$ be the anti-auto\-morphism of $U_q^-$ preserving the generators $F_a$. By \cite[Theorem 8.1]{Ka}, we have $B(\infty)^*=B(\infty)$ as sets. Therefore this anti-automorphism yields an involution on the crystal basis called \emph{Kashiwara involution}.

For $x\in \mathbb{N}^{\mathcal{T}_{\ii}}$ and $a\in [n-1]$ we define the operators $f_a^*x=(f_a x^* )^*$, $e_a^*x=\left(e_a x^*\right)^*$ and the function $\varepsilon_a^*(x)=\varepsilon_a(x^*)$. This equips the set $\mathbb{N}^{\mathcal{T}_{\ii}}$ with a second crystal structure isomorphic to $B(\infty)$. Explicitly, the $*$-crystal structure is given as follows (see e.g. \cite[Proposition 3.3 (iii)]{BZ2}).
\begin{defi}\label{starop} Let $a\in [n-1]$ and ${\ii}=(i_1, \ldots,i_{N})$. If $a=i_N$ then  we set for $x\in \N^{\mathcal{T}_{\ii}}$ and $T\in\mathcal{T}_{\ii}$ 
\begin{align*}
({f}^*_a{x})_T&=\begin{cases}
 x_T+1 & \text{ if }T=[a,a+1] \\
 x_T & \text{ else,}
 \end{cases} \\
 ({e}^*_a{x})_T&=\begin{cases}
0 & \text{if $x_{[a,a+1]}=0$,}\\
 x_T-1 & \text{ if $T=[a,a+1]$ and $x_{[a,a+1]} > 0$,} \\
 x_T & \text{ else,}
 \end{cases}\\
\varepsilon_a^* (x)&= x_{[a,a+1]}.
\end{align*}
Furthermore, for two reduced words $\ii$ and $\jj$ we require
\begin{align*}
{f}_a^*  &= \RR^{\jj}_{\ii} \circ f_a^* \circ \RR^{\ii}_{\jj}  \\
e_a^* &=\RR^{\jj}_{\ii} \circ e_a^* \circ \RR^{\ii}_{\jj}   \\
\varepsilon_a^*  &= \varepsilon_a^* \circ \RR^{\ii}_{\jj }.
\end{align*}
\end{defi}

\section{Crossing Formula}
In this section we derive a "Crossing Formula" for the crystal structure ($f_a$, $e_a$, $\varepsilon_a$) on Lusztig data $\mathbb{N}^{\mathcal{T}}$.
For this we introduce in \ref{orderlatticesec} a certain poset of crossings related to the Kashiwara operators. In Section \ref{mutseqsection} we provide preparatory results needed in the proof of the formula. In Section \ref{formula} we state and prove the Crossing Formula.
In Section \ref{dualcrossing1} we prove a dual Crossing Formula for the crystal structure  ($f_a^*$, $e_a^*$, $\varepsilon_a^*$) on Lusztig data and describe the explicit embedding of the highest weight crystal $B(\lambda)$ into $B(\infty)$ in terms of Lusztig data. 

\subsection{Poset of a-crossings}\label{orderlatticesec}
Let $a\in [n-1]$ and $\mathcal{T}$ be a tiling of the regular $2n$-gon.
With the notation of Definition \ref{Gidef} the main objects encoding the operation of the Kashiwara operators are  \emph{$a$-crossings} which are defined to be $a$-ascending neighbour sequences $(\gamma_i)_{1\leq i \leq m}\subset \mathcal{T}$ starting at $\gamma_1=\li^a_1$ and ending at $\gamma_m=\li^{a+1}_1$. Any such neighbour sequence is uniquely determined by its associated \emph{strip sequence} $(s_1, \dots, s_M)$, which is defined by
\begin{equation}\label{ssdef}
\gamma=\gamma^1 \circ \dots \circ \gamma^M, \quad \gamma^i \subset \li^{s_i},  \quad s_i\neq s_{i+1},\,  s_1=a,\, s_M=a+1.
\end{equation}
We therefore use the notations $(\gamma_i)\subset\mathcal{T}$ and $(s_i)\subset [n]$ interchangeably to denote an $a$-crossing. We denote the set of $a$-crossings by $\Gamma_a$.

\begin{ex} We continue with Example \ref{tilingex} and depict the $3$-crossing $$\gamma=([2,3],[1,3],[1,2],[2,5],[2,4],[4,5],[1,4]),$$ which is given by the tiles surrounding the line in the picture below. The associated strip sequence is $\gamma=(3,1,2,4)$.

$$
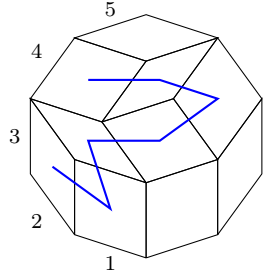
\begin{tikzpicture}
\draw (0.0000000000,1.0000000000)  -- (-0.5877852522,1.8090169943)  -- (0.3632712640,2.1180339887) --
(0.9510565162,1.3090169943)-- (0.0000000000,1.0000000000); 

\draw (0.0000000000,1.0000000000)  -- (-0.9510565162,1.3090169943) -- (-1.5388417685,2.1180339887) --
(-0.5877852522,1.8090169943)-- (0.0000000000,1.0000000000); 

\draw (-0.9510565162,0.3090169943)  -- node[font=\scriptsize, below left] {2} (-1.5388417685,1.1180339887)  -- node[font=\scriptsize, left] {3} (-1.5388417685,2.1180339887) --
(-0.9510565162,1.3090169943)-- (-0.9510565162,0.3090169943); 

\draw (-0.5877852522,1.8090169943)  -- (0.0000000000,2.6180339887) -- (0.9510565162,2.9270509831) --
(0.3632712640,2.1180339887)-- (-0.5877852522,1.8090169943); 

\draw (0.0000000000,0.0000000000)  --  node[font=\scriptsize, below] {1} (-0.9510565162,0.3090169943) -- (-0.9510565162,1.3090169943) --
(0.0000000000,1.0000000000)-- (0.0000000000,0.0000000000); 

\draw (0.0000000000,2.6180339887)  -- (-0.9510565162,2.9270509831) -- node[font=\scriptsize, above] {5}  (0.0000000000,3.2360679774) --
(0.9510565162,2.9270509831)-- (0.0000000000,2.6180339887); 

\draw (0.9510565162,1.3090169943)  -- (0.3632712640,2.1180339887) -- (0.9510565162,2.9270509831) --
(1.5388417685,2.1180339887)-- (0.9510565162,1.3090169943); 

\draw (0.9510565162,0.3090169943)  -- (0.9510565162,1.3090169943) -- (1.5388417685,2.1180339887) --
(1.5388417685,1.1180339887)-- (0.9510565162,0.3090169943); 

\draw (-0.5877852522,1.8090169943)  -- (-1.5388417685,2.1180339887)  -- node[font=\scriptsize, above left] {4}  (-0.9510565162,2.9270509831) --
(0.0000000000,2.6180339887)-- (-0.5877852522,1.8090169943); 

\draw (0.0000000000,0.0000000000)  -- (0.0000000000,1.0000000000) -- (0.9510565162,1.3090169943) --
(0.9510565162,0.3090169943)-- (0.0000000000,0.0000000000); 

\draw [thick,blue] (-1.2449491424,1.2135254915) -- (-0.4755282581,0.6545084971) --
(-0.7694208842,1.5590169943)-- (0.1816356320,1.5590169943)--
(0.9510565162,2.1180339887)--
(0.1816356320,2.3680339887) -- (-0.7694208842,2.3680339887);

\end{tikzpicture}$$
\end{ex}

We introduce a relation $\preceq$ on the set $\Gamma_a$ of $a$-crossings as follows. Since $\leq_a$ defined in Definition \ref{s-order} is anti-symmetric, $\gamma\in \Gamma_a$ cannot contain cycles. Consequently, the set $\mathcal{T}\setminus \{T \mid T \in \gamma\}$ is partitioned into the set of tiles lying on the left of $\gamma$ and the set of tiles lying on the right of $\gamma$. We denote the set consisting of those $T\in\mathcal{T}$ which do not lie right of $\gamma$ by $\cl{\gamma}$ and define for $\gamma, \lambda \in \Gamma_a$
\begin{align}\notag
\op{\gamma}&:=\cl{\gamma}-\gamma,\\\label{orderdef}
\gamma \preceq \lambda &:\Leftrightarrow
\left(\cl{\gamma}\subseteq \cl{\lambda}\right) \wedge
\left(\op{\gamma} \subseteq \op{\lambda}\right).
\end{align}
We show that $\preceq$ turns $\Gamma_a$ into a poset:
\begin{prop}\label{isorder}
 The relation $\preceq$ on $\Gamma_a$ is an order relation.
\end{prop}
\begin{proof}
 Clearly, $\preceq$ is reflexive and transitive. For $\gamma, \lambda \in \Gamma_a$ we have
 $$
 \gamma \preceq \lambda \preceq \gamma \quad \Leftrightarrow\quad  (\cl{\gamma}=  \cl{\lambda}) \wedge   (\op{\gamma}  =  \op{\lambda})  \quad\Rightarrow\quad \gamma =  \cl{\gamma} - \op{\gamma}=  \cl{\lambda}-\op{\lambda}
	=\lambda,
 $$
 where the equalities are equalities of sets. Since $\gamma$ and $\lambda$ are $a$-ascending neighbour sequences, $\gamma=\lambda$ also holds as an equality of sequences. Thus, $\preceq$ is anti-symmetric.
 \end{proof}
We illustrate the poset  $\Gamma_a$ in an example.
\begin{ex}\label{ex:poset} We continue with our running example and depict the poset of $3$-crossings in $\mathcal{T}$ given in Example \ref{tilingex}. The crossings are given by the tiles surrounding the blue lines.
$$\scalebox{.45}
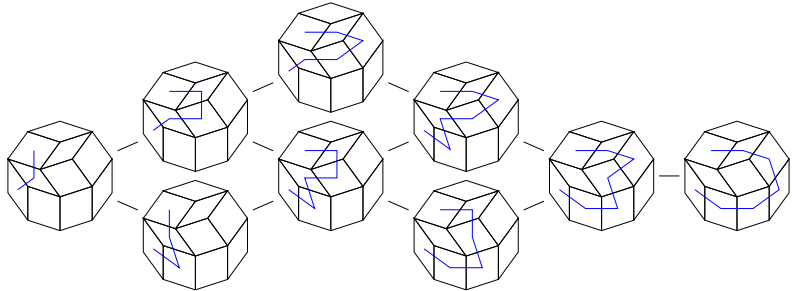
{
   \begin{tikzpicture}
  \begin{scope}[shift={(0,5.25)}]
\draw [thick] (-1.7, 2.3617839887) -- (-2.3, 2.6242839887);
\draw [thick] (-1.7, 0.8742839887) -- (-2.3, 0.6117839887);
\draw (0.0000000000,1.0000000000)  -- (-0.5877852522,1.8090169943)  -- (0.3632712640,2.1180339887) --
(0.9510565162,1.3090169943)-- (0.0000000000,1.0000000000); 
\draw (0.0000000000,1.0000000000)  -- (-0.9510565162,1.3090169943) -- (-1.5388417685,2.1180339887) --
(-0.5877852522,1.8090169943)-- (0.0000000000,1.0000000000); 
\draw (-0.9510565162,0.3090169943)  -- node[font=\scriptsize, below left] {} (-1.5388417685,1.1180339887)  -- node[font=\scriptsize, left] {} (-1.5388417685,2.1180339887) --
(-0.9510565162,1.3090169943)-- (-0.9510565162,0.3090169943); 
\draw (-0.5877852522,1.8090169943)  -- (0.0000000000,2.6180339887) -- (0.9510565162,2.9270509831) --
(0.3632712640,2.1180339887)-- (-0.5877852522,1.8090169943); 
\draw (0.0000000000,0.0000000000)  --  node[font=\scriptsize, below] {} (-0.9510565162,0.3090169943) -- (-0.9510565162,1.3090169943) --
(0.0000000000,1.0000000000)-- (0.0000000000,0.0000000000); 
\draw (0.0000000000,2.6180339887)  -- (-0.9510565162,2.9270509831) -- node[font=\scriptsize, above] {}  (0.0000000000,3.2360679774) --
(0.9510565162,2.9270509831)-- (0.0000000000,2.6180339887); 
\draw (0.9510565162,1.3090169943)  -- (0.3632712640,2.1180339887) -- (0.9510565162,2.9270509831) --
(1.5388417685,2.1180339887)-- (0.9510565162,1.3090169943); 
\draw (0.9510565162,0.3090169943)  -- (0.9510565162,1.3090169943) -- (1.5388417685,2.1180339887) --
(1.5388417685,1.1180339887)-- (0.9510565162,0.3090169943); 
\draw (-0.5877852522,1.8090169943)  -- (-1.5388417685,2.1180339887)  -- node[font=\scriptsize, above left] {}  (-0.9510565162,2.9270509831) --
(0.0000000000,2.6180339887)-- (-0.5877852522,1.8090169943); 
\draw (0.0000000000,0.0000000000)  -- (0.0000000000,1.0000000000) -- (0.9510565162,1.3090169943) --
(0.9510565162,0.3090169943)-- (0.0000000000,0.0000000000); 
{\draw [thick,blue] (-1.2449491424,1.2135254915) -- (-0.4755282581,0.6545084971) --
(-0.7694208842,1.5590169943)-- (0.1816356320,1.5590169943)--
(0.9510565162,2.1180339887)--
(0.1816356320,2.3680339887) -- (-0.7694208842,2.3680339887);}
  \end{scope}

\begin{scope}[shift={(-12,3.5)}]
\draw (0.0000000000,1.0000000000)  -- (-0.5877852522,1.8090169943)  -- (0.3632712640,2.1180339887) --
(0.9510565162,1.3090169943)-- (0.0000000000,1.0000000000); 
\draw (0.0000000000,1.0000000000)  -- (-0.9510565162,1.3090169943) -- (-1.5388417685,2.1180339887) --
(-0.5877852522,1.8090169943)-- (0.0000000000,1.0000000000); 
\draw (-0.9510565162,0.3090169943)  -- node[font=\scriptsize, below left] {} (-1.5388417685,1.1180339887)  -- node[font=\scriptsize, left] {} (-1.5388417685,2.1180339887) --
(-0.9510565162,1.3090169943)-- (-0.9510565162,0.3090169943); 
\draw (-0.5877852522,1.8090169943)  -- (0.0000000000,2.6180339887) -- (0.9510565162,2.9270509831) --
(0.3632712640,2.1180339887)-- (-0.5877852522,1.8090169943); 
\draw (0.0000000000,0.0000000000)  --  node[font=\scriptsize, below] {} (-0.9510565162,0.3090169943) -- (-0.9510565162,1.3090169943) --
(0.0000000000,1.0000000000)-- (0.0000000000,0.0000000000); 
\draw (0.0000000000,2.6180339887)  -- (-0.9510565162,2.9270509831) -- node[font=\scriptsize, above] {}  (0.0000000000,3.2360679774) --
(0.9510565162,2.9270509831)-- (0.0000000000,2.6180339887); 
\draw (0.9510565162,1.3090169943)  -- (0.3632712640,2.1180339887) -- (0.9510565162,2.9270509831) --
(1.5388417685,2.1180339887)-- (0.9510565162,1.3090169943); 
\draw (0.9510565162,0.3090169943)  -- (0.9510565162,1.3090169943) -- (1.5388417685,2.1180339887) --
(1.5388417685,1.1180339887)-- (0.9510565162,0.3090169943); 
\draw (-0.5877852522,1.8090169943)  -- (-1.5388417685,2.1180339887)  -- node[font=\scriptsize, above left] {}  (-0.9510565162,2.9270509831) --
(0.0000000000,2.6180339887)-- (-0.5877852522,1.8090169943); 
\draw (0.0000000000,0.0000000000)  -- (0.0000000000,1.0000000000) -- (0.9510565162,1.3090169943) --
(0.9510565162,0.3090169943)-- (0.0000000000,0.0000000000); 
{\draw [thick,blue]
(-1.2449491424,1.2135254915) --
(-0.7694208842,1.5590169943) -- (-0.7694208842,2.3680339887);}
  \end{scope}

\begin{scope}[shift={(-8,5.25)}]
\draw [thick] (-1.7, 0.8742839887) -- (-2.3, 0.6117839887);
\draw (0.0000000000,1.0000000000)  -- (-0.5877852522,1.8090169943)  -- (0.3632712640,2.1180339887) --
(0.9510565162,1.3090169943)-- (0.0000000000,1.0000000000); 
\draw (0.0000000000,1.0000000000)  -- (-0.9510565162,1.3090169943) -- (-1.5388417685,2.1180339887) --
(-0.5877852522,1.8090169943)-- (0.0000000000,1.0000000000); 
\draw (-0.9510565162,0.3090169943)  -- node[font=\scriptsize, below left] {} (-1.5388417685,1.1180339887)  -- node[font=\scriptsize, left] {} (-1.5388417685,2.1180339887) --
(-0.9510565162,1.3090169943)-- (-0.9510565162,0.3090169943); 
\draw (-0.5877852522,1.8090169943)  -- (0.0000000000,2.6180339887) -- (0.9510565162,2.9270509831) --
(0.3632712640,2.1180339887)-- (-0.5877852522,1.8090169943); 
\draw (0.0000000000,0.0000000000)  --  node[font=\scriptsize, below] {} (-0.9510565162,0.3090169943) -- (-0.9510565162,1.3090169943) --
(0.0000000000,1.0000000000)-- (0.0000000000,0.0000000000); 
\draw (0.0000000000,2.6180339887)  -- (-0.9510565162,2.9270509831) -- node[font=\scriptsize, above] {}  (0.0000000000,3.2360679774) --
(0.9510565162,2.9270509831)-- (0.0000000000,2.6180339887); 
\draw (0.9510565162,1.3090169943)  -- (0.3632712640,2.1180339887) -- (0.9510565162,2.9270509831) --
(1.5388417685,2.1180339887)-- (0.9510565162,1.3090169943); 
\draw (0.9510565162,0.3090169943)  -- (0.9510565162,1.3090169943) -- (1.5388417685,2.1180339887) --
(1.5388417685,1.1180339887)-- (0.9510565162,0.3090169943); 
\draw (-0.5877852522,1.8090169943)  -- (-1.5388417685,2.1180339887)  -- node[font=\scriptsize, above left] {}  (-0.9510565162,2.9270509831) --
(0.0000000000,2.6180339887)-- (-0.5877852522,1.8090169943); 
\draw (0.0000000000,0.0000000000)  -- (0.0000000000,1.0000000000) -- (0.9510565162,1.3090169943) --
(0.9510565162,0.3090169943)-- (0.0000000000,0.0000000000); 
{\draw [thick,blue] (-1.2449491424,1.2135254915) --  (-0.7694208842,1.5590169943)--
(0.1816356320,1.5590169943) -- (0.1816356320,2.3680339887) -- (-0.7694208842,2.3680339887);}
  \end{scope}

 \begin{scope}[shift={(-4,7)}]
\draw [thick] (-1.7, 0.8742839887) -- (-2.3, 0.6117839887);
\draw (0.0000000000,1.0000000000)  -- (-0.5877852522,1.8090169943)  -- (0.3632712640,2.1180339887) --
(0.9510565162,1.3090169943)-- (0.0000000000,1.0000000000); 
\draw (0.0000000000,1.0000000000)  -- (-0.9510565162,1.3090169943) -- (-1.5388417685,2.1180339887) --
(-0.5877852522,1.8090169943)-- (0.0000000000,1.0000000000); 
\draw (-0.9510565162,0.3090169943)  -- node[font=\scriptsize, below left] {} (-1.5388417685,1.1180339887)  -- node[font=\scriptsize, left] {} (-1.5388417685,2.1180339887) --
(-0.9510565162,1.3090169943)-- (-0.9510565162,0.3090169943); 
\draw (-0.5877852522,1.8090169943)  -- (0.0000000000,2.6180339887) -- (0.9510565162,2.9270509831) --
(0.3632712640,2.1180339887)-- (-0.5877852522,1.8090169943); 
\draw (0.0000000000,0.0000000000)  --  node[font=\scriptsize, below] {} (-0.9510565162,0.3090169943) -- (-0.9510565162,1.3090169943) --
(0.0000000000,1.0000000000)-- (0.0000000000,0.0000000000); 
\draw (0.0000000000,2.6180339887)  -- (-0.9510565162,2.9270509831) -- node[font=\scriptsize, above] {}  (0.0000000000,3.2360679774) --
(0.9510565162,2.9270509831)-- (0.0000000000,2.6180339887); 
\draw (0.9510565162,1.3090169943)  -- (0.3632712640,2.1180339887) -- (0.9510565162,2.9270509831) --
(1.5388417685,2.1180339887)-- (0.9510565162,1.3090169943); 
\draw (0.9510565162,0.3090169943)  -- (0.9510565162,1.3090169943) -- (1.5388417685,2.1180339887) --
(1.5388417685,1.1180339887)-- (0.9510565162,0.3090169943); 
\draw (-0.5877852522,1.8090169943)  -- (-1.5388417685,2.1180339887)  -- node[font=\scriptsize, above left] {}  (-0.9510565162,2.9270509831) --
(0.0000000000,2.6180339887)-- (-0.5877852522,1.8090169943); 
\draw (0.0000000000,0.0000000000)  -- (0.0000000000,1.0000000000) -- (0.9510565162,1.3090169943) --
(0.9510565162,0.3090169943)-- (0.0000000000,0.0000000000); 
{\draw [thick,blue]  (-1.2449491424,1.2135254915) --  (-0.7694208842,1.5590169943) --
(0.1816356320,1.5590169943) --  (0.9510565162,2.1180339887) -- (0.1816356320,2.3680339887) -- (-0.7694208842,2.3680339887);}
  \end{scope}

  \begin{scope}[shift={(-8,1.75)}]
\draw [thick] (-1.7, 2.3617839887) -- (-2.3, 2.6242839887);
\draw (0.0000000000,1.0000000000)  -- (-0.5877852522,1.8090169943)  -- (0.3632712640,2.1180339887) --
(0.9510565162,1.3090169943)-- (0.0000000000,1.0000000000); 
\draw (0.0000000000,1.0000000000)  -- (-0.9510565162,1.3090169943) -- (-1.5388417685,2.1180339887) --
(-0.5877852522,1.8090169943)-- (0.0000000000,1.0000000000); 
\draw (-0.9510565162,0.3090169943)  -- node[font=\scriptsize, below left] {} (-1.5388417685,1.1180339887)  -- node[font=\scriptsize, left] {} (-1.5388417685,2.1180339887) --
(-0.9510565162,1.3090169943)-- (-0.9510565162,0.3090169943); 
\draw (-0.5877852522,1.8090169943)  -- (0.0000000000,2.6180339887) -- (0.9510565162,2.9270509831) --
(0.3632712640,2.1180339887)-- (-0.5877852522,1.8090169943); 
\draw (0.0000000000,0.0000000000)  --  node[font=\scriptsize, below] {} (-0.9510565162,0.3090169943) -- (-0.9510565162,1.3090169943) --
(0.0000000000,1.0000000000)-- (0.0000000000,0.0000000000); 
\draw (0.0000000000,2.6180339887)  -- (-0.9510565162,2.9270509831) -- node[font=\scriptsize, above] {}  (0.0000000000,3.2360679774) --
(0.9510565162,2.9270509831)-- (0.0000000000,2.6180339887); 
\draw (0.9510565162,1.3090169943)  -- (0.3632712640,2.1180339887) -- (0.9510565162,2.9270509831) --
(1.5388417685,2.1180339887)-- (0.9510565162,1.3090169943); 
\draw (0.9510565162,0.3090169943)  -- (0.9510565162,1.3090169943) -- (1.5388417685,2.1180339887) --
(1.5388417685,1.1180339887)-- (0.9510565162,0.3090169943); 
\draw (-0.5877852522,1.8090169943)  -- (-1.5388417685,2.1180339887)  -- node[font=\scriptsize, above left] {}  (-0.9510565162,2.9270509831) --
(0.0000000000,2.6180339887)-- (-0.5877852522,1.8090169943); 
\draw (0.0000000000,0.0000000000)  -- (0.0000000000,1.0000000000) -- (0.9510565162,1.3090169943) --
(0.9510565162,0.3090169943)-- (0.0000000000,0.0000000000); 
{\draw [thick,blue] (-1.2449491424,1.2135254915)  -- (-0.4755282581,0.6545084971) --  (-0.7694208842,1.5590169943) -- (-0.7694208842,2.3680339887);}
  \end{scope}

 \begin{scope}[shift={(-4,3.5)}]
\draw [thick] (-1.7, 2.3617839887) -- (-2.3, 2.6242839887);
\draw [thick] (-1.7, 0.8742839887) -- (-2.3, 0.6117839887);
\draw (0.0000000000,1.0000000000)  -- (-0.5877852522,1.8090169943)  -- (0.3632712640,2.1180339887) --
(0.9510565162,1.3090169943)-- (0.0000000000,1.0000000000); 
\draw (0.0000000000,1.0000000000)  -- (-0.9510565162,1.3090169943) -- (-1.5388417685,2.1180339887) --
(-0.5877852522,1.8090169943)-- (0.0000000000,1.0000000000); 
\draw (-0.9510565162,0.3090169943)  -- node[font=\scriptsize, below left] {} (-1.5388417685,1.1180339887)  -- node[font=\scriptsize, left] {} (-1.5388417685,2.1180339887) --
(-0.9510565162,1.3090169943)-- (-0.9510565162,0.3090169943); 
\draw (-0.5877852522,1.8090169943)  -- (0.0000000000,2.6180339887) -- (0.9510565162,2.9270509831) --
(0.3632712640,2.1180339887)-- (-0.5877852522,1.8090169943); 
\draw (0.0000000000,0.0000000000)  --  node[font=\scriptsize, below] {} (-0.9510565162,0.3090169943) -- (-0.9510565162,1.3090169943) --
(0.0000000000,1.0000000000)-- (0.0000000000,0.0000000000); 
\draw (0.0000000000,2.6180339887)  -- (-0.9510565162,2.9270509831) -- node[font=\scriptsize, above] {}  (0.0000000000,3.2360679774) --
(0.9510565162,2.9270509831)-- (0.0000000000,2.6180339887); 
\draw (0.9510565162,1.3090169943)  -- (0.3632712640,2.1180339887) -- (0.9510565162,2.9270509831) --
(1.5388417685,2.1180339887)-- (0.9510565162,1.3090169943); 
\draw (0.9510565162,0.3090169943)  -- (0.9510565162,1.3090169943) -- (1.5388417685,2.1180339887) --
(1.5388417685,1.1180339887)-- (0.9510565162,0.3090169943); 
\draw (-0.5877852522,1.8090169943)  -- (-1.5388417685,2.1180339887)  -- node[font=\scriptsize, above left] {}  (-0.9510565162,2.9270509831) --
(0.0000000000,2.6180339887)-- (-0.5877852522,1.8090169943); 
\draw (0.0000000000,0.0000000000)  -- (0.0000000000,1.0000000000) -- (0.9510565162,1.3090169943) --
(0.9510565162,0.3090169943)-- (0.0000000000,0.0000000000); 
{\draw [thick,blue] (-1.2449491424,1.2135254915) -- (-0.4755282581,0.6545084971)
-- (-0.7694208842,1.5590169943)-- (0.1816356320,1.5590169943)--
(0.1816356320,2.3680339887) -- (-0.7694208842,2.3680339887);}
  \end{scope}

   \begin{scope}[shift={(0,1.75)}]
\draw [thick] (-1.7, 2.3617839887) -- (-2.3, 2.6242839887);
\draw (0.0000000000,1.0000000000)  -- (-0.5877852522,1.8090169943)  -- (0.3632712640,2.1180339887) --
(0.9510565162,1.3090169943)-- (0.0000000000,1.0000000000); 
\draw (0.0000000000,1.0000000000)  -- (-0.9510565162,1.3090169943) -- (-1.5388417685,2.1180339887) --
(-0.5877852522,1.8090169943)-- (0.0000000000,1.0000000000); 
\draw (-0.9510565162,0.3090169943)  -- node[font=\scriptsize, below left] {} (-1.5388417685,1.1180339887)  -- node[font=\scriptsize, left] {} (-1.5388417685,2.1180339887) --
(-0.9510565162,1.3090169943)-- (-0.9510565162,0.3090169943); 
\draw (-0.5877852522,1.8090169943)  -- (0.0000000000,2.6180339887) -- (0.9510565162,2.9270509831) --
(0.3632712640,2.1180339887)-- (-0.5877852522,1.8090169943); 
\draw (0.0000000000,0.0000000000)  --  node[font=\scriptsize, below] {} (-0.9510565162,0.3090169943) -- (-0.9510565162,1.3090169943) --
(0.0000000000,1.0000000000)-- (0.0000000000,0.0000000000); 
\draw (0.0000000000,2.6180339887)  -- (-0.9510565162,2.9270509831) -- node[font=\scriptsize, above] {}  (0.0000000000,3.2360679774) --
(0.9510565162,2.9270509831)-- (0.0000000000,2.6180339887); 
\draw (0.9510565162,1.3090169943)  -- (0.3632712640,2.1180339887) -- (0.9510565162,2.9270509831) --
(1.5388417685,2.1180339887)-- (0.9510565162,1.3090169943); 
\draw (0.9510565162,0.3090169943)  -- (0.9510565162,1.3090169943) -- (1.5388417685,2.1180339887) --
(1.5388417685,1.1180339887)-- (0.9510565162,0.3090169943); 
\draw (-0.5877852522,1.8090169943)  -- (-1.5388417685,2.1180339887)  -- node[font=\scriptsize, above left] {}  (-0.9510565162,2.9270509831) --
(0.0000000000,2.6180339887)-- (-0.5877852522,1.8090169943); 
\draw (0.0000000000,0.0000000000)  -- (0.0000000000,1.0000000000) -- (0.9510565162,1.3090169943) --
(0.9510565162,0.3090169943)-- (0.0000000000,0.0000000000); 
{\draw [thick,blue] (-1.2449491424,1.2135254915) -- (-0.4755282581,0.6545084971) --
(0.4755282581,0.6545084971) -- (0.1816356320,1.5590169943)
-- (0.1816356320,2.3680339887) -- (-0.7694208842,2.3680339887);;}
  \end{scope}

  \begin{scope}[shift={(4,3.5)}]
\draw [thick] (-1.7, 2.3617839887) -- (-2.3, 2.6242839887);
\draw [thick] (-1.7, 0.8742839887) -- (-2.3, 0.6117839887);
\draw (0.0000000000,1.0000000000)  -- (-0.5877852522,1.8090169943)  -- (0.3632712640,2.1180339887) --
(0.9510565162,1.3090169943)-- (0.0000000000,1.0000000000); 
\draw (0.0000000000,1.0000000000)  -- (-0.9510565162,1.3090169943) -- (-1.5388417685,2.1180339887) --
(-0.5877852522,1.8090169943)-- (0.0000000000,1.0000000000); 
\draw (-0.9510565162,0.3090169943)  -- node[font=\scriptsize, below left] {} (-1.5388417685,1.1180339887)  -- node[font=\scriptsize, left] {} (-1.5388417685,2.1180339887) --
(-0.9510565162,1.3090169943)-- (-0.9510565162,0.3090169943); 
\draw (-0.5877852522,1.8090169943)  -- (0.0000000000,2.6180339887) -- (0.9510565162,2.9270509831) --
(0.3632712640,2.1180339887)-- (-0.5877852522,1.8090169943); 
\draw (0.0000000000,0.0000000000)  --  node[font=\scriptsize, below] {} (-0.9510565162,0.3090169943) -- (-0.9510565162,1.3090169943) --
(0.0000000000,1.0000000000)-- (0.0000000000,0.0000000000); 
\draw (0.0000000000,2.6180339887)  -- (-0.9510565162,2.9270509831) -- node[font=\scriptsize, above] {}  (0.0000000000,3.2360679774) --
(0.9510565162,2.9270509831)-- (0.0000000000,2.6180339887); 
\draw (0.9510565162,1.3090169943)  -- (0.3632712640,2.1180339887) -- (0.9510565162,2.9270509831) --
(1.5388417685,2.1180339887)-- (0.9510565162,1.3090169943); 
\draw (0.9510565162,0.3090169943)  -- (0.9510565162,1.3090169943) -- (1.5388417685,2.1180339887) --
(1.5388417685,1.1180339887)-- (0.9510565162,0.3090169943); 
\draw (-0.5877852522,1.8090169943)  -- (-1.5388417685,2.1180339887)  -- node[font=\scriptsize, above left] {}  (-0.9510565162,2.9270509831) --
(0.0000000000,2.6180339887)-- (-0.5877852522,1.8090169943); 
\draw (0.0000000000,0.0000000000)  -- (0.0000000000,1.0000000000) -- (0.9510565162,1.3090169943) --
(0.9510565162,0.3090169943)-- (0.0000000000,0.0000000000); 
{\draw [thick,blue] (-1.2449491424,1.2135254915) -- (-0.4755282581,0.6545084971) --
(0.4755282581,0.6545084971) -- (0.1816356320,1.5590169943)
-- (0.9510565162,2.1180339887) --
(0.1816356320,2.3680339887) -- (-0.7694208842,2.3680339887);;}
  \end{scope}

   \begin{scope}[shift={(8,3.5)}]
\draw [thick] (-1.7, 1.6180339887) -- (-2.3, 1.6180339887);
\draw (0.0000000000,1.0000000000)  -- (-0.5877852522,1.8090169943)  -- (0.3632712640,2.1180339887) --
(0.9510565162,1.3090169943)-- (0.0000000000,1.0000000000); 
\draw (0.0000000000,1.0000000000)  -- (-0.9510565162,1.3090169943) -- (-1.5388417685,2.1180339887) --
(-0.5877852522,1.8090169943)-- (0.0000000000,1.0000000000); 
\draw (-0.9510565162,0.3090169943)  -- node[font=\scriptsize, below left] {} (-1.5388417685,1.1180339887)  -- node[font=\scriptsize, left] {} (-1.5388417685,2.1180339887) --
(-0.9510565162,1.3090169943)-- (-0.9510565162,0.3090169943); 
\draw (-0.5877852522,1.8090169943)  -- (0.0000000000,2.6180339887) -- (0.9510565162,2.9270509831) --
(0.3632712640,2.1180339887)-- (-0.5877852522,1.8090169943); 
\draw (0.0000000000,0.0000000000)  --  node[font=\scriptsize, below] {} (-0.9510565162,0.3090169943) -- (-0.9510565162,1.3090169943) --
(0.0000000000,1.0000000000)-- (0.0000000000,0.0000000000); 
\draw (0.0000000000,2.6180339887)  -- (-0.9510565162,2.9270509831) -- node[font=\scriptsize, above] {}  (0.0000000000,3.2360679774) --
(0.9510565162,2.9270509831)-- (0.0000000000,2.6180339887); 
\draw (0.9510565162,1.3090169943)  -- (0.3632712640,2.1180339887) -- (0.9510565162,2.9270509831) --
(1.5388417685,2.1180339887)-- (0.9510565162,1.3090169943); 
\draw (0.9510565162,0.3090169943)  -- (0.9510565162,1.3090169943) -- (1.5388417685,2.1180339887) --
(1.5388417685,1.1180339887)-- (0.9510565162,0.3090169943); 
\draw (-0.5877852522,1.8090169943)  -- (-1.5388417685,2.1180339887)  -- node[font=\scriptsize, above left] {}  (-0.9510565162,2.9270509831) --
(0.0000000000,2.6180339887)-- (-0.5877852522,1.8090169943); 
\draw (0.0000000000,0.0000000000)  -- (0.0000000000,1.0000000000) -- (0.9510565162,1.3090169943) --
(0.9510565162,0.3090169943)-- (0.0000000000,0.0000000000); 
{\draw [thick,blue] (-1.2449491424,1.2135254915) -- (-0.4755282581,0.6545084971) --
(0.4755282581,0.6545084971) -- (1.2449491424,1.2135254915) -- (0.9510565162,2.1180339887) --
(0.1816356320,2.3680339887) -- (-0.7694208842,2.3680339887);}
  \end{scope}

\end{tikzpicture}}$$
\end{ex}

There is a unique maximal element in $\Gamma_a$. It is given by the strip sequence $(a, a+1)$. Thus any $\gamma\in\Gamma_a$ consists of tiles in
\begin{equation}\label{def:comb}
\mathcal{W}_a := \cl{(a, a+1)}\subset \mathcal{T}.
\end{equation}
We call $\mathcal{W}_a$ the \emph{$a$-comb}.

\begin{ex}\label{ex:combs} We continue with Example \ref{tilingex} and depict $\mathcal{W}_1$ and $\mathcal{W}_3$ in $\mathcal{T}$
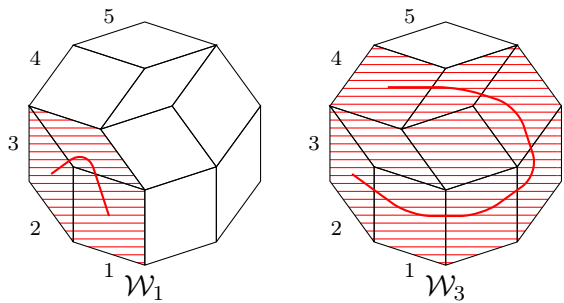
\begin{center}
\begin{tikzpicture}
\draw (0.0000000000,1.0000000000)  -- (-0.5877852522,1.8090169943)  -- (0.3632712640,2.1180339887) --
(0.9510565162,1.3090169943)-- (0.0000000000,1.0000000000); 

\draw  [pattern color=red, pattern=horizontal lines] (0.0000000000,1.0000000000)  -- (-0.9510565162,1.3090169943) -- (-1.5388417685,2.1180339887) --
(-0.5877852522,1.8090169943)-- (0.0000000000,1.0000000000); 

\draw [pattern color=red, pattern=horizontal lines] (-0.9510565162,0.3090169943)  -- node[font=\scriptsize, below left] {2} (-1.5388417685,1.1180339887)  -- node[font=\scriptsize, left] {3} (-1.5388417685,2.1180339887) --
(-0.9510565162,1.3090169943)-- (-0.9510565162,0.3090169943); 

\draw (-0.5877852522,1.8090169943)  -- (0.0000000000,2.6180339887) -- (0.9510565162,2.9270509831) --
(0.3632712640,2.1180339887)-- (-0.5877852522,1.8090169943); 

\draw [pattern color=red, pattern=horizontal lines] (0.0000000000,0.0000000000)  node[below] {$\mathcal{W}_1$} --  node[font=\scriptsize, below] {1} (-0.9510565162,0.3090169943) -- (-0.9510565162,1.3090169943) --
(0.0000000000,1.0000000000)-- (0.0000000000,0.0000000000); 

\draw (0.0000000000,2.6180339887)  -- (-0.9510565162,2.9270509831) -- node[font=\scriptsize, above] {5}  (0.0000000000,3.2360679774) --
(0.9510565162,2.9270509831)-- (0.0000000000,2.6180339887); 

\draw (0.9510565162,1.3090169943)  -- (0.3632712640,2.1180339887) -- (0.9510565162,2.9270509831) --
(1.5388417685,2.1180339887)-- (0.9510565162,1.3090169943); 

\draw (0.9510565162,0.3090169943)  -- (0.9510565162,1.3090169943) -- (1.5388417685,2.1180339887) --
(1.5388417685,1.1180339887)-- (0.9510565162,0.3090169943); 

\draw   (-0.5877852522,1.8090169943)  -- (-1.5388417685,2.1180339887)  -- node[font=\scriptsize, above left] {4}  (-0.9510565162,2.9270509831) --
(0.0000000000,2.6180339887)-- (-0.5877852522,1.8090169943); 

\draw (0.0000000000,0.0000000000)  -- (0.0000000000,1.0000000000) -- (0.9510565162,1.3090169943) --
(0.9510565162,0.3090169943)-- (0.0000000000,0.0000000000); 

\draw [thick,red,rounded corners=8pt]  (-0.4755282581,0.6545084971) -- (-0.7694208842,1.5590169943) -- (-1.2449491424,1.2135254915);

\begin{scope}[shift={(4,0)}]
\draw [pattern color=red, pattern=horizontal lines] (0.0000000000,1.0000000000)  -- (-0.5877852522,1.8090169943)  -- (0.3632712640,2.1180339887) --
(0.9510565162,1.3090169943)-- (0.0000000000,1.0000000000); 

\draw [pattern color=red, pattern=horizontal lines] (0.0000000000,1.0000000000)  -- (-0.9510565162,1.3090169943) -- (-1.5388417685,2.1180339887) --
(-0.5877852522,1.8090169943)-- (0.0000000000,1.0000000000); 

\draw [pattern color=red, pattern=horizontal lines] (-0.9510565162,0.3090169943)  -- node[font=\scriptsize, below left] {2} (-1.5388417685,1.1180339887)  -- node[font=\scriptsize, left] {3} (-1.5388417685,2.1180339887) --
(-0.9510565162,1.3090169943)-- (-0.9510565162,0.3090169943); 

\draw [pattern color=red, pattern=horizontal lines] (-0.5877852522,1.8090169943)  -- (0.0000000000,2.6180339887) -- (0.9510565162,2.9270509831) --
(0.3632712640,2.1180339887)-- (-0.5877852522,1.8090169943); 

\draw [pattern color=red, pattern=horizontal lines] (0.0000000000,0.0000000000) node[below] {$\mathcal{W}_3$} --  node[font=\scriptsize, below] {1} (-0.9510565162,0.3090169943) -- (-0.9510565162,1.3090169943) --
(0.0000000000,1.0000000000)-- (0.0000000000,0.0000000000); 

\draw (0.0000000000,2.6180339887)  -- (-0.9510565162,2.9270509831) -- node[font=\scriptsize, above] {5}  (0.0000000000,3.2360679774) --
(0.9510565162,2.9270509831)-- (0.0000000000,2.6180339887); 

\draw [pattern color=red, pattern=horizontal lines] (0.9510565162,1.3090169943)  -- (0.3632712640,2.1180339887) -- (0.9510565162,2.9270509831) --
(1.5388417685,2.1180339887)-- (0.9510565162,1.3090169943); 

\draw [pattern color=red, pattern=horizontal lines] (0.9510565162,0.3090169943)  -- (0.9510565162,1.3090169943) -- (1.5388417685,2.1180339887) --
(1.5388417685,1.1180339887)-- (0.9510565162,0.3090169943); 

\draw  [pattern color=red, pattern=horizontal lines] (-0.5877852522,1.8090169943)  -- (-1.5388417685,2.1180339887)  -- node[font=\scriptsize, above left] {4}  (-0.9510565162,2.9270509831) --
(0.0000000000,2.6180339887)-- (-0.5877852522,1.8090169943); 

\draw [pattern color=red, pattern=horizontal lines] (0.0000000000,0.0000000000)  -- (0.0000000000,1.0000000000) -- (0.9510565162,1.3090169943) --
(0.9510565162,0.3090169943)-- (0.0000000000,0.0000000000); 

\draw [thick,red,rounded corners=8pt] (-1.2449491424,1.2135254915) -- (-0.4755282581,0.6545084971) --
(0.4755282581,0.6545084971) -- (1.2449491424,1.2135254915) -- (0.9510565162,2.1180339887) --
(0.1816356320,2.3680339887) -- (-0.7694208842,2.3680339887);
\end{scope}
\end{tikzpicture}
\end{center}

\end{ex}

\subsection{Degeneration of combs}\label{mutseqsection}
In the following we explain an inductive way to transform transform $\mathcal{W}_a$ into a singleton by successive flips and how this reflects on the associated posets of $a$-crossings.

\begin{lem}\label{hexexists}
If 	$\#\mathcal{W}_a>1$ then there exists a hexagon $\h\subset \mathcal{W}_a$ with $\h =\{[a,s], [a,t], [s,t]\}$.
\end{lem}
\begin{proof}
For $s,t\in[n] - \{a\}$ with $s\neq t$ we denote by $\Delta(s,t)$ the subset of tiles of $\mathcal{T}$ cut out by the strip sequences $\li^s$, $\li^t$ and $\li^a$.  We consider the set $\mathfrak{S}$ of all such $\Delta(s,t)\subset\mathcal{W}_a$ such that $[s,t]$ is a neighbour of $[a,s]$. The set $\mathfrak{S}$ is not empty, since $\Delta(a+1,t)\in\mathfrak{S}$, where $[a+1,t]$ is the neighbour of $[a,a+1]$ in $\mathcal{W}_a$. 

If $\Delta(s,t) \in \mathfrak{S}$ does not satisfy $\Delta(s,t)=\{[a,s], [a,t], [s,t]\}$ then  there exists $s'\in[n]-\{a,s,t\}$ with $\Delta(s,t) \cap\li^{s'}\neq \emptyset$. Since $\Delta(s,t) \cap\li^{s}=\{[a,s],[s,t]\}$ it follows that $[a,s']\in\Delta(s,t) $ and that $[a,s']$ posses a neighbour $[s',t']\in\Delta(s,t) $. Thus $\Delta(s,t)\supsetneq \Delta(s',t')\in\mathfrak{S}$. Proceeding inductively we obtain an inclusion minimal element $\h\in\mathfrak{S}$ with the desired property.
\end{proof}
\begin{ex} We depict the induction of the proof of Lemma \ref{hexexists} for the example of $\mathcal{W}_3$ given in Example \ref{ex:combs}.
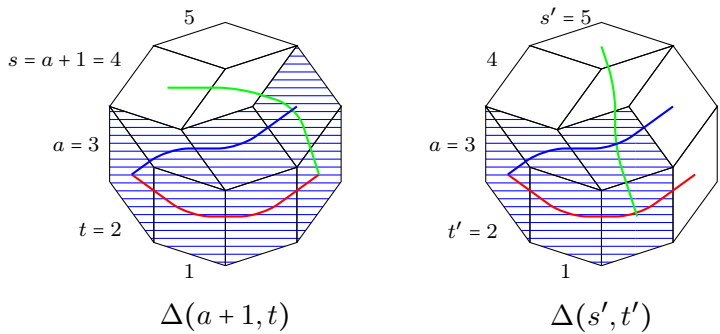
\[
\begin{tikzpicture}
\draw [pattern color=blue, pattern=horizontal lines] (0.0000000000,1.0000000000)  -- (-0.5877852522,1.8090169943)  -- (0.3632712640,2.1180339887) --
(0.9510565162,1.3090169943)-- (0.0000000000,1.0000000000); 

\draw [pattern color=blue, pattern=horizontal lines] (0.0000000000,1.0000000000)  -- (-0.9510565162,1.3090169943) -- (-1.5388417685,2.1180339887) --
(-0.5877852522,1.8090169943)-- (0.0000000000,1.0000000000); 

\draw [pattern color=blue, pattern=horizontal lines] (-0.9510565162,0.3090169943)    -- node[font=\scriptsize, below left] {$t=2$} (-1.5388417685,1.1180339887)  -- node[font=\scriptsize, left] {$a=3$} (-1.5388417685,2.1180339887) --
(-0.9510565162,1.3090169943) 
-- (-0.9510565162,0.3090169943); 

\draw (-0.5877852522,1.8090169943)  -- (0.0000000000,2.6180339887) -- (0.9510565162,2.9270509831) --
(0.3632712640,2.1180339887)-- (-0.5877852522,1.8090169943); 

\draw [pattern color=blue, pattern=horizontal lines] (0.0000000000,0.0000000000) node[below=10pt] {$\Delta(a+1,t)$} --  node[font=\scriptsize, below] {1} (-0.9510565162,0.3090169943) -- (-0.9510565162,1.3090169943) --
(0.0000000000,1.0000000000)-- (0.0000000000,0.0000000000); 

\draw (0.0000000000,2.6180339887)  -- (-0.9510565162,2.9270509831) -- node[font=\scriptsize, above] {5}  (0.0000000000,3.2360679774) --
(0.9510565162,2.9270509831)-- (0.0000000000,2.6180339887); 

\draw [pattern color=blue, pattern=horizontal lines] (0.9510565162,1.3090169943)   -- (0.3632712640,2.1180339887) 
 -- (0.9510565162,2.9270509831) --
(1.5388417685,2.1180339887)-- (0.9510565162,1.3090169943); 

\draw [pattern color=blue, pattern=horizontal lines] (0.9510565162,0.3090169943)  -- (0.9510565162,1.3090169943)
 -- (1.5388417685,2.1180339887) --
(1.5388417685,1.1180339887)-- (0.9510565162,0.3090169943); 

\draw (-0.5877852522,1.8090169943)  -- (-1.5388417685,2.1180339887)  -- node[font=\scriptsize, above left] {$s=a+1=4$}  (-0.9510565162,2.9270509831) --
(0.0000000000,2.6180339887)-- (-0.5877852522,1.8090169943); 

\draw [pattern color=blue, pattern=horizontal lines] (0.0000000000,0.0000000000)  -- (0.0000000000,1.0000000000) -- (0.9510565162,1.3090169943) --
(0.9510565162,0.3090169943)-- (0.0000000000,0.0000000000); 

\draw [thick,red,rounded corners=8pt] (-1.2449491424,1.2135254915) -- (-0.4755282581,0.6545084971) --
(0.4755282581,0.6545084971) -- (1.2449491424,1.2135254915);

\draw [thick,green,rounded corners=8pt] (1.2449491424,1.2135254915) -- (0.9510565162,2.1180339887) -- (0.1816356320,2.3680339887) -- (-0.7694208842,2.3680339887);
{\draw [thick,blue,rounded corners=8pt]  (-1.2449491424,1.2135254915) --  (-0.7694208842,1.5590169943) --
(0.1816356320,1.5590169943) --  (0.9510565162,2.1180339887);}
\begin{scope}[shift={(5,0)}]
\draw [pattern color=blue, pattern=horizontal lines] (0.0000000000,1.0000000000) 
-- (-0.5877852522,1.8090169943)  -- (0.3632712640,2.1180339887) --
(0.9510565162,1.3090169943)-- (0.0000000000,1.0000000000); 

\draw [pattern color=blue, pattern=horizontal lines] (0.0000000000,1.0000000000)  -- (-0.9510565162,1.3090169943) -- (-1.5388417685,2.1180339887) --
(-0.5877852522,1.8090169943)-- (0.0000000000,1.0000000000); 

\draw [pattern color=blue, pattern=horizontal lines] (-0.9510565162,0.3090169943)    -- node[font=\scriptsize, below left] {$t'=2$} (-1.5388417685,1.1180339887)  -- node[font=\scriptsize, left] {$a=3$} (-1.5388417685,2.1180339887) --
(-0.9510565162,1.3090169943) 
-- (-0.9510565162,0.3090169943); 

\draw (-0.5877852522,1.8090169943)  -- (0.0000000000,2.6180339887) -- (0.9510565162,2.9270509831) --
(0.3632712640,2.1180339887)-- (-0.5877852522,1.8090169943); 

\draw [pattern color=blue, pattern=horizontal lines] (0.0000000000,0.0000000000) node[below=10pt] {$\Delta(s',t')$} --  node[font=\scriptsize, below] {1} (-0.9510565162,0.3090169943) -- (-0.9510565162,1.3090169943) --
(0.0000000000,1.0000000000)-- (0.0000000000,0.0000000000); 

\draw (0.0000000000,2.6180339887)  -- (-0.9510565162,2.9270509831) -- node[font=\scriptsize, above] {$s'=5$}  (0.0000000000,3.2360679774) --
(0.9510565162,2.9270509831)-- (0.0000000000,2.6180339887); 

\draw (0.9510565162,1.3090169943)   -- (0.3632712640,2.1180339887)  -- (0.9510565162,2.9270509831) --
(1.5388417685,2.1180339887)-- (0.9510565162,1.3090169943); 

\draw (0.9510565162,0.3090169943)  -- (0.9510565162,1.3090169943) -- (1.5388417685,2.1180339887) --
(1.5388417685,1.1180339887)-- (0.9510565162,0.3090169943); 

\draw (-0.5877852522,1.8090169943)  -- (-1.5388417685,2.1180339887)  -- node[font=\scriptsize, above left] {$4$}  (-0.9510565162,2.9270509831) --
(0.0000000000,2.6180339887)-- (-0.5877852522,1.8090169943); 

\draw [pattern color=blue, pattern=horizontal lines] (0.0000000000,0.0000000000) 
 -- (0.0000000000,1.0000000000) -- (0.9510565162,1.3090169943) --
(0.9510565162,0.3090169943)-- (0.0000000000,0.0000000000); 

\draw [thick,red,rounded corners=8pt] (-1.2449491424,1.2135254915) -- (-0.4755282581,0.6545084971) --
(0.4755282581,0.6545084971) -- (1.2449491424,1.2135254915);
{\draw [thick,green,rounded corners=8pt]
(0.4755282581,0.6545084971) -- (0.1816356320,1.5590169943)
-- (0.1816356320,2.3680339887)  -- (0.0000000000,2.9180339887);;}

{\draw [thick,blue,rounded corners=8pt]  (-1.2449491424,1.2135254915) --  (-0.7694208842,1.5590169943) --
(0.1816356320,1.5590169943) --  (0.9510565162,2.1180339887);}
\end{scope}
-- (-0.5877852522,1.8090169943)  -- (0.3632712640,2.1180339887) --
\end{tikzpicture}
\]
The following induction step yields the desired hexagon $\Delta(s'',t'')$ with $s''=1$, $t''=2$.
$$
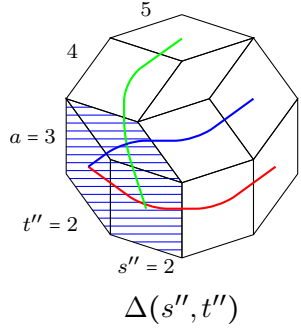
\begin{tikzpicture}
\draw (0.0000000000,1.0000000000)
-- (-0.5877852522,1.8090169943)  -- (0.3632712640,2.1180339887) --
(0.9510565162,1.3090169943)-- (0.0000000000,1.0000000000); 

\draw [pattern color=blue, pattern=horizontal lines] (0.0000000000,1.0000000000)  -- (-0.9510565162,1.3090169943) -- (-1.5388417685,2.1180339887) --
(-0.5877852522,1.8090169943)-- (0.0000000000,1.0000000000); 

\draw [pattern color=blue, pattern=horizontal lines] (-0.9510565162,0.3090169943)    -- node[font=\scriptsize, below left] {$t''=2$} (-1.5388417685,1.1180339887)  -- node[font=\scriptsize, left] {$a=3$} (-1.5388417685,2.1180339887) --
(-0.9510565162,1.3090169943) -- (-0.9510565162,0.3090169943); 

\draw (-0.5877852522,1.8090169943)  -- (0.0000000000,2.6180339887) -- (0.9510565162,2.9270509831) --
(0.3632712640,2.1180339887)-- (-0.5877852522,1.8090169943); 

\draw [pattern color=blue, pattern=horizontal lines] (0.0000000000,0.0000000000) node[below=10pt] {$\Delta(s'',t'')$} --  node[font=\scriptsize, below] {$s''=2$} (-0.9510565162,0.3090169943) -- (-0.9510565162,1.3090169943) --
(0.0000000000,1.0000000000)-- (0.0000000000,0.0000000000); 

\draw (0.0000000000,2.6180339887)  -- (-0.9510565162,2.9270509831) -- node[font=\scriptsize, above] {5}  (0.0000000000,3.2360679774) --
(0.9510565162,2.9270509831)-- (0.0000000000,2.6180339887); 

\draw (0.9510565162,1.3090169943)   -- (0.3632712640,2.1180339887)  -- (0.9510565162,2.9270509831) --
(1.5388417685,2.1180339887)-- (0.9510565162,1.3090169943); 

\draw (0.9510565162,0.3090169943)  -- (0.9510565162,1.3090169943) -- (1.5388417685,2.1180339887) --
(1.5388417685,1.1180339887)-- (0.9510565162,0.3090169943); 

\draw (-0.5877852522,1.8090169943)  -- (-1.5388417685,2.1180339887)  -- node[font=\scriptsize, above left] {$4$}  (-0.9510565162,2.9270509831) --
(0.0000000000,2.6180339887)-- (-0.5877852522,1.8090169943); 

\draw (0.0000000000,0.0000000000)
-- (0.0000000000,1.0000000000) -- (0.9510565162,1.3090169943) --
(0.9510565162,0.3090169943)-- (0.0000000000,0.0000000000); 

\draw [thick,red,rounded corners=8pt] (-1.2449491424,1.2135254915) -- (-0.4755282581,0.6545084971) --
(0.4755282581,0.6545084971) -- (1.2449491424,1.2135254915);
{\draw [thick,green,rounded corners=8pt]  (-0.4755282581,0.6545084971) --  (-0.7694208842,1.5590169943) -- (-0.7694208842,2.3680339887) -- (0.0000000000,2.9180339887);}
{\draw [thick,blue,rounded corners=8pt]  (-1.2449491424,1.2135254915) --  (-0.7694208842,1.5590169943) --
(0.1816356320,1.5590169943) --  (0.9510565162,2.1180339887);}
\end{tikzpicture}
$$
\end{ex}

Let $\h=\{[a,s], [a,t], [s,t]\} \subset\mathcal{W}_a$ be a hexagon as in Lemma \ref{hexexists} and let $\mathcal{S}$ be the tiling obtained from $\mathcal{T}$ by flipping $\h$. Using Definition \ref{s-order} we assume that $s,t$ are chosen such that $[a,s]_\mathcal{T} <_a [a,t]_\mathcal{T}$ (i.e. $s<a<t$ or $a<t<s$ or $t<s<a$). Then
$
\mathcal{V}_a:=\cl{\left(a,a+1\right)}\subset\mathcal{S}
$
satisfies \begin{equation}\label{eq:combs}
\mathcal{V}_a= \begin{cases} \left\{ [u_1, u_2]_\mathcal{S} \mid [u_1, u_2]_\mathcal{T} \in \mathcal{W}_a \right\} - \left\{[s,t]_\mathcal{S} \right\} & \text{if $t\neq a+1,$} \\
\left\{[u_1, u_2]_\mathcal{S} \mid [u_1, u_2]_\mathcal{T} \in \mathcal{W}_a \right\} - \left\{[s,t]_\mathcal{S}, [a,s]_\mathcal{S} \right\} & \text{if $t=a+1$}.\end{cases}
\end{equation}
In particular $\#\mathcal{V}_a < \#\mathcal{W}_a$.

We denote the set of $a$-crossing in $\mathcal{S}$ by $\Lambda_a$ and define the map
\begin{equation}\label{combext}
\begin{split}
\pi&=\pi_\h : \Gamma_a \twoheadrightarrow \Lambda_a
\\
\pi_\h  \left(u_1,  \dots, u_m \right) &:=\left\{ \begin{array}{ll} \left(u_1, u_3, u_4, \dots, u_m \right)&\text{if $u_2, u_3\in\{s,t\}$,} \\ \left(u_1,\dots, u_m \right) &\text{else.} \end{array} \right.
\end{split}
\end{equation}
Furthermore, we introduce the following subsets of crossings in $\Gamma_a$.
\begin{defi}\label{quots}
The subset $\Gamma^-\subset \Gamma_a$, resp. $\Gamma^+\subset \Gamma_a$, is the set of crossings whose first turn is at a tile $[a,u_2]$ which appears before $[a,s]$ in $\li^a$, resp. whose first turn is at a tile $[a,u_2]$ which appears after $[a,t]$ in $\li^a$, i.e.
$$\Gamma^- := \left\{\left(u_i \right)_{1\leq i \leq m} \in \Gamma_a \,\middle|\, \left[u_1, u_2\right]_\mathcal{T} <_a \left[a, s \right]_\mathcal{T}
\right\} \quad \Gamma^+ := \left\{\left(u_i \right)_{1\leq i \leq m} \in \Gamma_a \,\middle|\, \left[u_1, u_2\right]_\mathcal{T} >_a \left[a, t \right]_\mathcal{T}
\right\}.$$
The subset $\Gamma^{st}\subset \Gamma_a$, resp. $\Gamma^{s}\subset \Gamma_a$, is the set of crossings whose first turn is at the tile $[a,s]$ and whose second turn is at the tile $[s,t]$, resp. whose second turn is at a tile $[s,t']>_{a} [s,t]$ on $\li^s$, i.e.
$$ \Gamma^{st} := \left\{\left(u_i \right)_{1\leq i \leq m} \in \Gamma_a \,\middle|\, u_2=s, \, u_3= t\right\},\quad \Gamma^s := \left\{\left(u_i \right)_{1\leq i \leq m} \in \Gamma_a \,\middle|\, u_2=s, \, u_3\neq t\right\}.
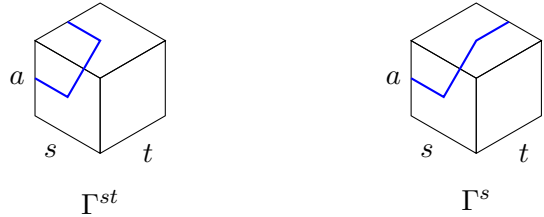
$$
$$\begin{tikzpicture}
\draw (0.0000000000,0.0000000000) node[below=10pt]{$\Gamma^{st}$}  -- node[below left]{$s$} (-0.8660254037,0.4999999999) -- node[left]{$a$} (-0.8660254037,1.5000000000) --
(0.0000000000,1.0000000000)-- (0.0000000000,0.0000000000);
\draw (0.0000000000,1.0000000000)  -- (-0.8660254037,1.5000000000) -- (0.0000000000,2.0000000000) --
(0.8660254037,1.5000000000)-- (0.0000000000,1.0000000000);
\draw (0.0000000000,0.0000000000)  -- (0.0000000000,1.0000000000) -- (0.8660254037,1.5000000000) --
(0.8660254037,0.4999999999)-- node[below right]{$t$} (0.0000000000,0.0000000000);
\draw [thick,blue](-0.8660254037,1) -- (-0.4330127018,0.7500000000) -- (0.0000000000,1.5000000000) -- (-0.4330127011, 1.75);
\begin{scope}[shift={(5,0)}]
\draw (0.0000000000,0.0000000000) node[below=10pt]{$\Gamma^{s}$}  -- node[below left]{$s$} (-0.8660254037,0.4999999999) -- node[left]{$a$} (-0.8660254037,1.5000000000) --
(0.0000000000,1.0000000000)-- (0.0000000000,0.0000000000);
\draw (0.0000000000,1.0000000000)  -- (-0.8660254037,1.5000000000) -- (0.0000000000,2.0000000000) --
(0.8660254037,1.5000000000)-- (0.0000000000,1.0000000000);
\draw (0.0000000000,0.0000000000)  -- (0.0000000000,1.0000000000) -- (0.8660254037,1.5000000000) --
(0.8660254037,0.4999999999)-- node[below right]{$t$} (0.0000000000,0.0000000000);
\draw [thick,blue](-0.8660254037,1) -- (-0.4330127018,0.7500000000) -- (0.0000000000,1.5000000000) -- (0.4330127011, 1.75);
\end{scope}
\end{tikzpicture}$$
The subset $\Gamma^{ts}\subset \Gamma_a$, resp. $\Gamma^{t}\subset \Gamma_a$, is the set of crossings whose first turn is at the tile $[a,t]$ and whose second turn is at the tile $[s,t]$, resp. whose second turn is at a tile $[s',t]>_{a} [s,t]$ on $\li^s$, i.e.
$$ \Gamma^{ts} := \left\{\left(u_i \right)_{1\leq i \leq m} \in \Gamma_a \,\middle|\, u_2=t, \, u_3= s\right\},\quad \Gamma^t := \left\{\left(u_i \right)_{1\leq i \leq m} \in \Gamma_a \,\middle|\, u_2=t, \, u_3\neq s\right\}.
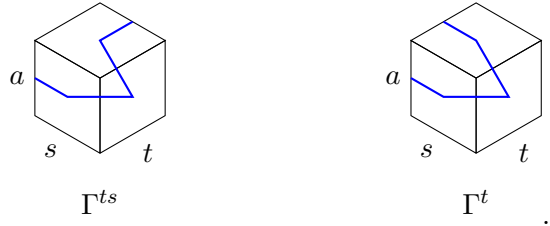
$$
$$\begin{tikzpicture}
\draw (0.0000000000,0.0000000000) node[below=10pt]{$\Gamma^{ts}$}  -- node[below left]{$s$} (-0.8660254037,0.4999999999) -- node[left]{$a$} (-0.8660254037,1.5000000000) --
(0.0000000000,1.0000000000)-- (0.0000000000,0.0000000000);
\draw (0.0000000000,1.0000000000)  -- (-0.8660254037,1.5000000000) -- (0.0000000000,2.0000000000) --
(0.8660254037,1.5000000000)-- (0.0000000000,1.0000000000);
\draw (0.0000000000,0.0000000000)  -- (0.0000000000,1.0000000000) -- (0.8660254037,1.5000000000) --
(0.8660254037,0.4999999999)-- node[below right]{$t$} (0.0000000000,0.0000000000);
\draw [thick,blue](-0.8660254037,1) -- (-0.4330127018,0.7500000000) -- (0.4330127018,0.7500000000) --
(0.0000000000,1.5000000000) -- (0.4330127011, 1.75);
\begin{scope}[shift={(5,0)}]
\draw (0.0000000000,0.0000000000) node[below=10pt]{$\Gamma^{t}$}  -- node[below left]{$s$} (-0.8660254037,0.4999999999) -- node[left]{$a$} (-0.8660254037,1.5000000000) --
(0.0000000000,1.0000000000)-- (0.0000000000,0.0000000000);
\draw (0.0000000000,1.0000000000)  -- (-0.8660254037,1.5000000000) -- (0.0000000000,2.0000000000) --
(0.8660254037,1.5000000000)-- (0.0000000000,1.0000000000);
\draw (0.0000000000,0.0000000000)  -- (0.0000000000,1.0000000000) -- (0.8660254037,1.5000000000) --
(0.8660254037,0.4999999999)-- node[below right]{$t$} (0.0000000000,0.0000000000);
\draw [thick,blue](-0.8660254037,1) -- (-0.4330127018,0.7500000000) -- (0.4330127018,0.7500000000) --
(0.0000000000,1.5000000000) -- (-0.4330127011, 1.75) ;
\end{scope}
\end{tikzpicture}.$$
Analogously we define the subsets $\Lambda^-$ and $\Lambda^+$ of $\Lambda_a$ by
$$\Lambda^- := \left\{\left(u_i \right)_{1\leq i \leq m} \in \Lambda_a \,\middle|\, \left[u_1, u_2\right]_\mathcal{S} <_a \left[a, t \right]_\mathcal{S}
\right\}, \quad \Lambda^+ := \left\{\left(u_i \right)_{1\leq i \leq m} \in \Lambda_a \,\middle|\, \left[u_1, u_2\right]_\mathcal{S} >_a \left[a, s \right]_\mathcal{S}
\right\}$$
and the subsets $\Lambda^s$ and $\Lambda^t$ of $\Lambda_a$ by
$$\Lambda^s := \left\{\left(u_i \right)_{1\leq i \leq m} \in \Lambda_a \,\middle|\, u_2=s\right\} \quad \Lambda^t := \left\{\left(u_i \right)_{1\leq i \leq m} \in \Lambda_a \,\middle|\, u_2=t\right\}.
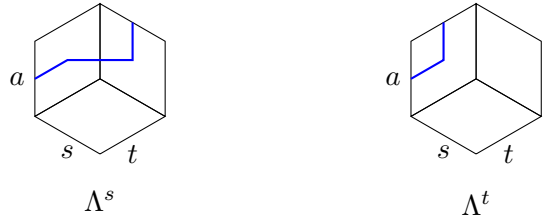
$$
$$\begin{tikzpicture}
\draw (0.0000000000,0.0000000000) node[below=10pt]{$\Lambda^s$}  -- node[below]{$s$} (-0.8660254037,0.4999999999) -- (0.0000000000,0.9999999999) --
(0.8660254037,0.4999999999)-- node[below]{$t$} (0.0000000000,0.0000000000);
\draw (0.8660254037,0.4999999999)  --  (0.0000000000,0.9999999999) --  (0.0000000000,1.9999999999) --
(0.8660254037,1.5000000000)-- (0.8660254037,0.4999999999);
\draw (-0.8660254037,0.4999999999)  -- node[left]{$a$} (-0.8660254037,1.5000000000) -- (0.0000000000,2.0000000000) --
(0.0000000000,0.9999999999)-- (-0.8660254037,0.4999999999);
\draw [thick,blue] (-0.8660254037,1) -- (-0.4330127018,1.2500000000) -- (0.4330127018,1.2500000000) -- (0.4330127018,1.7500000000);
\begin{scope}[shift={(5,0)}]
\draw (0.0000000000,0.0000000000) node[below=10pt]{$\Lambda^t$}  -- node[below]{$s$} (-0.8660254037,0.4999999999) -- (0.0000000000,0.9999999999) --
(0.8660254037,0.4999999999)-- node[below]{$t$} (0.0000000000,0.0000000000);
\draw (0.8660254037,0.4999999999)  -- (0.0000000000,0.9999999999) -- (0.0000000000,1.9999999999) --
(0.8660254037,1.5000000000)-- (0.8660254037,0.4999999999);
\draw (-0.8660254037,0.4999999999)  -- node[left]{$a$} (-0.8660254037,1.5000000000) -- (0.0000000000,2.0000000000) --
(0.0000000000,0.9999999999)-- (-0.8660254037,0.4999999999);
\draw [thick,blue] (-0.8660254037,1) -- (-0.4330127018,1.2500000000) -- (-0.4330127018,1.75);
\end{scope}
\end{tikzpicture}.$$
We further define
$$X:=\left\{\Gamma^-, \Gamma^s, \Gamma^{st}, \Gamma^t, \Gamma^{ts}, \Gamma^+ \right\} - \{\emptyset\}, \quad Y:=\left\{\Lambda^-, \Lambda^s, \Lambda^t,  \Lambda^+ \right\} -\{\emptyset\}.$$
\end{defi}
The quotient map $p_X : \Gamma_a \twoheadrightarrow X$ induces on $X$ the partial order $\leq_X$ given by
\begin{equation}\label{cartdiag}
\begin{matrix}
& &\Gamma^s \ \ \ \ \ {}_{\rotatebox{300}{\text{$<$}}} & & \\
\Gamma^-  < & {\Gamma^{st}} \ \ \ \ \ {}^{\rotatebox{60}{\text{$<$}}}_{\rotatebox{300}{\text{$<$}}} & &  \Gamma^{ts} & < \Gamma^+ \\
&  & {\Gamma^t} \ \ \ \ \ {}^{\rotatebox{60}{\text{$<$}}}. & &
\end{matrix}
\end{equation}
The quotient map $p_Y: \Lambda_a \twoheadrightarrow Y$ induces on $Y$ the order $\leq_Y$ given by
$$
\Lambda^- < \Lambda^t < \Lambda^s < \Lambda^+.
$$
The map
$
f: X \twoheadrightarrow Y,
$
defined by $f(\Gamma^-)=\Lambda^-$, $f(\Gamma^{st}) = f(\Gamma^t)=\Lambda^t$, $f(\Gamma^{ts}) = f(\Gamma^s)=\Lambda^s$, is order preserving and
\begin{equation}
\begin{split}\label{cart}
\begin{CD}
\Gamma_a @>{\pi_\mathcal{H}}>> \Lambda_a\\
@V{p_X}VV @V{p_Y}VV \\
X @>{f}>> Y,
\end{CD}
\end{split}
\end{equation}
is a Cartesian diagram of sets, i.e. $\Gamma_a=\Lambda_a \times_{Y} X$.
We show: 
\begin{prop}\label{orderpres}
The diagram \eqref{cart} is a Cartesian diagram of distributive lattices.
\end{prop}
\begin{proof}
We first show that \eqref{cart} is a Cartesian diagram of posets, i.e. 
for $\gamma_1=(\lambda_1,\xi_1), \gamma_2=(\lambda_2, \xi_2)\in \Gamma_a=\Lambda_a \times_{Y} X$ we have
\begin{equation}\label{cartpos}\gamma_1\preceq \gamma_2 \iff (\lambda_1 \preceq \lambda_2) \wedge (\xi_1 \leq \xi_2).
\end{equation}
Using the definition of $\preceq$ given in \eqref{orderdef} we write
\begin{align*}
\gamma_1\preceq \gamma_2 & \Leftrightarrow \underbrace{ \cl{\gamma}_1 \cap \h \subseteq \cl{\gamma}_2\cap \h}_{(i)}
\wedge \underbrace{ \cl{\gamma}_1 - \h \subseteq \cl{\gamma}_2-\h}_{(ii)} \wedge
\underbrace{ \op{\gamma}_1 \cap \h \subseteq \op{\gamma}_2\cap \h}_{(iii)} 
\wedge \underbrace{ \op{\gamma}_1 - \h \subseteq \op{\gamma}_2-\h}_{(iv)},\\\lambda_1\preceq \lambda_2 & \Leftrightarrow \underbrace{ \cl{\lambda}_1 \cap \h \subseteq \cl{\lambda}_2\cap \h}_{(i')}
\wedge \underbrace{ \cl{\lambda}_1 - \h \subseteq \cl{\lambda}_2-\h}_{(ii')}\wedge
\underbrace{ \op{\lambda}_1 \cap \h \subseteq \op{\lambda}_2\cap \h}_{(iii')}
\wedge \underbrace{ \op{\lambda}_1 - \h \subseteq \op{\lambda}_2-\h}_{(iv')}.
\end{align*}
Since $\lambda_1 = \pi \gamma_1$ and $\lambda_2=\pi\gamma_2$ we have $(ii)\Leftrightarrow (ii')$ and $(iv) \Leftrightarrow (iv')$.

For an arbitrary $\gamma\in\Gamma_a$ we have
\begin{align*}
\op{\gamma} \cap \h  &= \begin{cases} \left\{ [s,t]\right\}  &\text{if $\gamma\in\Gamma^+$,}\\\emptyset &\text{else,}\end{cases}\\
\cl{\gamma} \cap \h  &= \begin{cases} \emptyset &\text{if $\gamma\in\Gamma^-$,}\\\left\{[a,s],[s,t]\right\} &\text{if $\gamma\in\Gamma^{st} \cup \Gamma^{t}$,}\\\h &\text{else.}\end{cases}
\end{align*}
Thus, $\xi_1\leq \xi_2$ implies  $(i)$ and $(iii)$.

Condition $(iii')$ is empty, since for any $\lambda\in\Lambda_a$ one has $\op{\lambda}\cap\h=\emptyset.$ To establish \eqref{cartpos} it remains to show 
$
\gamma_1 \preceq \gamma_2 \Rightarrow (i').
$
This holds since $f\circ p_X$ is order preserving and for any $\gamma\in\Gamma_a$: 
$$
\cl{\pi \gamma} \cap \h = \begin{cases} \emptyset & f\circ p_X (\gamma) = \Lambda^-,\\ \left\{[a,t]
\right\} & f\circ p_X (\gamma) = \Lambda^t,
\\ \left\{[a,t], [a,s] \right\}& \text{else.}
\end{cases}
$$

One can show that $\Gamma_a$ and $\Lambda_a$ are lattices by explicitly constructing suprema and infima. We give a different proof using induction on $\#\mathcal{W}_a$.

We order $\li^a$ via 
$
\li^a_i \leq \li^a_j :\Leftrightarrow i\leq j
$
and define maps
\begin{align}\notag
\overline{p}_X &: \Gamma_a \rightarrow \li^a, \quad \left(a=u_1, \dots, u_m\right)\mapsto \left[a, u_2\right],\\\notag
\overline{p}_Y &: \Lambda_a \rightarrow \li^a, \quad \left(a=u_1, \dots, u_m\right)\mapsto \left[a, u_2\right].
\end{align}
If $\overline{p}_Y$ is a morphism of distributive lattices, then so is $p_Y$. This in turn implies using that $f$ is a morphism of distributive lattices that the Cartesian diagram of posets \eqref{cart} is a Cartesian diagram of distributive lattices. From this we obtain that $\overline{p}_X$ is a morphism of distributive lattices. 
Using induction we can assume by Lemma \ref{hexexists} and \eqref{eq:combs} that $\Lambda_a$ is a singleton and the claim follows.
\end{proof}

\subsection{Crossing Formula}\label{formula}
Let $a\in [n-1]$.
 To state the formula for the crystal structure $(f_a, e_a, \varepsilon_a)$ on $\mathbb{N}^{\mathcal{T}}$ we introduce the following notation.
For $\gamma\in\Gamma_a$ with strip sequence $(s_1=a, s_2, \dots, s_M=a+1)$ as in \eqref{ssdef} we define $\rvec{\gamma} \in\mathbb{Z}^{\mathcal{T}}$
\begin{align}
\rvecds{\gamma}_T &:= \begin{cases}\sgn \left(s_{i+1}-s_i \right)  & \text{if $T=[s_i, s_{i+1}]$ for some $i$,}\\ 0 & \text{else}\end{cases}
\label{Rdef}
\end{align}
 and $\drvec{\gamma}\in\Hom(\Z^{\mathcal{T}}, \Z)$ by
\begin{align}
\notag
\overline{\epsilon}_a \left( [s,t] \right) &:= \left\{ \begin{array}{cc} 1 & \text{if $s\leq a< a+1 \leq t$}, \\ -1 & \text{else} \end{array} \right. &&\text{for $[s,t]\in\mathcal{T}$,} \\\label{Fdef2}
\drvecxd{\gamma}{x}&= \sum_{\substack{T\in\gamma\\\overline{\epsilon}_a (T)=1  }} x_{T}- \sum_{\substack{T\in\gamma\\ \overline{\epsilon} (T)=-1 \\ \rvecs{\gamma}_T=0  }} x_{T} && \text{for $x\in\Z^{\mathcal{T}}$.}
\end{align}
\begin{defi}\label{Rcrossings}
We say $\gamma\in\Gamma_a$ is a \emph{Reineke $a$-crossing} if it satisfies the following condition: For any $\gamma_i=[s,t]$ such that $\gamma_{i-1}, \gamma_i$ and $\gamma_{i+1}$ lie on the same strip $\mathcal{L}^s$ we have
\begin{align*} s > t & \quad \text{if }t\le a, \\
s<t & \quad \text{if }t \geq a+1. \end{align*}
We denote the set of all Reineke $a$-crossings by $\mathcal{R}_a$.
\end{defi}

\begin{rem}\label{gprem} Using the relation between rhombic tilings and wiring diagrams (see \cite[Section 2]{DKK}) the notion of Reineke crossings translates into the notion of rigorous paths which appear in the work \cite{GP} of Gleizer and Postnikov in the description of string cone inequalities, see also Remark \ref{gpremark}.
\end{rem}
\begin{rem}\label{reinlat}
Using the notation of Section \ref{mutseqsection} we have that $\gamma\in\Gamma_a$ is a Reineke crossing precisely if $\pi(\gamma)\in\Lambda_a$ is a Reineke crossing and 
\begin{align*} \gamma\notin\Gamma^{t} & \quad \text{if }t\le a, \\
 \gamma\notin\Gamma^{s} & \quad \text{if }s\geq a+1. \end{align*}
 Thus, by induction we obtain from Proposition \ref{orderpres} that $\mathcal{R}_a \subseteq  \Gamma_a$ is a sublattice.
\end{rem}

\begin{ex} In the poset $\Gamma_3$ given in Example \ref{ex:poset} all $3$-crossings are Reineke $3$-crossings except $\gamma=(3,1,4)$ since we have $1 < 2 \le 3$.
$$ 
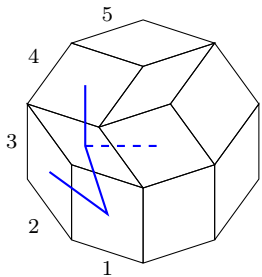
\begin{tikzpicture}
\draw (0.0000000000,1.0000000000)  -- (-0.5877852522,1.8090169943)  -- (0.3632712640,2.1180339887) --
(0.9510565162,1.3090169943)-- (0.0000000000,1.0000000000); 
\draw (0.0000000000,1.0000000000)  -- (-0.9510565162,1.3090169943) -- (-1.5388417685,2.1180339887) --
(-0.5877852522,1.8090169943)-- (0.0000000000,1.0000000000); 
\draw (-0.9510565162,0.3090169943)  -- node[font=\scriptsize, below left] {2} (-1.5388417685,1.1180339887)  -- node[font=\scriptsize, left] {3} (-1.5388417685,2.1180339887) --
(-0.9510565162,1.3090169943)-- (-0.9510565162,0.3090169943); 
\draw (-0.5877852522,1.8090169943)  -- (0.0000000000,2.6180339887) -- (0.9510565162,2.9270509831) --
(0.3632712640,2.1180339887)-- (-0.5877852522,1.8090169943); 
\draw (0.0000000000,0.0000000000)  --  node[font=\scriptsize, below] {1} (-0.9510565162,0.3090169943) -- (-0.9510565162,1.3090169943) --
(0.0000000000,1.0000000000)-- (0.0000000000,0.0000000000); 
\draw (0.0000000000,2.6180339887)  -- (-0.9510565162,2.9270509831) -- node[font=\scriptsize, above] {5}  (0.0000000000,3.2360679774) --
(0.9510565162,2.9270509831)-- (0.0000000000,2.6180339887); 
\draw (0.9510565162,1.3090169943)  -- (0.3632712640,2.1180339887) -- (0.9510565162,2.9270509831) --
(1.5388417685,2.1180339887)-- (0.9510565162,1.3090169943); 
\draw (0.9510565162,0.3090169943)  -- (0.9510565162,1.3090169943) -- (1.5388417685,2.1180339887) --
(1.5388417685,1.1180339887)-- (0.9510565162,0.3090169943); 
\draw (-0.5877852522,1.8090169943)  -- (-1.5388417685,2.1180339887)  -- node[font=\scriptsize, above left] {4}  (-0.9510565162,2.9270509831) --
(0.0000000000,2.6180339887)-- (-0.5877852522,1.8090169943); 
\draw (0.0000000000,0.0000000000)  -- (0.0000000000,1.0000000000) -- (0.9510565162,1.3090169943) --
(0.9510565162,0.3090169943)-- (0.0000000000,0.0000000000); 
{\draw [thick,blue] (-1.2449491424,1.2135254915)  -- (-0.4755282581,0.6545084971) --  (-0.7694208842,1.5590169943) -- (-0.7694208842,2.3680339887);}
{\draw [thick,dashed,blue]
(-0.7694208842,1.5590169943)-- (0.1816356320,1.5590169943);
;}
\end{tikzpicture}$$
\end{ex}
We prove the following Crossing Formula for the action of the Kashiwara operators on $\ii$-Lusztig data. 
\begin{thm}[Crossing Formula]\label{formula1}
Let $\ii$ be a reduced word and $\mathcal{T}$ the tiling associated to $\ii$. For $a\in [n-1]$ and an $\ii$-Lusztig datum $x\in \mathbb{N}^{\mathcal{T}}$ we have
\begin{equation*}
f_a x-x=\rvecd{\gamma^x},
\end{equation*}
where $\gamma^x$ is the $\preceq$-maximal element in $\Gamma_a$ satisfying
\begin{equation}\label{Fvcond}
\drvecxd{\gamma^x}{x} = \max_{\gamma\in\Gamma_a} \drvecx{\gamma}{x}.
\end{equation}
Furthermore, we have $\gamma^x\in\mathcal{R}_a$ and
$$
\varepsilon_a (x)  = \max_{\gamma\in\Gamma_a} \drvecx{\gamma}{x} = \max_{\gamma\in\mathcal{R}_a} \drvecx{\gamma}{x}.
$$
\end{thm}
\begin{proof}
We proceed by induction on $\#\mathcal{W}_a$. For $\#\mathcal{W}_a=1$ the statement follows by the definition of $f_a$ and $\varepsilon_a$ given in Definition \ref{def:crystal}. Assuming now $\#\mathcal{W}_a\geq 2$ by Lemma \ref{hexexists} there exists a hexagon $\h=\{[a,s], [a,t], [s,t]\}\subset \mathcal{W}_a$. We assume that $[a,s]<_a[a,t]$ and denote by $\mathcal{S}=\mathcal{T}_{\jj}$ the tiling obtained from by $\mathcal{T}$ by flipping $\h$. The reduced word $\jj$ is obtained from $\ii$ by the braid move corresponding to the flip of $\h$ (see Section \ref{sec:crystal}).
Further by (\ref{eq:combs}) $\mathcal{V}_a:=\cl{(a, a+1)}\subset\mathcal{S}$ satisfies $\#\mathcal{V}_a < \#\mathcal{W}_a$. 

We write $\Lambda_a$ for the set of $a$-crossings in $\mathcal{S}$ and define $\Gamma^-$, $\Gamma^{s}$, $\Gamma^{st}$, $\Gamma^{t}$, $\Gamma^{ts}$, $\Gamma^+$, $X$, $\Lambda^-$, $\Lambda^{s}$, $\Lambda^{t}$, $\Lambda^+$, $Y$ as in Definition \ref{quots}. Since \eqref{cartdiag} is Cartesian we can identify $\Gamma_a$ with $\Lambda_a \times_Y X$. 
Writing with the notations of \eqref{tlmutb}
$$
y=(y_S)_{S\in\mathcal{S}}:=\RR^{\ii}_{\jj} x
$$
by the induction hypothesis there exists a unique $\preceq$-maximal element
$\lambda^y\in\Lambda_a$ satisfying
\begin{equation*}
\varepsilon_a(y)=\drvecxd{\lambda^y}{y}=\max_{\lambda\in\Lambda_a} \drvecx{\lambda}{y}.
\end{equation*}
Let $\gamma^x\in \Gamma_a$ be a $\preceq$-maximal element satisfying \eqref{Fvcond}. Since a neighbour sequence is determined by its associated strip the vector $\rvec{\gamma}\in\mathbb{Z}^{\mathcal{T}}$ determines $\gamma\in\Gamma_a$.

By Definition \ref{def:crystal} we have $\varepsilon_a(y)=\varepsilon_a(x)$ and $f_a x=\RR^{\jj}_{\ii} f_a\RR^{\ii}_{\jj} x$. Thus using the induction hypothesis it is enough to show \begin{align}\label{ts1}
\RR^{\jj}_{\ii} \left(y+\rvecd{\lambda^y}\right)&=x+\rvecd{\gamma^x},\\\label{ts2}
\forall \lambda\in\Lambda_a \,:\, \max_{\gamma\in\pi^{-1} (\lambda)} \drvecxd{\gamma}{x} &= \drvecxd{\lambda}{y},\\\label{ts3}
\gamma^x&\in\mathcal{R}_a.
\end{align}
We consider the three cases $1) \ a<t<s$, $2) \ s<a<t$ and $3) \ t<s<a$ separately.

\underline{First case:} We first assume $a<t<s$. Then by (\ref{tlmutb}) we have for $\mm{s}{t}:=x_{[s,t]}$, $\nn{s}{t}:=y_{[s,t]}$ and
$\gamma\in \Gamma_a$
\begingroup\makeatletter\def\f@size{10.5}\check@mathfonts
\begin{equation}\label{fdel1}
\drvecxd{\gamma}{x} - \drvecxd{\pi\gamma}{y} =\begin{cases}0&\gamma\in\Gamma^-,\\
\mm{a}{s} - \nn{a}{t}=\min (0, \nn{s}{t}-\nn{a}{t})&\gamma\in\Gamma^{st},\\
\mm{a}{s}+\mm{a}{t}-\mm{s}{t}-\nn{a}{t} =  \min (0, \nn{a}{t}-\nn{s}{t})
&\gamma\in\Gamma^{t},\\
\mm{a}{s}-\mm{s}{t}-\nn{a}{t}-\nn{a}{s}=-\mm{s}{t}-\mm{a}{t}&\gamma\in\Gamma^{s},\\
\mm{a}{s}+\mm{a}{t}- \nn{a}{t}-\nn{a}{s}=0&\gamma\in\Gamma^{ts}\cup \Gamma^+.\end{cases}
\end{equation}
\endgroup
Since $\pi$ is surjective,
$\pi(\lambda, \Gamma^s)= \pi(\lambda, \Gamma^{ts})$  and $\pi(\lambda, \Gamma^t)= \pi(\lambda, \Gamma^{st})$
we obtain from \eqref{fdel1} the equality \eqref{ts2} and
\begin{equation}
\label{maxmax}
\drvecxd{\gamma^x}{x}=
\max_{\gamma\in\Gamma_a} \drvecxd{\gamma}{x} = \max_{\lambda\in\Lambda_a} \drvecxd{\lambda}{y} = \drvecxd{\lambda^y}{y}.
\end{equation}
Writing $\gamma^x=(\lambda^x, \Gam{x})\in \Lambda_a \times_Y X = \Gamma_a$ we obtain from \eqref{ts2} and \eqref{maxmax}
\begin{equation}
\label{pigMi}
\lambda^x  = \pi \gamma^x \preceq \lambda^y.
\end{equation}
By \eqref{ts2} and \eqref{maxmax} we can choose $\Gam{y}\in X$ maximal with $\drvecx{(\lambda^y, \Gam{y})}{x} = \drvecx{\gamma^x}{x}$. Then, by $\Gamma^{s} \prec \Gamma^{ts}$ and \eqref{fdel1}, we have $\Gam{y} \neq \Gamma^s$.
Since $X-\{\Gamma^s\}$ is linearly ordered we conclude from the $\preceq$-maximality of $\gamma^x$, \eqref{Fdef2}  and Proposition \ref{orderpres} that $\Gamma^x=\Gam{y}$. Thus, using Proposition \ref{orderpres} again and \eqref{pigMi} we obtain
$\gamma^x \preceq (\lambda^y, \Gam{y})$. By the $\preceq$-maximality of $\gamma^x$
we conclude $\gamma^x = (\lambda^y, \Gam{y})$. Using the induction hypothesis, $\gamma^x = (\lambda^y, \Gam{y})$ implies \eqref{ts3}. Furthermore, we obtain
\begin{equation}
\label{pigM}\pi \gamma^x=\lambda^y.\end{equation}

By (\ref{tlmutb}) we have
\begin{equation}
\label{fdel2}
\RR^{\jj}_{\ii} \left(y+ \rvecd{\lambda^y} \right) - x = \begin{cases}
\rvecd{\left(\lambda^y, \Gamma^{st}\right)} &\text{if $\lambda^y\in\Lambda^t  \wedge \nn{a}{t} < \nn{s}{t},$}\\
\rvecd{\left(\lambda^y, \Gamma^{t}\right)} &\text{if $\lambda^y\in\Lambda^t  \wedge \nn{a}{t} \geq \nn{s}{t},$}\\
\rvecd{\left(\lambda^y, \Gamma^{ts} \right)} &\text{if $\lambda^y\in \Lambda^s,$}\\
\rvecd{\left(\lambda^y, \Gamma^{-} \right)} &\text{if $\lambda^y\in \Lambda^-,$}\\
\rvecd{\left(\lambda^y, \Gamma^{+} \right)} &\text{if $\lambda^y\in \Lambda^+.$}\end{cases}
\end{equation}
If $\lambda^y\notin\Lambda^t$ then $\#\pi^{-1} (\lambda^y) \cap \mathcal{R}_a=1$ and \eqref{ts1} follows from \eqref{ts3}, \eqref{pigM} and \eqref{fdel2}.

If $\lambda^y\in\Lambda^t$, we have
$\pi^{-1} (\lambda^y) = \{(\lambda^y, \Gamma^{st}), (\lambda^y, \Gamma^t)
\}$. Furthermore, from Proposition \ref{orderpres} we obtain $(\lambda^y, \Gamma^{st}) \preceq (\lambda^y, \Gamma^t)$. Thus, \eqref{ts1} follows from \eqref{fdel1}, \eqref{pigM}, \eqref{fdel2} and \eqref{maxmax}.

\underline{Second Case:} We next consider the case $t<s<a$. By (\ref{tlmutb}), we have for $\mm{s}{t}:=x_{[s,t]}$, $\nn{s}{t}:=y_{[s,t]}$ and
$\gamma\in \Gamma_a$
\begin{equation}
\label{fdel1b}
\drvecxd{\gamma}{x} - \drvecxd{\pi\gamma}{y} =\begin{cases}0&\gamma\in\Gamma^- \cup \Gamma^{st},\\
-\mm{s}{t}-\mm{s}{a}&\gamma\in\Gamma^{t},\\
-\mm{s}{t}+\nn{t}{a} =  \min (0, \nn{s}{a}-\nn{s}{t})
&\gamma\in\Gamma^{s},\\
-\mm{s}{a}+\nn{t}{a} =  \min (0, \nn{s}{t}- \nn{s}{a})&\gamma\in\Gamma^{ts}\\
-\mm{s}{a} - \mm{t}{a} + \nn{t}{a} + \nn{s}{a}=0&\gamma\in\Gamma^+. \end{cases}
\end{equation}
Since $\pi$ is surjective, $\pi(\lambda, \Gamma^s)= \pi(\lambda, \Gamma^{ts})$ and $\pi(\lambda, \Gamma^t)= \pi(\lambda, \Gamma^{st})$
we obtain from \eqref{fdel1b} the equality \eqref{ts2} and
\begin{equation}
\label{maxmaxb}
\drvecxd{\gamma^x}{x}=
\max_{\gamma\in\Gamma_a} \drvecxd{\gamma}{x} = \max_{\lambda\in\Lambda_a} \drvecxd{\lambda}{y} = \drvecxd{\lambda^y}{y}.
\end{equation}
Writing $\gamma^x=(\lambda^x, \Gamma^x)\in \Lambda_a \times_Y X = \Gamma_a$ we obtain from \eqref{ts2} and \eqref{maxmaxb}
\begin{equation}
\label{pigMib}
\lambda^x  = \pi \gamma^x \preceq \lambda^y.
\end{equation}
By \eqref{ts2} and \eqref{maxmaxb} we can choose $\Gam{y}\in X$ maximal with $\drvecx{(\lambda^y, \Gam{y})}{x} = \drvecx{\gamma^x}{x}$. Then $\Gam{y} \neq \Gamma^t$, since
otherwise by \eqref{fdel1b} we have $y_{a,t}=\min(x_{a,s}, x_{s,t})=0$ and using the induction hypothesis we obtain the contradiction
$f_a y=y+\rvec{\lambda^y}\notin \mathbb{N}^{\mathcal{S}}.$
Since $X-\{\Gamma^t\}$ is linearly ordered we conclude from the $\preceq$-maximality of $\gamma^x$, \eqref{Fdef2}  and Proposition \ref{orderpres} that $\Gamma^x=\Gam{y}$. Thus, using Proposition \ref{orderpres} again and \eqref{pigMib} we obtain
$\gamma^x \preceq (\lambda^y, \Gam{y})$. By the $\preceq$-maximality of $\gamma^x$
we conclude $\gamma^x = (\lambda^y, \Gam{y})$. Using the induction hypothesis, $\gamma^x = (\lambda^y, \Gam{y})$ implies \eqref{ts3}. Furthermore, we obtain
\begin{equation}
\label{pigMb}\pi \gamma^x=\lambda^y.\end{equation}

By (\ref{tlmutb}) we have
\begin{equation}
\label{fdel2b}
\RR^{\jj}_{\ii} \left(y + \rvecd{\lambda^y} \right) - x = \begin{cases}
\rvecd{\left(\lambda^y, \Gamma^{ts}\right)} &\text{if $\lambda^y\in\Lambda^s  \wedge \nn{s}{a} \leq \nn{s}{t},$}\\
\rvecd{\left(\lambda^y, \Gamma^{s}\right)} &\text{if $\lambda^y\in\Lambda^s  \wedge \nn{s}{a} > \nn{s}{t},$}\\
\rvecd{\left(\lambda^y, \Gamma^{st} \right)} &\text{if $\lambda^y\in \Lambda^t,$}\\
\rvecd{\left(\lambda^y, \Gamma^{-} \right)} &\text{if $\lambda^y\in \Lambda^-,$}\\
\rvecd{\left(\lambda^y, \Gamma^{+} \right)} &\text{if $\lambda^y\in \Lambda^+.$}\end{cases}
\end{equation}

If $\lambda^y\notin\Lambda^s$ then $\#\pi^{-1} (\lambda^y) \cap \mathcal{R}_a=1$ and \eqref{ts1} follows from \eqref{ts3}, \eqref{pigMb} and \eqref{fdel2b}.

If $\lambda^y\in\Lambda^s$, we have
$\pi^{-1} (\lambda^y) = \{(\lambda^y, \Gamma^s), (\lambda^y, \Gamma^{ts})
\}$. Furthermore, from Proposition \ref{orderpres} we obtain $(\lambda^y, \Gamma^{s}) \preceq (\lambda^y, \Gamma^{ts})$.
Thus, \eqref{ts1} follows from \eqref{fdel1b}, \eqref{pigMb}, \eqref{fdel2b} and \eqref{maxmaxb}.

\underline{Third Case:} We finally consider the case $s<a<t$. In this case \eqref{ts3} follows directly from the induction hypothesis. Furthermore, by (\ref{tlmutb}) we have for $\mm{s}{t}:=x_{[s,t]}$, $\nn{s}{t}:=y_{[s,t]}$ and
$\gamma\in \Gamma_a$
\begin{equation}
\label{fdel1c}
\drvecxd{\gamma}{x} - \drvecxd{\pi\gamma}{y} =\begin{cases}0&\gamma\in\Gamma^-,\\
\mm{s}{t}-\nn{a}{t} =  \min \left(0, \nn{s}{a}-\nn{a}{t}\right)
 &\gamma\in \Gamma^{st} \cup \Gamma^s,\\
-\mm{s}{a}+\mm{a}{t} + \mm{s}{t} - \nn{a}{t} = \min \left(0, \nn{a}{t}-\nn{s}{a}\right) & \gamma\in\Gamma^{t} \cup \Gamma^{ts},\\
-\mm{s}{a} + \mm{a}{t} - \nn{a}{t} + \nn{s}{a}=0&\gamma\in\Gamma^+. \end{cases}
\end{equation}
Since $\pi$ is surjective, $\pi(\lambda, \Gamma^s)= \pi(\lambda, \Gamma^{ts})$  and $\pi(\lambda, \Gamma^t)= \pi(\lambda, \Gamma^{st})$
we obtain from \eqref{fdel1c} the equality \eqref{ts2} and
\begin{equation}
\label{maxmaxc}
\drvecxd{\gamma^x}{x}=
\max_{\gamma\in\Gamma_a} \drvecxd{\gamma}{x} = \max_{\lambda\in\Lambda_a} \drvecxd{\lambda}{y} = \drvecxd{\lambda^y}{y}.
\end{equation}
Writing $\gamma^x=(\lambda^x, \Gamma^x)\in \Lambda_a \times_Y X = \Gamma_a$ we obtain from \eqref{ts2} and \eqref{maxmaxc}
\begin{equation}
\label{pigMic}
\lambda^x  = \pi \gamma^x \preceq \lambda^y.
\end{equation}
By \eqref{ts2} and \eqref{maxmaxc} we can choose $\Gam{y} \in X$ maximal with $\drvecx{(\lambda^y, \Gam{y})}{x} = \drvecx{\gamma^x}{x}$. Then $\{\Gam{y}, \Gamma^x\} \neq \{\Gamma^s, \Gamma^t\}$, since
otherwise by \eqref{fdel1c} we obtain $\nn{s}{a}=\nn{a}{t}$ and the contradiction $\drvecx{(\lambda^y, \Gamma^{ts})}{x} = \drvecx{\gamma^x}{x}$ to $\Gamma^s, \Gamma^t < \Gamma^{st}$.
Since $X-\{\Gamma^s\}$ and $X-\{\Gamma^t\}$ are linearly ordered we conclude from the $\preceq$-maximality of $\gamma^x$, \eqref{Fdef2}  and Proposition \ref{orderpres} that $\Gamma^x=\Gam{y}$. 
Thus, using Proposition \ref{orderpres} and \eqref{pigMic} we obtain
$\gamma^x \preceq (\lambda^y, \Gam{y})$ and by the $\preceq$-maximality of $\gamma^x$
we conclude
\begin{equation}
\label{pigMc}\pi \gamma^x=\lambda^y.\end{equation}

By (\ref{tlmutb}) we have
\begin{equation}
\label{fdel2c}
\RR^{\jj}_{\ii} \left(y + \rvecd{\lambda^y} \right) - x = \begin{cases}
\rvecd{\left(\lambda^y, \Gamma^{ts}\right)} &\text{if $\lambda^y\in\Lambda^s  \wedge \nn{s}{a} \leq \nn{a}{t},$}\\
\rvecd{\left(\lambda^y, \Gamma^{s}\right)} &\text{if $\lambda^y\in\Lambda^s  \wedge \nn{s}{a} > \nn{s}{t},$}\\
\rvecd{\left(\lambda^y, \Gamma^{t} \right)} &\text{if $\lambda^y\in \Lambda^t \wedge \nn{s}{a} \leq \nn{a}{t},$}\\
\rvecd{\left(\lambda^y, \Gamma^{st} \right)} &\text{if $\lambda^y\in \Lambda^t \wedge \nn{s}{a} > \nn{a}{t},$}\\
\rvecd{\left(\lambda^y, \Gamma^{-} \right)} &\text{if $\lambda^y\in \Lambda^-,$}\\
\rvecd{\left(\lambda^y, \Gamma^{+} \right)} &\text{if $\lambda^y\in \Lambda^+.$}\end{cases}
\end{equation}
If $\lambda^y\notin\Lambda^s \cup \Lambda^t$ then $\#\pi^{-1} (\lambda^y) =1$ and \eqref{ts1} follows from \eqref{pigMc} and \eqref{fdel2c}.

If $\lambda^y\in\Lambda^s$, we have
$\pi^{-1} (\lambda^y) = \{(\lambda^y, \Gamma^s), (\lambda^y, \Gamma^{ts})
\}$. Furthermore, from Proposition \ref{orderpres} we obtain $(\lambda^y, \Gamma^{s}) \preceq (\lambda^y, \Gamma^{ts})$. Thus, 
\eqref{ts1} follows from \eqref{fdel1c}, \eqref{pigMc}, \eqref{fdel2c} and \eqref{maxmaxc}.

If $\lambda^y\in\Lambda^t$, we have
$\pi^{-1} (\lambda^y) = \{(\lambda^y, \Gamma^{st}), (\lambda^y, \Gamma^{t})
\}$. Furthermore, from Proposition \ref{orderpres} we obtain $(\lambda^y, \Gamma^{st}) \preceq (\lambda^y, \Gamma^{t})$. Thus, 
\eqref{ts1} follows from \eqref{fdel1c}, \eqref{pigMc}, \eqref{fdel2c} and \eqref{maxmaxc}.
\end{proof}

For $x\in\Z^{\mathcal{T}}$ we define $[x]^-:=-(\min(x_T,0))_{T\in\mathcal{T}}\in\N^{\mathcal{T}}$. Then  we have the following proposition.
\begin{prop}\label{prop:Reineke} For $\gamma\in\mathcal{R}_a$ and $x=[\rvec{\gamma}]^-$ we have
$$
\gamma^{x} = \gamma.
$$
\end{prop}
\begin{proof}
By the Crossing Formula (Theorem \ref{formula1}) the statement is equivalent to
\begin{equation}\label{allkash}
f_a[\rvec{\gamma}]^-= [\rvec{\gamma}]^- + \rvec{\gamma}.
\end{equation}
We prove \eqref{allkash} by induction on $\#\mathcal{W}_a$. For $\#\mathcal{W_a}=1$ the claim is true by the definition of $f_a$ given in Definition \ref{def:crystal}. Assuming now $\#\mathcal{W}_a\geq 2$ by Lemma \ref{hexexists} there exists a hexagon $\h=\{[a,s], [a,t], [s,t]\}\subset \mathcal{W}_a$. We assume that $[a,s]<_a[a,t]$ and denote by $\mathcal{S}=\mathcal{T}_{\jj}$ the tiling obtained from by $\mathcal{T}$ by flipping $\h$. 
The reduced word $\jj$ is obtained from $\ii$ by the braid move corresponding to the flip of $\h$ (see Section \ref{sec:crystal}).
Further by (\ref{eq:combs}) $\mathcal{V}_a:=\cl{(a, a+1)}\subset\mathcal{S}$ satisfies $\#\mathcal{V}_a < \#\mathcal{W}_a$. 

We write $\Lambda_a$ for the set of $a$-crossings in $\mathcal{S}$ and define $\Gamma^-$, $\Gamma^{s}$, $\Gamma^{st}$, $\Gamma^{t}$, $\Gamma^{ts}$, $\Gamma^+$, $X$, $\Lambda^-$, $\Lambda^{s}$, $\Lambda^{t}$, $\Lambda^+$, $Y$ as in Definition \ref{quots}.

We first assume that either $\gamma \in \Gamma^{s}\cup \Gamma^{t}\cup \Gamma^- \cup \Gamma^+$ or $t<s<a$ and $\gamma \in \Gamma^{st}$ or $a<t<s$ and $\gamma \in \Gamma^{ts}$. One verifies case by case that
\begin{align} \label{eq:1}
\RR^{\ii}_{\jj} \left(\left[\rvec{\gamma}\right]^-\right)& = \left[\rvec{\pi\gamma}\right]^-, \\ \label{eq:2}
\RR^{\jj}_{\ii} \left(\left[\rvec{\pi\gamma}\right]^- + \rvec{\pi\gamma}\right) &= \left[\rvec{\gamma}\right]^- + \rvec{\gamma}.
\end{align}
Since by definition $f_a= \RR^{\jj}_{\ii} f_a \RR^{\ii}_{\jj}$ we conclude
$$f_a \left[\rvec{\gamma} \right]^-=\RR^{\jj}_{\ii}\left(f_a\left[\rvec{\pi\gamma}\right]^-\right)=\RR^{\jj}_{\ii}\left(\left[\rvec{\pi\gamma}\right]^- + \rvec{\pi\gamma}\right) = \rvec{\gamma}+\left[\rvec{\gamma}\right]^-,$$
where the first equality follows from (\ref{eq:1}), the second equality follows from the induction hypothesis and the third equality from (\ref{eq:2}).

We now consider the remaining cases and fix $\gamma \in \widetilde{\Gamma}$ where
$$\widetilde{\Gamma}=\begin{cases} \Gamma^{st} & \text{ if }a<t<s \ \vee \ s<a<t, \\
\Gamma^{ts} & \text{ if } t<s<a \ \vee \ s<a<t.
\end{cases}$$
In these cases we have for all $\lambda \in \Lambda_a$
$$ \drvecx{\lambda}{\RR^{\ii}_{\jj}\left[\rvec{\gamma}\right]^-} \le
\begin{cases}
\drvecxd{\lambda}{\left[\rvec{\pi\gamma}\right]^-}+1 & \text{ if }s<a<t \ \vee \gamma \in \Gamma^{ts}, \\
\drvecxd{\lambda}{\left[\rvec{\pi\gamma}\right]^-} & \text{else,}
\end{cases}$$
with equality for $\lambda=\pi \gamma$. Thus, using the induction hypothesis and the Crossing Formula we obtain
\begin{equation}\label{kashremcases}
f_a \RR^{\ii}_{\jj}\left[\rvec{\gamma}\right]^-  = \RR^{\ii}_{\jj}\left[\rvec{\gamma}\right]^-  + \rvec{\pi\gamma}.
\end{equation}
Using \eqref{kashremcases} and the definition of $f_a$ one directly computes
$$f_a [\rvec{\gamma}]^-= \RR^{\jj}_{\ii}\left(f_a \RR^{\ii}_{\jj} [\rvec{\gamma}]^-\right)
=\RR^{\jj}_{\ii}\left(\RR^{\ii}_{\jj}
\left[\rvec{\gamma}\right]^- + \rvec{\pi\gamma}\right)= \left[\rvec{\gamma}\right]^-+\rvec{\gamma}.$$
\end{proof}
By the Crossing Formula (Theorem \ref{formula1}) and Proposition \ref{prop:Reineke} we obtain the following.
\begin{thm}\label{thm:Reinekevectors}
We have
$$\mathbf{R}_a=\mathbf{R}_a(\mathcal{T}):=\{f_a x -x \mid x \in \mathbb{N}^{\mathcal{T}} \}=\{\rvec{\gamma} \mid \gamma \in \mathcal{R}_a\}.$$
\end{thm}
\begin{defi}\label{def:reineke}
The subset $\mathbf{R}_a \subset \mathbb{Z}^{\mathcal{T}}$ is called the set of \emph{$a$-Reineke vectors}.
\end{defi}

\begin{rem}\label{zelrem} For reduced words $\ii$ adapted to quivers of type $A$, the set of vectors of the form $\{f_a x -x \mid x \in \mathbb{N}^{\mathcal{T}_{\ii}} \}$ was already studied in the work of Zelikson (\cite{Ze}) in the setup of pseudoline arrangements. Here the term Lusztig moves was used and the set of such moves was shown to give defining inequalities of a certain string cone (see Remark \ref{zelremark}).
\end{rem}
The following result can be proved analogously to the Crossing Formula.
\begin{thm}\label{formula2}
	For $a\in[n-1]$ and $x\in \N^{\mathcal{T}}$ with $\varepsilon_a(x)\geq 1$ we have
\begin{equation*}
e_a x+x=\rvec{\gamma_x},
\end{equation*} where $\gamma_x$ is the $\preceq$-minimal element in $\Gamma_a$ satisfying
	\begin{equation*}
	\drvecxd{\gamma_x}{x} = \max_{\gamma\in\Gamma_a} \drvecx{\gamma}{x}.
	\end{equation*}
\end{thm}

\subsection{Dual Crossing Formula and highest weight crystals}\label{dualcrossing1}

Let $\mathcal{T}_{\ii}$ be the tiling associated to the reduced word $\ii$ of $w_0$ as defined in Section \ref{wordtiling}. For $\lambda$ a dominant integral weight of $\sln(\mathbb{C})$, the crystal graph of highest weight crystal $B(\lambda)$ is a full subgraph of $B(\infty)$ (with weights appropriately shifted). In this section we give an explicit description of the ${}^*$-crystal structure on $\mathbb{N}^{\mathcal{T}_{\ii}}$ and the explicit embedding of the highest weight crystals into $\mathbb{N}^{\mathcal{T}_{\ii}}$ where ${}^*$ denotes the Kashiwara involution introduced in Section \ref{Kashiwarainvolution}.

We introduce the dual notions of those used in the Crossing Formula (Theorem \ref{formula1}).
Using the notation of Definition \ref{Gidef} we define the set ${\Gamma}_a^*$ of dual $a$-crossings as the set of $n+a$-ascending neighbour sequences $(\gamma_i)_{i\in [m]}\subset \mathcal{T}_{\ii}$ starting at $\gamma_1=\li^a_{n-1}$ and ending at $\gamma_m=\li^{a+1}_{n-1}$.

We introduce a relation $\preceq^*$ on $\Gamma_a^*$ as follows. Since $\leq_{n+a}$ defined in Definition \ref{s-order} is anti-symmetric, $\gamma\in \Gamma_a^*$ cannot contain cycles. Consequently, the set $\mathcal{T}\setminus \{T \mid T \in \gamma\}$ is partitioned into the set of tiles lying on the left of $\gamma$ and the set of tiles lying on the right of $\gamma$.
We denote the set consisting of those $T\in\mathcal{T}$ which do not lie left of $\gamma$ by$\cl{\gamma}^*$ and define for $\gamma, \lambda \in \Gamma_a^*$
\begin{align*}
\op{\gamma}^*&:=\cl{\gamma}^*-\gamma,\\
\gamma \preceq^* \lambda &:\Leftrightarrow
(\cl{\gamma}^*\subseteq \cl{\lambda}^*) \wedge
(\op{\gamma}^* \subseteq \op{\lambda}^*).
\end{align*}
The proof of Proposition \ref{isorder} shows that $\preceq^*$ is an order relation on $\Gamma^*$.

To $\gamma\in\Gamma_a^*$ with strip sequence $(s_1=a, s_2, \dots, s_M=a+1)$ (see \eqref{ssdef}) we  associate $\rvec{\gamma} \in\mathbb{Z}^{\mathcal{T}_{\ii}}$ and $\drvec{\gamma}\in\Hom(\Z^{\mathcal{T}_{\ii}}, \Z)$ by \eqref{Rdef} and \eqref{Fdef2} respectively.

\begin{defi}
We say $\gamma\in\Gamma_a^*$ is a \emph{dual Reineke $a$-crossing} if it satisfies the following condition: For any $\gamma_i=[ s,t]$ such that $\gamma_{i-1}, \gamma_i$ and $\gamma_{i+1}$ lie in the same strip sequence $\mathcal{L}^s$ we have
\begin{align*} s > t & \quad \text{if }t\le a \\
s<t & \quad \text{if }t \geq a+1. \end{align*}
We denote the set of all dual Reineke $a$-crossings by $\mathcal{R}_a^*.$
\end{defi}

We obtain the following dual Crossing Formula.
\begin{thm}[Dual Crossing Formula]\label{dualcrossing}
Let $\ii$ be a reduced word, $a\in [n-1]$ and $x\in \mathbb{N}^{\mathcal{T}_{\ii}}$. Then we have
\begin{align*}
f^*_a x-x&=\rvec{\gamma^x},\\
e^*_a x-x&=\begin{cases} {\rvec{\gamma_x}} & \text{if $\varepsilon_a^* (x) \geq 1$ ,}\\ 0&\text{else,} \end{cases}\\
\end{align*}
where $\gamma^x$ (resp. $\gamma_x$) is the $\preceq^*$-maximal (resp. minimal) element $\tilde{\gamma}$ in $\Gamma_a^*$ satisfying
\begin{equation*}
\drvecxd{\tilde{\gamma}}{x} = \max_{\gamma\in\Gamma^*_a} \drvecx{\gamma}{x}.
\end{equation*}
Furthermore, we have $\gamma^x, \gamma_x\in\mathcal{R}^*_a$ and
$$
\varepsilon^*_a (x)  = \max_{\gamma\in\Gamma^*_a} \drvecx{\gamma}{x} = \max_{\gamma\in\mathcal{R}^*_a} \drvecx{\gamma}{x}.
$$
Moreover we have
$$\{\rvec{\gamma} \mid \gamma \in \mathcal{R}^{*}_a\}=\{f^*_a x - x \mid x \in \mathbb{N}^{\mathcal{T}_{\ii}} \}.$$
\end{thm}

\begin{proof}
To $\ii=(i_1, \dots, i_N)$ we associate the reduced word $\ii^{\starop}=(i^{\starop}_1, \dots, i^{\starop}_N)$ with $i^{\starop}_{k}=n-i_{N+1-k}$. We denote by $\Gamma_a^*(\mathcal{T}_{\ii})$ the set of dual $a$-crossings in $\mathcal{T}_{\ii}$ and by $\Gamma_a(\mathcal{T}_{\ii^{\starop}})$ the set of $a$-crossings in $\mathcal{T}_{\ii^{\starop}}$.
We note that $\mathcal{T}_{\ii}$ is the unique tiling such that under the identification of tiles with positive roots as in \eqref{eq:roots} the ordering $\le_{\ii^{\starop}}$ is a refinement of $\le_{2n}$ to a total order.  As a consequence, the map 
$$
\Gamma_a \left( \mathcal{T}_{\ii^{\starop}}\right) \rightarrow \Gamma_a^* \left( \mathcal{T}_{\ii}\right), \quad \gamma=([k_1, \ell_1], \dots, [k_m, \ell_m]) \mapsto \gamma^{\diamond}:=([k_1, \ell_1], \dots, [k_m, \ell_m])
$$
is well defined and an order isomorphism.

 For $x\in \Z^{\mathcal{T}_{\ii}}$ let $x^{\diamond}\in \Z^{\mathcal{T}_{\ii^{\starop}}}$ be defined by 
$x^{\diamond}_{[k,\ell]}=x^{\vphantom{\diamond}}_{[k,\ell]}.$ For $y\in\Hom(\Z^{\mathcal{T}_{\ii}}, \Z)$ we denote the map $x\mapsto y(x^{\diamond})$ by $y^{\diamond}$.
Then for $\gamma\in\Gamma_a ( \mathcal{T}_{\ii^{\starop}})$ we have $\rveck{\gamma}^{\diamond}=\rveck{\gamma^{\diamond}}$ and $\drveck{\gamma}^{\diamond}=\drveck{\gamma^{\diamond}}$. Furthermore, by Definition \ref{starop} we have 
$
f_a^* x=\left(f_a x^{\diamond}\right)^{\diamond}.
$
The statement now follows from Theorem \ref{formula1}, Theorem \ref{thm:Reinekevectors} and  Theorem \ref{formula2}.
\end{proof}

\begin{defi}\label{def:dualreineke}
The subset $\mathbf{R}^*_a:=\{f^*_a x - x \mid x \in \mathbb{N}^{\mathcal{T}_{\ii}}\} \subset \mathbb{Z}^{\mathcal{T}_{\ii}}$ is called the set of \emph{dual a-Reineke vectors}.
\end{defi}

By \cite[Proposition 8.2]{Ka} we have the following description of highest weight crystals analogously to \cite[Proposition 7.4]{Rei}.
\begin{prop} Let $\lambda=\sum_{a \in [n-1]}\lambda_a \omega_a$ be a dominant integral weight of $\mathfrak{g}$. The corresponding crystal graph $B(\lambda)$ is the full subgraph of the crystal graph of $\mathbb{N}^{\mathcal{T}_{\ii}}$  given by all Lusztig data $x\in \mathbb{N}^{\mathcal{T}_{\ii}}$ such that $\drvecx{\gamma}{x}\le \lambda_a$ for all $a\in [n-1]$ and for all $\gamma\in\Gamma^*_a$.
\end{prop}

\section{Application 1 : The Crossing Formula applied to MV-polytopes}\label{MV}
In this section we apply the Crossing Formula (Theorem \ref{formula1}) to MV-polytopes. By this we obtain a proof of a stronger version of the Anderson-Mirkovi\'c (AM) conjecture. The AM conjecture was originally proved by Kamnitzer in \cite{Kam} and subsequently by Saito in \cite{Sai}.

\subsection{MV polytopes and BZ data}

As in Section \ref{notation} let $\{\alpha_a\}_{a\in [n-1]}$ be the simple roots and $\{\omega_a\}_{a\in [n-1]}$ the fundamental weights. Furthermore denote by $\mathfrak{h}_{\mathbb R}$ the real span of the basis of roots.

Recall the notation $v_S$ for a vertex of a tiling $\mathcal{T}$ from Section \ref{wordtiling}. We denote by $\mathcal{P} ([n])$ the power set of $[n]$ and identify the vertices of a tiling with a subset of  $\mathcal{P} ([n])$ via
\begin{equation}\label{eq:vertices}
v_S \mapsto S.
\end{equation}
Recall that the {chamber weights} of $\mathfrak{g}$ are the elements of $\{\sigma \omega_a \mid \sigma \in W, a\in [n-1]\}$. For $\emptyset \ne S \subset [n]$ we denote the chamber weight $\sum_{s\in S} \epsilon_s$ also by $v_S$, where $\epsilon_s$ was defined in Section \ref{notation}.

We denote by ${\mathcal{P}^o([n])}$ the set of proper subsets of $[n]$. A collection of integers $(\zzz_{S})_{S\in {\mathcal{P}^o([n])}}$ is called a $w_0$-normalized \emph{BZ-datum} if it satisfies 
$$\begin{array}{lll}
\zzz_{w_0 [a]}=0 & \forall a\in [n-1] & (w_0-\text{\emph{normalization}}) \\
\zzz_{\sigma [a]} + \zzz_{\sigma \sigma_a [a]}+\displaystyle\sum_{b\in [n-1]\setminus \{a\}} \langle \alpha_a, \alpha_b \rangle \zzz_{\sigma [b]} \le 0 & \forall \sigma\in W, \forall a\in [n-1] & (\text{\emph{edge inequalities}})
\end{array}$$
$\text{and for each triple }(\sigma,a,b)\in W \times [n-1] \times [n-1]\text{ such that }\sigma \sigma_a > \sigma,$ $\sigma \sigma_b > \sigma$, $a \ne b$, we have
$$\begin{array}{ll}
\zzz_{\sigma \sigma_a [a]}+\zzz_{\sigma \sigma_b [b]}=\min(\zzz_{\sigma [a]}+ \zzz_{\sigma \sigma_a \sigma_b[b]},\zzz_{\sigma [b]}+\zzz_{\sigma \sigma_b \sigma_a [a]}) &  \quad (\text{\emph{tropical Pl\"ucker relations}}).
\end{array}$$
We denote the set of $w_0$-normalized BZ-data by $\BZ$.

For each $(\zzz_{S})_{S\in {\mathcal{P}^o([n])}}\in \BZ$ the polytope
$$
P(\zzz) = \left\{ h \in \mathfrak{h}_\mathbb{R} \mid \left\langle h, v_{S} \right\rangle \geq   \zzz_{S} \quad \forall S \in {\mathcal{P}^o([n])}\right\},
$$
is called a ($w_0$-normalized) \emph{Mirkovi\'c-Vilonen (MV) polytope}. For a reduced word $\ii$ of $w_0$ let $\mathcal{T}_{\ii}$ be the tiling associated to $\ii$ as in Section \ref{wordtiling}. We denote the set of MV polytopes by $\mathcal{MV}$. By \cite{Kam} the set $\mathcal{MV}$ has a crystal structure isomorphic to $B(\infty)$ and for each reduced word $\ii$ the map $[\CA_{\ii}]_{\trop}:\mathcal{MV} \rightarrow \mathbb{Z}_{\ge 0}^{\mathcal{T}_{\ii}}$ induces a crystal isomorphism, where $[\CA_{\ii}]_{\trop}$ is defined as follows. Let $P(\zzz)$ be an MV-polytope with corresponding BZ-datum $\zzz$, then
\begin{align}\label{eq:bzdef}
[\CA_{\ii}]_{\trop} \left(P(\zzz)\right) &=
\left(\zzz_{o(T)} + \zzz_{u(T)} - \zzz_{\ell(T)} - \zzz_{r(T)}\right)_{T\in\mathcal{T}_{\ii}},
\end{align}
where for $T=[s,t ; S]\in\mathcal{T}_{\ii}$ in the notation of Section \ref{wordtiling} we set
\begin{align}\label{oben}
u(T)&:=S,&
o(T)&:=S \cup\{s,t\},\\ \notag
\ell(T)&:=S\cup  \{\min(s,t)\},&
r(T)&:=S\cup \{\max(s,t)\}.
\end{align}
The notation is illustrated in the following picture.
$$\begin{tikzpicture}
 \draw (0.9510565162,1.3090169943)  node[font=\scriptsize, below]{{$u(T)$}} -- (0.3632712640,2.1180339887) node[font=\scriptsize, left] {${\ell(T)}$}
 node[font=\scriptsize, right=0.35cm] {{$T$}} --  (0.9510565162,2.9270509831) node[font=\scriptsize, above] {{$o(T)$}}--
(1.5388417685,2.1180339887) node[font=\scriptsize, right] {{{$r(T)$}}}
-- (0.9510565162,1.3090169943);
\end{tikzpicture}$$

\begin{rem} The naming $[\CA_{\ii}]_{\trop}$ comes from the fact that this map is given by the tropicalization of the Chamber Ansatz (see \cite{BFZ}).
\end{rem}

\subsection{The Anderson-Mirkovi\'c conjecture in type $A$}

For $a\in [n-1]$ and $S\subset [n]$ we set $\sigma_a S := (S\backslash \{ a \}) \cup \{ a+1 \}$. In Section \ref{orderlatticesec} we introduced the poset $(\Gamma_a, \preceq)$ of $a$-crossings. By translating $\drvecx{\gamma}{x}$ defined in \eqref{Fdef2} for the $\preceq$-maximal element $\gamma=(a,a+1)\in\Gamma_a$ to the language of BZ-data we obtain:
\begin{thm}\label{sAM}
Let $\zzz,\zzz'$ be BZ-data such that ${f}_a P(\zzz)=P(\zzz')$. Then we have for $S\subset [n]$
\begin{align}\label{sAMform}
\zzz_S - \zzz'_S &= \begin{cases}  1 & \text{if $a\in S$, $a+1\notin S$ and $\zzz_S - \zzz_{\sigma_a S} \geq \zzz_{[a]} - \zzz_{\sigma_a [a]},$} \\ 0 & \text{else.} \end{cases}
\end{align}
Furthermore, if $a\in S$ and $a+1\notin S$, then
$$
\zzz_S - \zzz_{\sigma_a S} \leq \zzz_{[a]} - \zzz_{\sigma_a [a]}.
$$
\end{thm}
\begin{proof}
For any $S\subset[n]$ there exists a tiling $\mathcal{T}$ such that $v_S$ is a vertex of $\mathcal{T}$. If $a\notin S$ or $a+1\in S$ then $v_S$ is a vertex of a tile $T$ which is not in $\mathcal{W}_a$ and using Lemma \ref{hexexists} one finds by induction that
\begin{equation}\label{eq:simple}
\zzz_S - \zzz'_S = 0.
\end{equation}
Thus, \eqref{sAMform} holds if $a\notin S$ or $a+1\in S$.

From now on we assume $a\in S$ and $a+1 \notin S$.
We choose a tiling $\mathcal{T}_{\ii}$ such that $v_{\sigma_a [a]}$ is a vertex of $\mathcal{T}_{\ii}$. Then the tile $[a,a+1]$ has vertices $v_{[a-1]}$, $v_{[a]}$, $v_{[a+1]}$, $v_{\sigma_a [a]}$ and $\gamma=(a,a+1)$ is the only $a$-crossing since $[a,a+1]$ intersects the left boundary of $\mathcal{T}_{\ii}$ with two edges. Thus, using (\ref{eq:bzdef}) we obtain from the Crossing Formula (Theorem \ref{formula1})
\begin{align}\label{epsilonCA}
\varepsilon_a \left(P(\zzz)\right)&=\varepsilon_a \left(\left[\CA_{\ii}\right]_{\trop}\left(P(\zzz)\right)\right) = \drvecxd{\left( a,a+1\right)}{\left[\CA_{\ii}\right]_{\trop} \left(P(\zzz)\right)} = \\ \notag
&=\left( \left[\CA_{\ii}\right]_{\trop}\left(P(\zzz)\right)\right)_{[a,a+1]} = \zzz_{[a+1]} + \zzz_{[a-1]} - \zzz_{[a]} - \zzz_{\sigma_a [a]}.
\end{align}

From  \cite[Theorem 9.3]{DKK} (see also \cite{LZ}) we deduce the existence of a tiling $\mathcal{T}_{\jj}$ such that both $v_S$ and $v_{\sigma_a S}$ are vertices of $\mathcal{T}_{\jj}$ (since $S$ and $\sigma_a S$ form a strongly separated collection). For $T=[a,a+1]\in\mathcal{T}_{\jj}$ we have in the notation \eqref{oben} that $\ell(T)=S$, $r(t)=\sigma_a S$, $u(T)=S \backslash \{a\}$, $o(T)=S\cup \{a+1\}$. Thus, by (\ref{eq:simple}) we have $\zzz_{o(T)}-\zzz_{o(T)}'=\zzz_{u(T)}-\zzz_{u(T)}'=\zzz_{r(T)}-\zzz_{r(T)}'=0$.
Setting $x=[\CA_{\jj}]_{\trop} (P(\zzz))\in \mathbb{N}^{\mathcal{T}_{\jj}}$ we obtain from the Crossing Formula that
\begin{equation}\label{maincase}
\zzz_{S}-\zzz_S' = -(x_{[a,a+1]} - {f}_a x_{[a,a+1]}) =
\begin{cases} 1 & \text{if }\drvecxd{\left( a,a+1\right)}{x}  \geq \varepsilon_a(P(\zzz))=\varepsilon_a(x),\\0 &\text{else.} \end{cases}
\end{equation}
Since by the Crossing Formula we have $\drvecx{(a,a+1)}{x}\leq \varepsilon_a(x)$, the claim follows from \eqref{epsilonCA}, \eqref{maincase} and
\begin{equation*}\label{FmaxCA}
\drvecxd{\left(a,a+1\right)}{x}= \zzz_S - \zzz_{\sigma_a S}+
\zzz_{[a+1]} + \zzz_{[a-1]} - 2 \zzz_{[a]}.
\end{equation*}
\end{proof}
We obtain the Anderson-Mirkovi\'c conjecture in type $A$ as an immediate consequence of Theorem \ref{sAM}:
\begin{cor}[{\cite[Corollary 5.6]{Kam}}]
Let $\zzz, \zzz'$ be BZ-data such that ${f}_a P(\zzz)=P(\zzz')$. Then we have for $S\subset [n]$
$$\zzz_S - \zzz'_S = \begin{cases} \max\left(0, \zzz_S - \zzz_{\sigma_a a} + \zzz_{\sigma_a [a]} - \zzz_{[a]} +1 \right) & \text{if $a\in S$ and $a+1\notin S,$}\\ 0 &\text{else.} \end{cases}$$
\end{cor}
\section{Application 2: A duality between Lusztig's and Kashiwara's parametrization}\label{string}
Let $\mathcal{T}_{\ii}$ be the tiling associated to the reduced word $\ii$ for $w_0$ as in Section \ref{wordtiling}. In Definition \ref{def:dualreineke} we defined the set $\mathbf{R}_a^*=\{f_a^* x - x \mid x \in \mathbb{N}^{\mathcal{T}_{\ii}}\}$ of dual $a$-Reineke vectors associated $\ii$ where $^*$ is the Kashiwara involution on the set of $\ii$-Lusztig data (see Section \ref{Kashiwarainvolution}). Let $\mathbb{G}_m$ be the multiplicative group. In this section we associate to $\mathbf{R}_a^*$ a function $r_a$ on the torus $\TTis:=\GGi$ and show that it transforms under a certain geometric lift of the transition maps between the string parametrizations defined by Berenstein and Zelevinsky in \cite{BZ2}.
 As a consequence we obtain that the cone corresponding to Kashiwara's string parametrization of the dual canonical basis is polar to the cone spanned by $\mathbf{R}_a^*$.

\subsection{String parametrizations}\label{sparas}
We recall the \emph{string parametrization} of the dual canonical basis $\mathbb{B}^{\text{dual}}$ corresponding to the reduced word $\ii=(i_1, \ldots,i_N)$. The parametrizing set
is the set of $\mathbf i$-string data of elements of the crystal $B(\infty)$. An \emph{$\mathbf i$-string datum} $\stp_{\ii }(\bbb)$ of $\bbb\in B(\infty)$ is a tuple $\stp_{\ii }(\bbb)=(x_1, \ldots, x_N)\in \N^N$ defined inductively as follows.
\begin{align*}
x_1 &= \max\left\{ k \in \N \mid {e}_{i_1}^k \bbb \ne 0\right\}, \\
x_2 &= \max\left\{ k \in \N \mid {e}_{i_2}^{k}{e}_{i_1}^{x_1} \bbb \ne 0\right\}, \\
& \vdots \\
x_N &= \max\left\{ k \in \N \mid {e}_{i_N}^k {e}_{i_{N-1}}^{x_{N-1}}\cdots {e}_{i_1}^{x_1} \bbb \ne 0 \right\}.
\end{align*}
By \cite{BZ}, \cite{Lit} the set
\begin{equation}\label{scdef}
\stc_{\ii}:=
\left\{\stp_{\mathbf i}(\bbb) \mid \bbb \in B(\infty)\right\}\subset \N^N
\end{equation}
is a polyhedral cone called the \emph{string cone associated to $\ii$}.

We equip the cone $\stc_{\ii}$ with a crystal structure following \cite{NZ,Ka2}. First we define for $a\in[n-1]$ the operator ${f}_a$ on $\mathbb{Z}^N$ as follows. We set for $x=(x_1, \ldots, x_N)\in \mathbb{Z}^N$ and $a\in[n-1]$
\begin{align*}
\nu_j(x)&:=x_j+ \sum_{ j < t\leq n-1}\left<h_{i_j},\alpha_{i_t}\right>x_t,\\
\nu^{(a)}(x)&:=\max\{\nu_j(x) \mid j\in[N], i_j=a \}.
\end{align*}
We let $j_{0}^{(a)}$ be the smallest $j'$ such that $\nu_{j'}(x)=\nu^{(a)}(x)$ and define
\begin{equation}\label{eq:stringop}
{f}_a x:=\left(x_k+\delta_{k,j_0}\right)_{k\in[N]}\in \mathbb{Z}^N.
\end{equation}
By \cite{NZ} the map
$$B(\infty)\hookrightarrow \mathbb{Z}_{\ge 0}^N\subset \mathbb{Z}^N, \quad \bbb \mapsto \stp_{\ii}(\bbb^*)$$
is an embedding of crystals. Here ${}^*$ is the Kashiwara involution defined in Section \ref{Kashiwarainvolution}.

\begin{ex} Let $n=3$ and $\ii=(1,2,1)$. Recall the associated tiling $\mathcal{T}_{\ii}$:
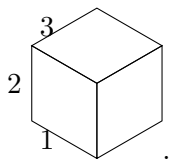
$$\begin{tikzpicture}
\draw (0.0000000000,0.0000000000)  -- node[midway, below, left]{$1$} (-0.8660254037,0.4999999999) -- node[midway, left]{$2$} (-0.8660254037,1.5000000000) --
(0.0000000000,1.0000000000)-- (0.0000000000,0.0000000000);
\draw (0.0000000000,1.0000000000)  -- (-0.8660254037,1.5000000000) -- node[midway, left]{$3$} (0.0000000000,2.0000000000) --
(0.8660254037,1.5000000000)-- (0.0000000000,1.0000000000);
\draw (0.0000000000,0.0000000000)  -- (0.0000000000,1.0000000000) -- (0.8660254037,1.5000000000) --
(0.8660254037,0.4999999999)-- (0.0000000000,0.0000000000);
\end{tikzpicture}.$$
Writing $x\in {\N}^{\mathcal{T}_{\ii}}$ as $x=(x_1,x_2,x_3)= (x_{[1,2]}, x_{[1,3]}, x_{[2,3]})$ we compute using the dual Crossing formula (Theorem \ref{dualcrossing}), Theorem \ref{formula2} and  \eqref{eq:stringop}:
\begin{align*}
\stp_{\ii}\left(\left( f_2 (0,0,0)\right)^*\right)&=\stp_{\ii}(f_2^* (0,0,0))  = \stp_{\ii} (0,0,1)  = (0,1,0)=f_2 (0,0,0), \\
\stp_{\ii}\left(\left(f_1 {f}_2 (0,0,0)\right)^*\right)&=\stp_{\ii}\left(f_1^* f_2^* (0,0,0)\right)= \stp_{\ii} (0,1,0)   = (0,1,1)=f_1 f_2 (0,0,0),\\
\stp_{\ii}\left( \left( \left(f_1\right)^2 f_2 (0,0,0)\right)^*\right)&=\stp_{\ii} \left( \left(f_1^*\right)^2 f_2^* (0,0,0) \right)= \stp_{\ii} (1,1,0) = (1,1,1) = \left(f_1\right)^2 f_2 (0,0,0). 
\end{align*}
\end{ex}

Analogously to the definition of the set of dual Reineke vectors $\mathbf{R}^*(\mathcal{T}_{\ii})=\bigcup_a \mathbf{R}_a^*(\mathcal{T}_{\ii})$ (see \eqref{darvecs2}) for $\mathcal{T}_{\ii}$
we associate to the string cone $\stc_{\ii}$ the sets
\begin{equation*}
\mathbf{R}_a\left(\stc_{\ii} \right):=\left\{
{f}_a x -x \,\middle|\, x\in \stc_{\ii}\right\}, \qquad \mathbf{R}\left(\stc_{\ii} \right):= \bigcup_a \mathbf{R}_a\left(\stc_{\ii} \right).
\end{equation*}
Using \eqref{eq:stringop} we obtain the following description of $\mathbf{R}_a(\stc_{\ii})$:
\begin{lem}\label{stringreinvecs}
For a reduced word $\ii=(i_1,\dots, i_N)$ and $a\in[n-1]$ we have
$$
\mathbf{R}_a\left(\stc_{\ii} \right)=\left\{e_k:=\left(\delta_{k,j} \right)_{j\in[N]} \,\middle|\, i_k=a \right\}.
$$
\end{lem}
\begin{proof}
By \eqref{eq:stringop} we have $\mathbf{R}_a(\stc_{\ii} ) \subset \{e_k | i_k=a\}$. We fix $k\in [N]$ with $i_k=a$. Since $\dim \mathbb{S}_{\ii} = N$, there exists $x=(x_1, \dots, x_N)\in \stc_{\ii}$ with $x_{k}>0$. Writing $x=f_{a_1}  \cdots f_{a_M} (0,\dots, 0)$ we obtain from \eqref{eq:stringop}
$$
e_k \in \left\{ f_a y - y \,\middle|\, a_{\ell}=a, y=f_{a_{\ell+1}} \cdots f_{a_M} \left(0, \dots ,0\right) \right\}\subset \mathbf{R}_a\left(\stc_{\ii}\right).
$$
\end{proof}

\subsection{Geometric liftings}
As in \cite[Chapter 5]{BZ2} we associate to a reduced word $\ii$ tori $\TTil:=\Hom(\Hom(\Z^{\mathcal{T}_{\ii}}, \Z), 
\GG)$, $\TTis:=\Hom(\Z^{\mathcal{T}_{\ii}}, \GG)$ and define maps
\begin{align*}
 \trs^{\ii}_{\jj} &: \TTis \rightarrow \TTisp,  \\
 \trl^{\ii}_{\jj} &: \TTil \rightarrow \TTilp
\end{align*}
as follows.
If {$\jj$} is obtained from {$\ii$} by replacing the subword $(i_k,i_{k+1})$ by $(i_{k+1},i_k)$, where $|i_k - i_{k+1}|>1$ (commutation move), then
\begin{equation}\label{smutc}
(\trs_{\jj}^{\ii} x)_{[s,t]}=(\trl_{\jj}^{\ii} x)_{[s,t]}=x_{[s,t]}.
\end{equation}
If {$\jj$} is obtained from {$\ii$} by a braid move that corresponds to a flip at $\mathcal{H}:=\{[s,t], [s,u], [t,u]\} \subset \mathcal{T}_{\ii}$ with $s<t<u$, we set $y=\trs_{\jj}^{\ii} {x}$, where $y_T=x_T$ for $T\notin \{[s,t], [s,u], [t,u]\}$ and 
\begin{align} \label{smutb}
\begin{split}
y_{[s,t]}&=\frac{x_{[s,t]}x_{[t,u]} + x_{[s,u]}}{x_{[t,u]}} ,\\ 
y_{[s,u]}&=x_{[s,t]}x_{[t,u]},\\ 
y_{[t,u]}&=\frac{x_{[s,u]}x_{[t,u]} }{x_{[s,t]}x_{[t,u]} + x_{[s,u]}}.
\end{split}\end{align}
Furthermore, we set $y=\trl_{\jj}^{\ii} {x}$, where $y_T=x_T$ for $T\notin \{[s,t], [s,u], [t,u]\}$ and
\begin{align} \label{lmutb}
\begin{split}
y_{[s,t]}&=\frac{x_{[s,t]} x_{[s,u]}}{x_{[s,t]}+ x_{[t,u]}},\\ 
y_{[s,u]}&=x_{[s,t]}+x_{[t,u]},\\ 
y_{[t,u]}&=\frac{x_{[s,u]}x_{[t,u]}}{x_{[s,t]}+ x_{[t,u]}}.
\end{split}\end{align}
One checks by direct computation, that the compositions of maps of type $\eqref{smutc}$ and $\eqref{smutb}$ respect the relations of the symmetric group and we thus obtain a definition of $\trs^{\ii}_{\jj}$ for $\ii$, $\jj$ arbitrary. Similarly, we obtain for reduced words $\ii$ and $\jj$ maps $\trl_{\jj}^{\ii}$ by composing maps of type \eqref{smutc} and \eqref{lmutb}.

By \cite[Theorem 2.2]{BZ} the tropicalizations of $\trs_{\jj}^{\ii}$ intertwine the string parametrizations, i.e.
 $$
 \stp_{\ii} = \stp_{\jj} \circ \trod{\trs^{\ii}_{\jj}}.
 $$
Similarly, since the tropicalization of \eqref{lmutb} is given by \eqref{tlmutb}, we have that $\tro{\trl_{\jj}^{\ii}}=\RR_{\jj}^{\ii}$ intertwines the Lusztig parametrizations associated to $\ii$ and $\jj$.

\subsection{Transformation behaviour of Reineke vectors}
In Definition \ref{def:dualreineke} we defined the set of dual $a$-Reineke vectors
\begin{equation}\label{darvecs2}
\mathbf{R}_a^* \left(\mathcal{T}_{\ii}\right)= \left\{f_a^* x - x \,\middle|\, x\in\mathbb{N}^{\mathcal{T}_{\ii}}  \right\}
\end{equation}
associated to a reduced word $\ii$ and $a\in [n-1]$.
With the notation of Section \ref{dualcrossing1} we obtain by the dual Crossing Formula (Theorem \ref{dualcrossing})
\begin{equation}\label{darvecs}
\mathbf{R}_a^* \left(\mathcal{T}_{\ii}\right)= \left\{  \rvecd{\gamma} \,\middle|\, \gamma\in\mathcal{R}_a^* \right\}.
\end{equation}
Similarly, by Lemma \ref{stringreinvecs}  we have $\mathbf{R}_a(\stc_{\ii} ) := \{
{f}_a x -x \,|\, x\in \stc_{\ii} \} =\{e_k | i_k=a\}$.
We define the functions
\begin{align}\label{rpot}
\ra = \rai : \TTis \rightarrow \mathbb{A} &, \quad x \mapsto
\sum_{ y \in \mathbf{R}_a^* \left(\mathcal{T}_{\ii}\right)} x^y,\\\label{lpot}
\ssa = \ssai : \TTil \rightarrow \mathbb{A} &, \quad x \mapsto
\sum_{ y \in \mathbf{R}_a \left(\stc_{\ii}\right)} x^y,\\
\notag
x^y&=\prod_{T\in\mathcal{T}_{\ii}} x_T^{y_T}.
\end{align}

Under a change of reduced words the functions $\ra$ and $\ssa$ transform the following way.
\begin{thm}\label{rtrans}
For $a\in [n-1]$ and reduced words $\ii$ and $\jj$ we have:
\begin{align}\label{eq:rtrans}
\rai&=
 \raj \circ \trs^{\ii}_{\jj},\\\label{eq:ltrans}
\ssai&=
 \ssaj \circ \trl^{\ii}_{\jj}.
\end{align}
\end{thm}

\begin{proof}
There is nothing show if $\jj$ is obtained by a commutation move or a braid move replacing $(i_{k-1}, i_k, i_{k+1})$ by $(i_k, i_{k-1}, i_k)$ with $a\notin\{i_{k-1}, i_k\}$. Thus, by induction and exchanging $\ii$ and $\jj$ if necessary we can assume $\jj$ is obtained from $\ii$ by a braid move replacing $(a, a+1, a)$ by $(a+1, a, a+1)$. We denote the corresponding hexagon (see \eqref{hexbraid}) by  $\mathcal{H}=\{[s,t], [s,u], [t,u]\} \subset \mathcal{T}_{\ii}$ and set $y=\trl_{\jj}^{\ii} {x}$.
Equality \eqref{eq:ltrans} follows, since by \eqref{lpot} and \eqref{lmutb} we have
$$
\ssai (x) - \ssaj (y) =  x_{[s,t]}+x_{[t,u]}-y_{[s,u]}  = 0.
$$

We next show \eqref{eq:rtrans}. Using induction over $\# \mathcal{W}_a$ (see \eqref{def:comb}) we can assume by Lemma \ref{hexexists} that
 $\mathcal{T}_{\jj}$ is obtained from $\mathcal{T}_{\ii}$ by flipping a hexagon hexagon $\h=\{[a,s], [a,t], [s,t]\}\subset \mathcal{W}_a$. In Section \ref{dualcrossing1} we introduced the set of dual $a$-crossings $\Gamma_a^* (\mathcal{T}_{\ii})$ and  the set of dual $a$-Reineke crossings $\mathcal{R}_a^* (\mathcal{T}_{\ii})\subset \Gamma_a^* (\mathcal{T}_{\ii})$. Let $\pi: \Gamma_a^* (\mathcal{T}_{\ii}) \rightarrow \Gamma_a^* (\mathcal{T}_{\jj})$ be the map defined by \eqref{combext}. Then we have for  $\lambda\in\mathcal{R}_a^* (\mathcal{T}_{\jj})$ and $R:= \pi^{-1} (\lambda) \cap \mathcal{R}_a^*(\mathcal{T}_{\ii})$
\begin{equation}\label{reinmut}
\sum_{\gamma\in R}
x^{\rvec{\gamma}}=\left( \trs^{\ii}_{\jj}(x)\right)^{\rvec{\lambda}}.
\end{equation}
 Since $\pi\mathcal{R}_a^* (\mathcal{T}_{\ii}) = \mathcal{R}_a^* (\mathcal{T}_{\jj})$, we obtain from \eqref{reinmut} and \eqref{darvecs}
$$
\rai (x)=\sum_{\gamma\in \mathcal{R}_a^* \left( \mathcal{T}_{\ii}\right)}
x^{\rvec{\gamma}} = \sum_{\lambda\in \mathcal{R}_a^* \left( \mathcal{T}_{\jj}\right)}\left( \trs^{\ii}_{\jj}(x)\right)^{\rvec{\lambda}}=\raj \circ \trs^{\ii}_{\jj}(x).
$$
\end{proof}

\subsection{Duality of cones}
As an application of Theorem \ref{rtrans} we derive the following duality between Lusztig's' parametrizations of $B(\infty)$ and the string parametrizations.

For $\emptyset \ne C\subset \mathbb{R}^N$ we define the \emph{polar cone} $C^{\pol}$ by
$$C^{\pol}:=\{x\in \mathbb{R}^N \mid  \forall y\in C \, : \, \langle x, y\rangle \ge 0 \}.$$

We have the following duality.
\begin{thm}\label{duality} Let $\mathbf i$ be a reduced word for $w_0$ and $\mathcal{T}_{\ii}$ the tiling associated to $\ii$. Then
$$\stc_{\ii}  =\mathbf{R}^*(\mathcal{T}_{\ii})^{\pol} \qquad \text{ and } \qquad
\mathbb{N}^{\mathcal{T}_{\ii}}= \mathbf{R}(\stc_{\ii} )^{\pol} $$
\end{thm}
\begin{proof}
Applying Theorem \ref{dualcrossing} to the tiling associated to the lexicographically minimal word $\ii_0=(1,2,1,\dots,n , n-1, \dots 1 )$ and comparing $\mathbf{R}^*(\mathcal{T}_{\ii_0})$ with the defining inequalities of the string cone $\stc_{\ii}$ as computed in \cite{BZ, Lit} one obtains $\stc_{\ii_0}  =\mathbf{R}^*(\mathcal{T}_{\ii_0})^{\pol}$.
Tropicalizing  \eqref{eq:rtrans} and using $\tro{\trs^{\ii}_{\jj}} \stc_{\ii} = \stc_{\jj}$ the statement
 $\stc_{\jj}  =\mathbf{R}^*(\mathcal{T}_{\jj})^{\pol}$ now follows for all $\jj$.
 The statement
 $\mathbb{N}^{\mathcal{T}_{\ii}}= \mathbf{R}(\stc_{\ii} )^{\pol}$ follows from Lemma \ref{stringreinvecs}.
\end{proof}

\begin{rem}\label{gpremark} As noted in Remark \ref{gprem}, Reineke crossings translate into the notion of rigorous paths which appear in the work \cite{GP}. Here Gleizer and Postnikov associate vectors to rigorous paths, which can be seen to coincide with the dual Reineke vectors $\mathbf{R}^*(\mathcal{T}_{\ii})=\cup_a \mathbf{R}_a^*(\mathcal{T}_{\ii})$.
As a consequence of the transformation behaviour of certain tropical functions $M^{\ii,a}_{a,a+1}$ under the tropicalization of the map \eqref{smutb} Gleizer and Postnikov show in loc. cit. that the vectors associated to rigorous paths are the defining inequalities of the string cones $\stc_{\ii}$. One finds that $M^{\ii,a}_{a,a+1}$ coincides with the tropicalization of $\rai$ and thus by Theorem \ref{dualcrossing} and \cite[Theorem 5.11]{GP} we obtain the first part of Theorem \ref{duality} stating $\stc_{\ii}  =\mathbf{R}^*(\mathcal{T}_{\ii})^{\pol}$.
\end{rem}

\begin{rem}\label{zelremark} For the special case of reduced words adapted to quivers the equality
 $$\stc_{\ii} =\mathbf{R}^*(\mathcal{T}_{\ii})^{\pol}$$
was already obtained by Zelikson in \cite{Ze} using the results of \cite{Rei, GP}. Zelikson conjectured this relation for simply laced Lie algebras and reduced words adapted to quivers satisfying a certain homological condition. We conjecture that this relation holds for all simple Lie algebras and arbitrary reduced words, i.e. the string cone $\stc_{\ii}$ is polar to the cone spanned by the vectors of the form $f_a^* x-x$ for an $\ii$-Lusztig datum $x$ (see \cite[Section 7]{Ze} for an example in type $D$).
\end{rem}

\begin{ex} We give an example for $n=3$. Consider the following tiling corresponding to the reduced word $\ii=(2,1,2)$.
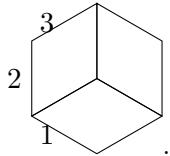
\begin{center}
\begin{tikzpicture}
\draw (0.0000000000,0.0000000000)  -- node[midway, left] {$1$} (-0.8660254037,0.4999999999) -- (0.0000000000,0.9999999999) --
(0.8660254037,0.4999999999)-- (0.0000000000,0.0000000000);
\draw (0.8660254037,0.4999999999)  -- (0.0000000000,0.9999999999) -- (0.0000000000,1.9999999999) --
(0.8660254037,1.5000000000)-- (0.8660254037,0.4999999999);
\draw (-0.8660254037,0.4999999999)  -- node[midway,left] {$2$} (-0.8660254037,1.5000000000) -- node[midway,left] {$3$}  (0.0000000000,2.0000000000) --
(0.0000000000,0.9999999999)-- (-0.8660254037,0.4999999999);
\end{tikzpicture}.
\end{center}
The ordering of the tiles in the $3$-order is as follows:
$$[2,3] \le_3 [1,3] \le_3 [1,2].$$
We label the coordinates with respect to that order. The dual Reineke vectors are given by
$$\mathcal{R}^*\left(\mathcal{T}_{\ii}\right)=\left\{(1,0,0), (0,1,-1), (0,0,1)\right\}.$$
By Theorem \ref{duality} we obtain  the inequalities for the string cone for this reduced word:
$$v_{2,3}\ge 0 \qquad \qquad v_{1,3} \ge v_{1,2} \ge 0.$$
\end{ex}

\section{Application 3: Potential functions for cluster varieties and crystal operations}\label{cluster}
In Section \ref{string} we associated to a reduced word $\ii$ and
$a\in [n-1]$ the function $\ra(x) = \rai(x)=\sum_{y \in \mathbf{R}_a^* \left(\mathcal{T}_{\ii}\right)} x^y$ on the torus $\TTis$ encoding the crystal operations on $\ii$-Lusztig data.
In this section we relate $\ra$  to potential functions arising from the theory of geometric crystals in \cite{BK} and from mirror symmetry for cluster varieties in \cite{GHKK}. The relation is given by regular changes of coordinates given by maps introduced by Berenstein and Zelevinsky (see \cite[Theorem 4.8]{BZ2}) to solve certain factorization problems.

As a consequence we obtain explicit descriptions of these potential functions written in certain torus coordinates associated to reduced words $\ii$ on the reduced double Bruhat cell $L^{e,w_0}$ as well as the coordinate transformation translating one potential function to the other. Further we obtain unimodular identifications (i.e. lattice isomorphisms translating the defining inequalities) of the cones defined by the tropicalizations of the potential functions with the string cones $\stc_{\ii}$ defined in \eqref{scdef}.

\subsection{Unipotent radicals as cluster varieties}\label{clusterintro}
Let $B_+$ denote the upper triangular matrices, $B_-$ the lower triangular matrices in $G=\text{SL}_n(\C)$ and $\unip\subset B_+$ be the unipotent radical of $B_+$. The \emph{reduced double Bruhat cell $L^{e,w_0}$ associated to $e$ and $w_0$} is the non-vanishing locus of the upper right minors in $\unip$:
$$L^{e,w_0}=B_-w_0B_-\cap \mathcal{N}.$$

Following Berenstein-Fomin-Zelevinsky, Fomin-Zelevinsky and Fock-Goncharov \cite{FZ, BFZ2, FZ2, FG} we endow $L^{e,w_0}$ with an $\mathcal{A}$-cluster structure $(\TTsea)$ and a mirror dual $\mathcal{X}$-cluster structure $(\TTsex)$.  The families $(\TTsea)$ and $(\TTsex)$ are up to codimension $2$ open coverings of $L^{e,w_0}$ by tori $\TTsea:=\mathbb{G}_m^{\SE}$ and $\TTsex:=\Hom(\Hom(\GG,\TTsea), \GG)$. The birational transition maps $\ttsa$ and $\ttsx$, called $\mathcal{A}$- and $\mathcal{X}$-cluster mutation respectively, are given as follows.
Assume that the exchange graph $Q_{\SE'}$ of the seed $\SEp$ is obtained from the exchange graph $Q_{\SE}$ of the seed $\SE$ by cluster mutation at a vertex $k$. Let $I$ be a set containing $k$ which consistently labels both the vertices $Q_{\SE}$ and $Q_{\SE'}$. Then (see \cite[Equations (13) and (14)]{FG})
\begin{alignat*}{3}
{\ttsa}&: \TTsep   \dashedrightarrow \TTse, \quad
\left({\ttsa} x \right)_i &&= \begin{cases} \displaystyle\prod_{j \,:\, \varepsilon_{j,k}>0} \frac{x_j}{x_k} + \displaystyle\prod_{j \,:\, \varepsilon_{j,k} < 0} \frac{x_j}{x_k} & \text{if $i=k$,} \\x_i &  \text{else,} \end{cases} \\
{\ttsx} &: \TTsexp  \dashedrightarrow \TTsex , \quad
\left({\ttsx} x \right)_i &&= \begin{cases} x_k^{-1} & \text{if $i=k$,} \\ x_i \left(1+x_k^{-\sgn \varepsilon_{i,k}}\right)^{-\varepsilon_{i,k}} &  \text{else,} \end{cases}
\end{alignat*}
where $\varepsilon_{i,k}$ denotes the number of arrows in $Q$ from $i$ to $k$.

The coordinate functions on $\TTsea$ and $\TTsex$ are called the $\mathcal{A}$- and $\mathcal{X}$-cluster variables, respectively. We refer to the set of such variables for a given torus $\TTsea$ or $\TTsex$ together with the exchange graph $Q_{\SE}$ as an $\mathcal{A}$- and $\mathcal{X}$-cluster seed, respectively.

As initial seed we may chose the seed $\SEi$ associated to a reduced word $\ii$ in \cite[Section 2.2]{BFZ2}. 
We recall from \cite{DKK2} the construction of the $\mathcal{A}$-cluster seed $\SE=\SEi$ associated to a tiling $\mathcal{T}=\mathcal{T}_{\ii}$.  Let $V(\mathcal{T})$ be the set of vertices of $\mathcal{T}$ that do not lie on the left boundary of $\mathcal{T}$.
As in (\ref{eq:vertices}) we identify a vertex $v=v_S$ of $\mathcal{T}$ with a subset of  $S\subset\mathcal{P} ([n])$.
To a vertex $v_S\in V(\mathcal{T})$ we associate the chamber minor $\Delta_{\scriptsize S}=\Delta_{\scriptsize S}^{\scriptsize{[\#S]}}$, which is the regular function on $L^{e,w_0}$ taking the determinant of the submatrix consisting of the first $\#S$ rows and columns labeled by $S$.   The collection of chamber minors $\{\Delta_S \mid v_S \in V(\mathcal{T})\}$ associated to the vertices $V(\mathcal{T})$ is the set of $\mathcal{A}-$cluster variables in the seed $\SE_{\ii}$ and thus provides the coordinate functions of $\TTia:=\TTse$, i.e. $\TTia=\mathbb{G}_m^{V(\mathcal{T}_{\ii})}$.
The chamber minors corresponding to the vertices on the right boundary are the frozen variables.

The $\mathcal{X}$-cluster torus $\TTix$ is the dual torus of $\TTia$, i.e. $\TTix:=\Hom(\Hom(\GG,\TTia),\GG)$.

The exchange graph for $\SE_{\ii}$ can be obtained from $\mathcal{T}=\mathcal{T}_{\ii}$ in the following
 way. The vertices of the exchange graph are given by the elements in $V(\mathcal{T})$. Using the notation of Section \ref{MV} we draw for each tile $T\in \mathcal{T}$ with $\ell(T)\in V(\mathcal{T})$ an arrow pointing from $\ell(T)$ to $r(T)$. The edges of $\mathcal{T}$ get oriented such that we have oriented cycles $\ell(T) \rightarrow r(T) \rightarrow o(T) \rightarrow \ell(T)$ and $\ell(T) \rightarrow r(T) \rightarrow u(T) \rightarrow \ell(T)$ for all $T\in \mathcal{T}$. If this procedure does not yield a unique orientation for an edge of $\mathcal{T}$ this edge gets deleted. Furthermore we delete all arrows between vertices on the boundary.

\begin{ex}We describe the $\mathcal{A}$-cluster seed $\SEi$ for the reduced word  $\ii=(1,2,1)$ in the case $n=3$. The cluster variables $\{\Delta_{\{2\}}, \Delta_{\{2,3\}}, \Delta_{\{3\}}\}$ are associated to the vertices of $\mathcal{T}_{\ii}$ as follows.
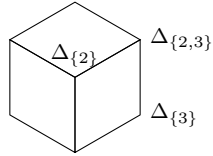
\begin{center}
\begin{tikzpicture}
\draw (0.0000000000,0.0000000000) node[font=\scriptsize, below] {} -- (-0.8660254037,0.4999999999) node[font=\scriptsize, left] {} -- (-0.8660254037,1.5000000000) node[font=\scriptsize, left] {} --
(0.0000000000,1.0000000000) -- (0.0000000000,0.0000000000);
\draw (0.0000000000,1.0000000000) node[font=\scriptsize, above] {$\Delta_{\{2\}}$}  -- (-0.8660254037,1.5000000000)  -- (0.0000000000,2.0000000000) node[font=\scriptsize, above] {}  --
(0.8660254037,1.5000000000)node[font=\scriptsize, right] {$\Delta_{\{2,3\}}$}-- (0.0000000000,1.0000000000);
\draw (0.0000000000,0.0000000000)  -- (0.0000000000,1.0000000000) -- (0.8660254037,1.5000000000) --
(0.8660254037,0.4999999999) node[font=\scriptsize, right] {$\Delta_{\{3\}}$} -- (0.0000000000,0.0000000000);
\end{tikzpicture}
\end{center}
The exchange graph $Q_{\ii}$ is given by:
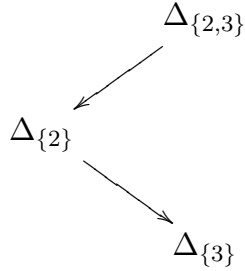
\begin{center}
\resizebox{12em}{!}{$\scriptscriptstyle{\xymatrix@R-=0.75cm{
   &  & \Delta_{\{2,3\}}  \ar[ld]  \\
 & \Delta_{\{2\}} \ar[rd]& \\
 & & \Delta_{\{3\}} \\
}}
$}
\end{center}
\end{ex}

\begin{rem}
Our convention differs slightly from the one used in \cite{BFZ2} in that we associate the $\mathcal{A}$-cluster variables to the vertices of the exchange graph in the order reversed to \cite[Equation (2.11)]{BFZ2}. 
\end{rem}

The following notion is crucial for the definition of the two potential functions we deal with in this section.

\begin{defi} An $\mathcal{A}$-cluster (resp. $\mathcal{X}$-cluster) seed $\SE$ consisting of the cluster variables $\{c_1,\ldots,c_N\}$ is \emph{optimized} for a frozen variable $c_i$ if there is no arrow in $Q_{\SE}$ pointing from $c_i$ towards a non-frozen variable.
\end{defi}
If $\SE$ corresponds to a reduced word $\ii$ we can reformulate this definition as follows. Denoting the $\mathcal{A}$- or $\mathcal{X}-$cluster variable attached to the vertex $v_S$ of $\mathcal{T}_{\ii}$ by $c_S$ we have:
\begin{lem}\label{optimized} Let $\ii$ be a reduced word and $\mathcal{T}_{\ii}$ be the corresponding tiling. Then the seed $\SEi$ corresponding to $\ii$ is optimized for $c_{\{a+1,\ldots,n\}}$ if and only if the tile $[a,a+1]$ intersects the right boundary of $\mathcal{T}_{\ii}$ in two edges. In other words $\SEi$ is optimized for $c_{\{a+1,\ldots,n\}}$ if and only of $\ii$ is commutation equivalent to a reduced word ending with $n-a$.
\end{lem}
\begin{proof} This follows directly from the construction of the exchange graph of  $\SE_{\mathcal{T}}$ and the definition of $\mathcal{T}_{\ii}$ given in \eqref{eq:roots}.
\end{proof}

\subsection{Gross-Hacking-Keel-Kontsevich potential functions and Reineke vectors}

The unipotent radical $\unip$ is the partial compactification of the reduced double Bruhat cell $L^{e,w_0}$:
$$\unip=L^{e,w_0} \bigsqcup \displaystyle\bigcup_{a\in[n-1]} D_a,$$
where $D_a$ is the divisor of zeros of the frozen variable $\Delta_{\{n-a+1,\ldots,n\}}$.

In \cite{GHKK} a Landau-Ginzburg potential $W=\sum_{a\in[n-1]} W_a$ on the $\mathcal{X}$ cluster variety associated to $\unip$ is defined as the sum of certain global monomials $W_a$ attached to the divisors $D_a$. The function $W_a$ can be defined as follows. Let $\ii=(i_1, \dots, i_N)$ be a reduced word with $i_N=n-a$ and let $x_S$ denote the cluster variable attached to the vertex $v_S$ of $\mathcal{T}_{\ii}$. Then we have 
by Lemma \ref{optimized} and \cite[Corollary 9.17]{GHKK} 
\begin{equation}
 \label{eq:GHKKpot}
\restr{W_a}{\TTix}=x^{-1}_{\{a+1,\ldots, n\}}.
\end{equation}
We call $W$ the GHKK-potential.

The potential $W_a$ is related to the function $\rai$ arising from crystal operations (see \eqref{rpot})
via
the map $ \CAid\in\Hom(\TTis, \TTix)$ defined by
\begin{equation*}
\CAid:  \TTis = \mathbb{G}_m^{\mathcal{T}_{\ii}}  \rightarrow \mathbb{G}_m^{V(\mathcal{T}_{\ii})} = \TTix
\end{equation*}
\begin{equation*}
 \left(\CAid x\right)_v= 
 \prod_{T}  x_T^{\underline{\epsilon}(v,T)}, 
\qquad \underline{\epsilon}(v,T)=\begin{cases}-1&\text{if $v\in\{\ell(T), r(T)\}$, }\\1&\text{if $v\in\{o(T), u(T)\}$}\\ 0&\text{else,} \end{cases}
 \end{equation*}
where $o(T), u(T), \ell(T)$ and $r(T)$ are defined in \eqref{oben}. 
The map $\CAid$ is the dual map of the Chamber Ansatz $\CAi\in\Hom(\TTia, \TTil)$
of Berenstein, Fomin and Zelevinsky \cite{BFZ}.
The family $(\CAid)$ is compatible with $\mathcal{X}$-mutations in the following sense:
 \begin{prop}\label{diagX} For two reduced words $\ii$ and $\jj$ for $w_0$ the following diagram commutes.
$$
 \xymatrix{\TTis \ar[rr]^{\CAid}\ar[dd]_{\trs_{\jj}^{\ii}} & &\TTix \ar[dd]^{\ttix} \\
& \\
\TTisp \ar[rr]^{\CAjd} & & \TTixp}
$$
 \end{prop}
 \begin{proof}By induction we can assume that $\mathcal{T}_{\ii}$ and $\mathcal{T}_{\jj}$ are obtained from each other by braid move corresponding to a flip of a hexagon $\mathcal{H}\subset \mathcal{T}_{\ii}$ as in \eqref{hexbraid}. By direct computation it can be checked in this case that $\ttix$ is given by mutation at the inner vertex of $\mathcal{H}$ and that the diagram commutes.
 \end{proof}
 Using Theorem \ref{rtrans} and the enhanced AM-formula (Theorem \ref{sAM}) we obtain:
  \begin{thm}\label{GHKK=Reineke} Let $\ii$ be a reduced word for $w_0$ and $a\in [n-1]$. Then we have
 \begin{equation}\label{g=r}
 \rai =
 \restr{W_a}{\TTix} \circ \CAid.
 \end{equation}
 Furthermore, the GHKK-potential
 $\restr{W}{\TTix}=\sum_{a \in [n-1]} \rai \circ (\CAid)^{-1}$
 is a Laurent polynomial without constant term, coefficients in $\{0,1\}$ and exponents in $\{0,-1\}$.
 \end{thm}
 \begin{proof}
In order to prove \eqref{g=r} it suffices by Theorem \ref{rtrans} and Proposition \ref{diagX} to show the statement for the case that the tile $[a,a+1]$ intersects the right boundary of $\mathcal{T}=\mathcal{T}_{\ii}$ in two edges (i.e. $\ii$ is commutation equivalent to a reduced word ending with $i_N=n-a$). In this case we have by the definition of $\rai$ and \eqref{eq:GHKKpot} that
 $\rai (x) = x_{[a,a+1]_{\mathcal{T}}}=
  \restr{W_a}{\TTix} \circ \CAid (x).$
By the duality of $\CAid$ and $\CAi$ we obtain
$$
\restr{W_a}{\TTix}(x)=\rai \circ (\CAid)^{-1} (x) = \sum_{y\in \mathbf{R}_a^* \left(\mathcal{T}_{\ii}\right)} x^{\left[\CAi\right]_{\trop} ^{-1} y}.
$$
The map $[\CAi]_{\trop}$ is explicitly given in \eqref{eq:bzdef} (we identify $P(z)$ with the corresponding point $(z_v)_{v\in V(\mathcal{T}_{\ii})}\in\Z^{V(\mathcal{T}_{\ii})}$).
For $y\in \mathbf{R}_a^* \left(\mathcal{T}_{\ii}\right)$ there exists, by definition, $z\in\N^{\mathcal{T}_{\ii}}$ with $y=f_a z- z$, where $f_a$ is the Kashiwara operator associated to the simple root $\alpha_a$. By Theorem \ref{sAM} we conclude for $S\subset [n]$
$$
\left(\left[\CAi\right]_{\trop} ^{-1} y\right)_S =
\left(f_a \left[\CAi\right]_{\trop} ^{-1}  z - \left[\CAi\right]_{\trop} ^{-1} z \right)_S \in  \{0,-1\}.
$$
 \end{proof}

Gross, Hacking, Keel and Kontsevich construct in \cite{GHKK} a canonical basis $\mathbb{B}^{\text{GHKK}}$ of $\C[\unip]$ and parametrizations of $\mathbb{B}^{\text{GHKK}}$ by the cones
  \begin{equation}\label{ghkkcone}
 \mathcal{C}_{\SE}^{\text{GHKK}}:=\left\{z \,\,\middle\mid\,\, \left[\restr{W}{\TTse}\right]_{\trop} (z)\ge 0 \right\}\subset \Hom(\GG, \TTsex).
 \end{equation}
  As a consequence of Theorem \ref{rtrans} and Theorem \ref{GHKK=Reineke}  we obtain:
 \begin{cor}\label{string=xcluster} For any reduced word $\ii$ for $w_0$ we have
$$\mathcal{C}_{\SEi}^{\text{GHKK}}= [\CAid]_{\trop} \left( \stc_{\ii}\right).$$
 \end{cor}
 
 \begin{rem} Let $\mathcal{T}_0$ be the tiling associated to the lexicographically minimal reduced word $\ii_0$. In \cite[Corollary 21]{M} the function $\restr{W}{\check{\mathbb{T}}_{\ii_0}}$ is explicitly computed. This description agrees with the one given by Theorem \ref{GHKK=Reineke} in this case. Furthermore, in \cite{BF} a unimodular transformation between the cone cut out by the GHKK potential function for the partial compactification $SL_n/N$ of the double Bruhat cell $G^{e,w_0}=B^-\cap B w_0 B$ to the weighted string cone was independently discovered by Bossinger and Fourier.
\end{rem}

\subsection{Berenstein-Kazhdan potential functions and Reineke vectors}
In analogy to \eqref{eq:GHKKpot} we define a potential function on the $\mathcal{A}$-cluster variety associated to $\unip$. Note that the frozen $\mathcal{A}$-cluster variables do not change under mutation and are therefore independent of the given $\mathcal{A}$-cluster seed. We define a function by a local condition around the frozen variable in an optimized seed as follows.
\begin{defi} For $a\in [n-1]$ the function $f_{\chi,a}$ on $\unip$ is defined by requiring for reduced words $\ii=(i_1, \dots, i_N)$ of $w_0$ with $i_N=n-a$ that
$$
f_{\chi,a}= \frac{\Delta_{\ell [a,a+1]_{\mathcal{T}_{\ii}}}}{{\Delta_{r [a,a+1]_{\mathcal{T}_{\ii}}}}}.
$$
We further define
\begin{equation*}
f_{\chi}:=\sum_{a\in [n-1]} f_{\chi,a}
\end{equation*}
and call $f_{\chi}$ the \emph{Berenstein-Kazhdan (BK) potential function}. 
\end{defi}
Note that, by Lemma \ref{optimized}, the function $f_{\chi,a}$ is well-defined and we have
$$ f_{\chi,a}=\frac{\Delta_{\{a, a+2, \ldots,n\}}}{\Delta_{\{a+1, \ldots,n\}}}.$$

\begin{rem} The naming comes from the fact that $f_{\chi}$ appears in the \emph{decoration function} $f_{G,\chi}$ introduced by Berenstein and Kazhdan as part of the data of a unipotent crystal structure on $G=\text{SL}_n$. Namely by \cite[Example 1.10]{BK}, we have
$$f_{G,\chi}(g)=\sum_{a\in [n-1]}f_{\chi,a}(g)+\frac{\Delta^{\{1,\ldots,n-a-1,n-a+1\}}_{\{a+1, \ldots,n\}}(g)}{\Delta_{\{a+1,\ldots,n\}}(g)},$$
where $\Delta^{\{1,\ldots,n-a-1,n-a+1\}}_{\{a+1, \ldots,n\}}(g)$ is the determinant of the submatrix of $g$ with rows labeled by $\{1,\ldots,n-a-1,n-a+1\}$ and columns labeled by $\{a+1, \ldots,n\}$ and $\Delta_S=\Delta_S^{[\#S]}$. 

Note that in contrary to \cite{BK} we write $f_{\chi}$ as a Laurent polynomial in the coordinates of the $\mathcal{A}$-cluster torus $\TTia$. 
\end{rem}

The BK potential $f_\chi$ is related to the function $\rai$ arising from the crystal operations (see \eqref{rpot}) via the maps $\ioti\in \Hom(\TTia, \TTis)$ given by
\begin{align}\notag
\ioti:  \TTia=\mathbb{G}_m^{V(\mathcal{T}_{\ii})}& \rightarrow \mathbb{G}_m^{\mathcal{T}_{\ii}}=  \TTis  \\ \label{iotavar}
 x &\mapsto \left(\frac{x_{\ell (T)}}{x_{r(T)}}\right)_{T\in \mathcal{T}_{\ii}},
 \end{align}
where $\ell(T)$ and $r(T)$ are defined in \eqref{oben}.

\begin{rem}
The map $\ioti$ appears in \cite[Equation (4.14)]{BZ2}) as a crucial part of the solution of a factorization problem. In analogy to the naming of Chamber Ansatz we call $\ioti$ \emph{Neighbour Ansatz}. \end{rem}

The family $(\ioti)$ is compatible with $\mathcal{A}$-mutations in the following sense:
 \begin{prop}\label{diagA} For two reduced words $\ii$ and $\jj$ for $w_0$ the following diagram commutes.
$$
 \xymatrix{\TTia \ar[rr]^{\ioti}\ar[dd]_{\ttia} & & \TTis\ar[dd]^{\trs_{\jj}^{\ii}} \\
& \\
\TTiap \ar[rr]^{\iotj} & & \TTisp.}
$$
 \end{prop}
 \begin{proof}By induction we can assume that $\mathcal{T}_{\ii}$ and $\mathcal{T}_{\jj}$ are obtained from each other by braid move corresponding to a flip of a hexagon $\mathcal{H}\subset \mathcal{T}_{\ii}$ as in \eqref{hexbraid}. By direct computation it can be checked in this case that $\ttia$ is given by mutation at the inner vertex of $\mathcal{H}$ and that the diagram commutes.
 \end{proof}
Using the dual Crossing Formula (Theorem \ref{dualcrossing}) and
Theorem \ref{rtrans} we obtain:
  \begin{thm}\label{BK=Reineke} Let $\ii$ be a reduced word for $w_0$ and $a\in [n-1]$. We have
  \begin{equation}\label{bk=r}
\restr{f_{\chi, a}}{\TTia}=\rai \circ \ioti
 \end{equation}
Furthermore, the BK-potential
 $\restr{f_{\chi}}{\TTia}=\sum_{a \in [n-1]} \rai \circ \ioti$
 is a Laurent polynomial with coefficients in $\{0,1\}$ and exponents in $\{-1,0,1\}$ under the change of variables \eqref{iotavar}.
 \end{thm}
 \begin{proof} In order to prove \eqref{bk=r} it suffices by Theorem \ref{rtrans} and Proposition \ref{diagA} to show the statement for the case that the tile $[a,a+1]$ intersects the right boundary of $\mathcal{T}=\mathcal{T}_{\ii}$ in two edges (i.e. $\ii$ is commutation equivalent to a reduced word ending with $i_N=a$). In this case we have by the definition of $\rai$ and Lemma \ref{optimized} 
 $\rai = x_{[a,a+1]_{\mathcal{T}}}=f_{\chi,a}\circ \ioti^{-1} (x)$. By the dual Crossing Formula we obtain that the exponents of $\rai$ are contained in $\{-1,0,1\}$. 
 \end{proof}

 In analogy to \eqref{ghkkcone} we define the polyhedral cones 
   \begin{equation*}
\Csebk:=\left\{z \,\,\middle\mid\,\, \left[\restr{f_{\chi,a}}{\TTse}\right]_{\trop} (z)\ge 0 \right\}\subset \Hom(\GG, \TTse).
 \end{equation*}
 From Theorem \ref{duality} and Theorem \ref{BK=Reineke} we conclude:
 \begin{cor}\label{string=acluster} For any reduced word $\ii$ for $w_0$ we have
 $$\stc_{\ii}={[\ioti]}_{\trop}\left(\mathcal{C}_{\SEi}^{BK}\right).$$
 \end{cor}

From Theorem \ref{BK=Reineke} and Theorem \ref{GHKK=Reineke} we obtain the following relation between the potential functions $f_{\chi}$ arising from the theory geometric crystals and the potential function $W$ arising from the partial compactification of $L^{e,w_0}$:
\begin{cor}\label{relpot} Let $\ii$ be a reduced word for $w_0$. Then
$$\restr{f_{\chi,a }}{\TTia}=\restr{W_a}{\TTix}\circ \CAid \circ \ioti.$$
\end{cor}
\begin{rem}
In \cite{GHKK} a result about the relation of the BK-potential function and $W$ is announced to appear in \cite{M2}.
\end{rem}

\begin{rem}
The relation given in Corollary \ref{relpot} does not arise from the coordinate change $\pp$ between $\mathcal{A}-$ and $\mathcal{X}-$cluster variables given in \cite[Equation (6)]{FG}, i.e. we have $\restr{f_{\chi,a }}{\TTia} \neq \restr{W_a}{\TTix}\circ \ppi$. 
 The reason for this is that the equality $p_{\ii}=\CAid \circ \ioti$ only holds restricted to vertices which do not correspond to frozen variables. Furthermore, the map $\ppi$ is not canonically defined for frozen variables.

In \cite{Ge} the first author introduces a canonical modification $\overline{\pp}_{\ii}$ of $\pp_{\ii}$ and  shows $\overline{\pp}_{\ii}=\CAid^{-1} \circ \ioti$.  The canonical $\pp$-map $\overline{\pp}_{\ii}$ furthermore relates the two main ingredients  $\rvec{}_{\ii} \,:\, \mathcal{R}_a^*(\mathcal{T}_{\ii}) \rightarrow \Z^{\mathcal{T}_{\ii}}$ and $\drvec{}_{\ii} \,:\, \mathcal{R}_a^*(\mathcal{T}_{\ii}) \rightarrow \Hom(\Z^{\mathcal{T}_{\ii}},\Z)$ of the dual Crossing Formula 
as follows:
In \cite{Ge} the first author shows $ \drvec{}_{\ii} = \tro{\ioti \circ {\CA_{\ii}^{-1}} } \circ   \rvec{}_{\ii} $. Consequently, after change of coordinates by $\CAi$ the linear form $\drvec{}_{\ii}  \gamma$ is obtained from the Reineke vector $\rvec{}_{\ii} \gamma$ by applying the tropicalized canonical $\pp$-map $[\overline{\pp}]_{\trop}$.
\end{rem}

\def\cprime{$'$} \def\cprime{$'$} \def\cprime{$'$} \def\cprime{$'$}

\end{document}